\documentclass{dis}

\usepackage{amsmath, amssymb, amsthm, amsfonts, enumerate, color, comment}
\usepackage{mathrsfs}
\usepackage{parskip}
\usepackage{xcolor}

\newcommand{\phm}{\phantom{-}}
\usepackage{mathtools}
\mathtoolsset{showonlyrefs}

\usepackage{palatino}
\usepackage{graphicx}

\usepackage{parskip}

\theoremstyle{plain}
\newtheorem{Proposition}[equation]{Proposition}
\newtheorem{Corollary}[equation]{Corollary}
\newtheorem*{Corollary*}{Corollary}
\newtheorem{Theorem}[equation]{Theorem}
\newtheorem*{Theorem*}{Theorem}
\newtheorem{Lemma}[equation]{Lemma}
\theoremstyle{definition}
\newtheorem{Definition}[equation]{Definition}

\newtheorem{Example}[equation]{Example}
\newtheorem{Remark}[equation]{Remark}

\usepackage{enumitem}
\setlist[enumerate]{leftmargin=*}
\setlist[itemize]{leftmargin=*}

\setlist[enumerate,1]{label=(\alph*),font=\upshape}

\setlist[enumerate,2]{label=(\roman*),font=\upshape}

\def\C{\mathbb{C}}
\def\R{\mathbb{R}}
\def\D{\mathbb{D}}
\def\T{\mathbb{T}}
\def\N{\mathbb{N}}
\def\Z{\mathbb{Z}}

\newcommand{\h}{\mathcal{H}}

\newcommand{\M}{\emph{\textbf{M}}}

\newcommand{\U}{\mathbf{U}}

\renewcommand{\leq}{\leqslant}
\renewcommand{\geq}{\geqslant}
\renewcommand{\subset}{\subseteq}
\renewcommand{\phi}{\varphi}
\renewcommand{\vec}[1]{{\bf #1}}

\usepackage{xcolor}

\chaptersnewpage
\refnewpage
\begin{document}

\keywords{Conjugations, unitary operators }
\mathclass{ 47B35, 47B02, 47A05}

\thanks{Work supported by the NSERC Discovery Grant (Canada), the Canada Research Chairs program,  and by the Ministry of Science and
Higher Education of the Republic of Poland.}

\abbrevauthors{J. Mashreghi and M. Ptak and W. T. Ross}
\abbrevtitle{Conjugate orbit}

\title{The Conjugate Orbit of a Unitary Operator}

\author{Javad Mashreghi}
\address{D\'epartement de Math\'ematiques et de Statistique\\ 
Universit\'e Laval\\ Qu\'ebec, QC
Canada, G1K 0A6\\
E-mail: javad.mashreghi@mat.ulaval.ca}

\author{Marek Ptak}
\address{Department of Applied Mathematics\\
University of Agriculture, ul. Balicka 253c\\ 30-198 Krak\'ow, Poland\\
Email: rmptak@cyf-kr.edu.pl}

\author{William T. Ross}
	\address{Department of Mathematics and Statistics\\
	 University of Richmond\\Richmond, VA 23173, USA\\
	Email: wross@richmond.edu}

\maketitledis

\tableofcontents

\begin{abstract}
This paper discusses various aspects of  the collection of unitary operators $CUC$, where $U$ is a fixed unitary operator  on a complex Hilbert space $\mathcal{H}$ and $C$ varies over the set of all conjugations on $\mathcal{H}$ (antilinear, isometric, involutions). We call this class of unitary operators, the {\em conjugate orbit }of $U$ and denote it  by $\mathfrak{O}_c(U)$. We will see that $U^{*}$, the Hilbert space adjoint of $U$, always belongs to $\mathfrak{O}_c(U)$, while $U$ belongs to $\mathfrak{O}_c(U)$ only when $U$ is unitarily equivalent to $U^{*}$, making $U$ a member of $\mathfrak{O}_c(U)$ an uncommon event.  We completely describe the conjugate orbit of the classical bilateral shift and discuss when a unitary multiplication operator on the classical Lebesgue space of the unit circle belongs to this conjugate orbit. We also broaden this discussion to include the bilateral shifts of higher multiplicity which, via unitary equivalence, makes connections to other interesting unitary operators such as the translation and dilation operators on the Lebesgue space of the real line. Finite unitary matrices provide us with a rich source of examples of conjugate unitary orbits to discuss. In particular, we determine which diagonal matrices, if any, belong to the conjugate orbit of a fixed unitary matrix. 
Closely related to the finite unitary matrices are the diagonalizable unitary operators with respect to some, possibly infinite, orthonormal basis. We give a large class of variations of these unitary operators that belong to the conjugate orbit and establish a connection to the classical Fourier--Plancherel and Hilbert transforms. 
Finally, we develop a model for a unitary operator using real Hilbert spaces and use it to describe the conjugate orbit as well as revisit some of our previous discussions in another light. 
\end{abstract}

\makeabstract

\chapter{Introduction}

This paper initiates a study of certain classes of linear transformations on a complex Hilbert space and their connections to a well-studied notion of a conjugation in a vector space that shares the same properties as complex conjugation in $\C$. The subject of our investigation is inspired by the following idea. 
For a fixed unitary operator  $U$ on a complex Hilbert space $\mathcal{H}$  ($U$ is linear, surjective, and isometric -- equivalently $U^{*} U = U U^{*} = I$), one can define the {\em unitary orbit} $\mathfrak{O}(U)$ of $U$
 to be the collection of operators $V U V^{*}$, where $V$ ranges over all the unitary operators on  $\mathcal{H}$. In other words, $\mathfrak{O}(U)$ is the set of all unitary operators on $\mathcal{H}$ that are unitarily equivalent to $U$. By the spectral theorem, $\mathfrak{O}(U)$ consists of all the unitary operators  on $\mathcal{H}$ with the same spectral parameters as $U$ \cite[Ch.~IX]{ConwayFA}. For example, when $U$ is an $n \times n$ unitary matrix, regarded as a linear transformation on $\C^n$ (viewed as column vectors) by $\vec{x} \mapsto U \vec{x}$, the unitary orbit $\mathfrak{O}(U)$ is the collection of all $n \times n$ unitary matrices which have the same eigenvalues as $U$ along with the same corresponding (algebraic/geometric) multiplicities. There is a generalization of this spectral and multiplicity theory for unitary operators on infinite dimensional Hilbert spaces (see \cite{ConwayFA} or \cite{MR0045309}). So, in a way, the unitary orbit $\mathfrak{O}(U)$ is well understood. 

In this paper we examine the {\em conjugate orbit}
$$\mathfrak{O}_c(U) := \{C U C: \mbox{$C$ is a conjugation on $\h$}\}$$
of a unitary operator $U$ on $\mathcal{H}$.
Here, a {\em conjugation} $C$ on a complex Hilbert space $\mathcal{H}$ is an antilinear,  involutive, isometry in that
$C(\alpha \vec{x} + \vec{y}) = \overline{\alpha} C \vec{x} + C \vec{y}$ for all $\vec{x}, \vec{y} \in \mathcal{H}$ and $\alpha \in \C$ (antilinear), $C^2 := C \circ C = I$ (involutive), and $\|C \vec{x}\| = \|\vec{x}\|$ for all $\vec{x} \in \mathcal{H}$ (isometric). We refer the reader to the papers \cite{MR2187654, MR2302518, MR2749452, MR4747045, MR4745001} for a thorough discussion of conjugations. Simple examples of conjugations include $C \vec{x} = \overline{\vec{x}}$ on $\C^{n}$ (complex conjugate each entry of $\vec{x}$); $C \vec{x} = V \overline{\vec{x}}$ on $\C^{n}$, where $V$ is an $n \times n$  self-transpose unitary matrix; $C f = \bar f$ on $L^2(m)$ ($m$ is normalized Lebesgue measure on the unit circle $\T$); $C f = u \bar f$ on $L^2(m)$ ($u$ is a measurable unimodular function on $\T$);  $(C f)(x) = \bar f(-x)$ on $L^2(\R)$; and $C f = \bar f \circ \alpha$ on $L^2(\R)$, where $\alpha$ is (Lebesgue) measure preserving on $\R$ with $\alpha \circ \alpha = \operatorname{id}$. We will encounter a rich variety of other conjugations throughout this paper.

Since $C U C$ is {\em linear} (the conjugation $C$, which is antilinear,  is applied twice), isometric, and onto, then
 $\mathfrak{O}_c(U)$ is certainly contained inside the class of all unitary operators on $\mathcal{H}$. The purpose of this paper is to initiate a study of 
  $\mathfrak{O}_c(U)$ and, in certain tangible cases, give its complete description. Moreover, we discuss how the two orbits $\mathfrak{O}(U)$ (the unitary orbit of $U$) and $\mathfrak{O}_c(U)$ (the conjugate orbit of $U$) differ. For example, for any $n \times n$ unitary matrix $U$, $\mathfrak{O}(U)$ contains every $n \times n$ diagonal matrix consisting of the eigenvalues of $U$ (with respective multiplicities). However, there are examples of unitary matrices $U$ for which $\mathfrak{O}_c(U)$ contains some diagonal matrices consisting of the eigenvalues of $U^{*}$, but not others. Still further, there are examples of unitary matrices $U$ where $\mathfrak{O}_c(U)$ contains no diagonal matrices. For a general unitary operator $U$ we will show that  although $U \in \mathfrak{O}(U)$, it is only under quite special circumstances that $U \in \mathfrak{O}_c(U)$. However, it is always the case that  $\mathfrak{O}_c(U) \subset \mathfrak{O}(U^{*})$. 

We also discuss the relative ``size'' of $\mathfrak{O}_c(U)$ and describe when the unitary operators  $C U C$ and $C' U C'$, where $C$ and $C'$ are conjugations, represent the same element of $\mathfrak{O}_c(U)$. In addition, we give a tangible characterization of the conjugate orbit of the classical bilateral shift $(M f)(\xi) = \xi f(\xi)$ on $L^2(m)$, as certain unitary weighted shifts, and discuss when a unitary multiplication operator $M_{\phi} f = \phi f$, $\phi \in L^{\infty}(m)$ and unimodular, belongs to the conjugate orbit of $M$.  Finally, we develop a model for unitary operators involving real Hilbert spaces, along with the notion of complexification, and use it to describe the conjugate orbit of a general unitary operator. 

This paper is structured as follows. In Chapter \ref{S-One} we give some basic facts about the conjugate orbit $\mathfrak{O}_c(U)$ and remark that $U^{*} $, the Hilbert space adjoint of $U$,  always belongs to $\mathfrak{O}_c(U)$ while $U$ belongs to $\mathfrak{O}_c(U)$ if and only if $U$ is unitarily equivalent to $U^{*}$. We also show that $\mathfrak{O}_c(U) \subset \mathfrak{O}(U^{*})$. When a unitary operator $U$ is diagonalizable with respect to an orthonormal basis $(\vec{u}_j)_{j \geq 1}$, in other words,  $U = \sum_{j \geq 1} \xi_j (\vec{u}_j \otimes \vec{u}_j)$, where $|\xi_j| = 1$ for all $j \geq 1$, then certain  unitary operators  of the form $\sum_{j \geq 1} \bar \xi_{\sigma(j)}  (\vec{u}_j \otimes \vec{u}_j)$, where $\sigma$ belongs to a particular class of permutations of $\N$, belong to $\mathfrak{O}_c(U)$ while others of this type  do not. It is usually the case that the conjugate orbit  of a diagonalizable operator contains many nondiagonalizable unitary operators. Examples discussed here include the classical Fourier--Plancherel  and Hilbert transforms.
 Chapter \ref{S-One} also explores, through various specific examples, when the unitary operators $C_1 U C_1$ and $C_2 U C_2$, for conjugations $C_1$ and $C_2$, represent the same member of the conjugate orbit of $U$. Interesting examples explored here include the classical bilateral shift where the discussion involves certain doubly infinite Toeplitz (Laurent) matrices.

 The special case of $n \times n$  unitary matrices $U$ is taken up in Chapter \ref{S-Four}, where we develop a description of $\mathfrak{O}_c(U)$ as the collection of matrices $V \overline{U} \,\overline{V}$, where $V$ is an $n \times n$  self transpose, i.e., $V^{t} = V$, unitary matrix ($V$ could be a Householder matrix for example) and $\overline{U}$ denotes the matrix $U$ with all of its entries complex conjugated. In addition, we also state a criterion as to which classes of diagonal matrices (if any) belong to $\mathfrak{O}_c(U)$ and point out further differences between the conjugate orbit $\mathfrak{O}_c(U)$ and the unitary orbit $\mathfrak{O}(U)$.

  The bilateral shift $M$, defined on the standard orthonormal basis $(\vec{e}_j)_{j \in \Z}$ for the classical sequence space $\ell^2(\Z)$ by $M \vec{e}_{j} = \vec{e}_{j + 1}$,  equivalently, via Fourier series, $(M f)(\xi) = \xi f(\xi)$ on $L^2(m)$, is taken up in Chapter \ref{S-Five}. Here we show that $\mathfrak{O}_c(M)$ consists precisely of the unitary shifts of the form $\vec{v}_{j} \mapsto \vec{v}_{j + 1}$, where $V = [\cdots| \vec{v}_{-2}|\vec{v}_{-1}|\vec{v}_{0}| \vec{v}_{1}|\vec{v}_{2}|\cdots]$ is a doubly infinite unitary matrix with $V^{t} = V$. Results in this section make the case that very few doubly infinite Toeplitz matrices generated by unimodular symbols (or, equivalently,  unitary multiplication operators on $L^2(m)$) belong to the conjugate orbit of the bilateral shift $M$. Using  a certain form of the spectral theorem for unitary operators, we show that when $\sigma(U)$, the spectrum of a general unitary operator $U$,  has  more then one point, then $\mathfrak{O}_c(U)$ is quite an abundant set. An extension of our discussion of the classical bilateral shift to  bilateral shifts of higher multiplicity continues in Chapter \ref{S-Five}, where we show the complexity and the great differences between the conjugate orbit of the bilateral shift of multiplicity one (where we have a description as certain unitary shifts) and the conjugate orbit of the bilateral shifts of higher multiplicity. We also give several examples of naturally occurring unitary operators, such as the translation and the dilation operators on  $L^2(\R)$, as well as multiplication by an inner function on $L^2(m)$, that are represented by shifts of higher multiplicity. Thus, though these operators are easy to describe, they have quite complicated conjugate orbits. 

In the last several chapters, we develop a model for a unitary operator using real Hilbert spaces along with the concept of complexification and use this model to give a description of the conjugate orbit of a general unitary operator. We will revisit some of the results from the previous sections involving unitary matrices, the Fourier--Plancherel and Hilbert transforms, as well as  the bilateral shift through this model. 

This paper is meant to give a solid foundation as well as an invitation to this topic of conjugate orbits and we invite the reader to continue our initial investigation and describe the conjugate orbit for their favorite unitary operator. 

\chapter{Some Basic Facts About the Conjugate Orbit}\label{S-One}

 Throughout  this paper, $\mathcal{H}$ will denote a complex separable Hilbert space with inner product $\langle \vec{x}, \vec{y}\rangle$ and corresponding norm $\|\vec{x}\| := \sqrt{\langle \vec{x}, \vec{x}\rangle}$. A linear transformation $A$ on $\mathcal{H}$ will be {\em bounded} if its norm $\|A\| := \sup\{\|A \vec{x}\|: \|\vec{x}\| = 1\}$
is finite.
These bounded linear transformations on $\mathcal{H}$ will be called linear {\em operators}.
The focus of this paper will be the {\em unitary operators} $U$ on $\mathcal{H}$. These are the linear transformations $U$ which are isometric ($\|U \vec{x}\| = \|\vec{x}\|$ for all $\vec{x} \in \mathcal{H}$) and surjective ($U \mathcal{H} = \mathcal{H}$). With each bounded linear operator $A$ on $\mathcal{H}$, there is a unique bounded linear operator $A^{*}$, called the {\em adjoint} of $A$, such that 
$\langle A \vec{x}, \vec{y}\rangle = \langle \vec{x}, A^{*} \vec{y}\rangle$ for all $\vec{x}, \vec{y} \in \mathcal{H}$. The criterion that $U$ is a unitary operator on $\mathcal{H}$ can be restated  algebraically as $U^{*} U = U U^{*} = I$ (the identity operator on $\mathcal{H}$).

\section{Some inhabitants of the conjugate orbit}

As $\mathfrak{O}_c(U)$ is contained in the set of all unitary operators on $\mathcal{H}$, it would be helpful to have some readily identifiable elements of $\mathfrak{O}_c(U)$. We begin with the following  fact which follows  from an observation in \cite[Sec.~4]{MR2187654}. We will arrive at this same fact in Remark \ref{2..99999} and again in Remark \ref{sdfsdfRemaRKUU***}.

\begin{Proposition}\label{SdfsdfdsfcvvvvVV}
For any unitary operator $U$ on $\mathcal{H}$ we have $U^{*} \in \mathfrak{O}_c(U)$.
\end{Proposition}

\begin{proof}
A version of the spectral theorem for unitary operators says that $U$ is unitarily equivalent to a multiplication operator $M_{\phi} f = \phi f$ on $L^2(\mu, X)$ ($X$ a compact Hausdorff space, $\mu$ a positive finite Borel measure on $X$, and $\phi \in L^{\infty}(\mu, X)$ and unimodular) in that there is an isometric isomorphism $W: \mathcal{H} \to L^2(\mu, X)$  with $M_{\phi} W = W U$ \cite[p.~279]{ConwayFA}. If $J$ is the conjugation on $L^2(\mu, X)$ defined by $J f = \bar{f}$, a calculation verifies that $J M_{\phi} J = M_{\bar{\phi}}  = M_{\phi}^{*}$. Furthermore, one can also verify that $C := W^{*} J W$ is a conjugation on $\mathcal{H}$ which satisfies $C U C = U^{*}$. Thus, $U^{*} \in \mathfrak{O}_c(U)$.
\end{proof}

\begin{Remark}\label{nNNnniiinnnndddD}
For a fixed unitary operator $U$ on $\mathcal{H}$, the paper \cite{MR4747045} characterizes {\em all} of the conjugations $C$ for which $CUC = U^{*}$. For example, when 
$$(U f)(\xi) = \xi f(\xi)$$ is the {\em bilateral shift} on $L^2(m)$ (recall that $m$ is normalized Lebesgue measure on the unit circle $\T$),  a result from  \cite[Theorem 4.1]{MR4040358}  and \cite[Theorem 2.2]{MR4083641}, which will be important later in Example \ref{examplecalculus}, says that any conjugation $C$ on $L^2(m)$ for which $C U C = U^{*}$ must take the form $C f = u \bar{f}$, where $u \in L^{\infty}(m)$ and unimodular. We will discuss the bilateral shift  below and again in greater detail  in Chapter \ref{S-Five}.
\end{Remark}

When is $U \in \mathfrak{O}_c(U)$? It seems like it should be -- just by the use of the term ``orbit''. However,  for this to be true, there must be a conjugation $C$ on $\mathcal{H}$ such that  $C U C = U$.  From \cite[Thm.~1.2]{MR4745001} this only happens under very special circumstances.

\begin{Proposition}\label{e0roiekfgdFF}
For a unitary operator  $U$ on $\mathcal{H}$ we have $U \in \mathfrak{O}_c(U)$ if and only if $U$ is unitarily equivalent to $U^{*}$.
\end{Proposition}

One direction of this is easy to prove. Indeed, suppose that $C U C = U$. Then, letting $C_1$ be a conjugation for which $C_1 U C_1 = U^{*}$ (Proposition \ref{SdfsdfdsfcvvvvVV}), the mapping $V = C_1 C$ is linear, isometric, and onto -- and thus is a unitary operator. Finally, 
\begin{align*}
U^{*} V & = (C_1 U C_1) (C_1 C)\\
& = C_1 (U C)\\
& = C_1 (C U)\\
& = V U
\end{align*}
and so $U$ and unitarily equivalent to $U^{*}$. 
The other direction is more involved and we refer the reader to \cite[Thm.~1.2]{MR4745001}  for the details. 

For example, by the spectral theorem from linear algebra, the $n \times n$ unitary matrices whose eigenvalues are either one or minus one (or both) are examples of unitary operators that are unitarily equivalent to their adjoints (as are unitary matrices whose eigenvalues occur in conjugate pairs with equal corresponding multiplicities). Another example is the bilateral shift 
$(U f)(\xi) = \xi f(\xi)$ on $L^2(m)$. Indeed, it is easy to check that the unitary operator $(W f)(\xi) = f(\overline{\xi})$ on $L^2(m)$ intertwines $U$ and its adjoint  $(U^{*} f)(\xi) = \overline{\xi} f(\xi)$. Other well-known unitary operators having this property (unitarily equivalent to their adjoint) are the Hilbert and Fourier--Plancherel transforms on $L^2(\R)$. We will discuss both of these operators several times in this paper (see Example \ref{FTFTFT} and Example \ref{ldfgjdfjg888uuUU})  and confirm the above proposition for these operators in Example \ref{examdepleFFTTT} and Example \ref{HTwlkdfj77YY}. For a given unitary operator $U$ on $\mathcal{H}$, the paper \cite{MR4745001} characterizes the  conjugations $C$ on $\mathcal{H}$  for which $C U C = U$ (when such conjugations exist). For example, for the bilateral shift $U$ (as above) on $L^2(m)$, any conjugation $C$ on $L^2(m)$ for which $CUC = U$ is of the form $(C f(\xi) = u(\xi) \overline{f(\bar \xi)}$, where $u \in L^{\infty}(m)$, is unimodular, and satisfies $u(\xi) = u(\bar \xi)$ almost everywhere.

\begin{Remark}\label{singlepoint}
When $U = \lambda I$, where $|\lambda| = 1$ and $I$ is the identity operator on $\mathcal{H}$,  we can use the facts that a conjugation is antilinear and involutive to see that $\mathfrak{O}_c(U) = \{\overline{\lambda} I\}$. So, in this case, the conjugate orbit consists of just one operator. In fact, if $U$ is {\em any} unitary operator whose spectrum $\sigma(U)$ is a single point $\{\lambda\}$, then $\lambda$ is clearly an isolated point of the spectrum and thus must be an eigenvalue. By the spectral theorem, it must be the case that $U = \lambda I$. 
However, we shall see throughout this paper that when its spectrum $\sigma(U)$ consists of more than one point, $\mathfrak{O}_c(U)$ is usually quite a large and rich class of unitary operators  (see Remark \ref{morethanone}). That being said, this next result shows  that, up to unitary equivalence, the conjugate orbit can be regarded as  ``small''.
\end{Remark}

\begin{Proposition}\label{uuueeeqq}
If $U$ is a unitary operator on $\mathcal{H}$ and  $V_1, V_2 \in \mathfrak{O}_c(U)$, then $V_1$ and $V_2$ are unitarily equivalent.
\end{Proposition}

\begin{proof}
Suppose $V_1 = C_1 U C_1$ and $V_2 = C_2 U C_2$ for conjugations $C_1$ and $C_2$ on $\mathcal{H}$. Then, since $C_{1}^{2} = C_{2}^{2} = I$, we have
$C_1 V_1 C_1 = C_2 V_2 C_2$ and so $(C_2 C_1) V_1 = V_2 (C_2 C_1)$. Since $C_2 C_1$ is unitary (linear, isometric, and onto), we see that $V_1$ is unitarily equivalent to $V_2$.
\end{proof}

Combine Proposition \ref{SdfsdfdsfcvvvvVV} with Proposition \ref{uuueeeqq} to obtain the following. 

\begin{Corollary}\label{Sdfsdfdsf11112}
For any unitary  operator $U$ on $\mathcal{H}$, we have $\mathfrak{O}_c(U) \subset \mathfrak{O}(U^{*})$. 
\end{Corollary}

 We will  see in Example \ref{diagoanlalalal} and Example \ref{qudiagonall9} that the two orbits $\mathfrak{O}_c(U)$ and $\mathfrak{O}(U)$ are rarely equal. 

\section{An initial characterization}

This next result  is our initial description of the conjugate orbit of a unitary operator. Though it is a simple observation, it will be  important to later results and  will be refined and expanded throughout this paper. Recall from  Proposition \ref{SdfsdfdsfcvvvvVV} that for any unitary operator  $U$, there is always a conjugation $C$, in fact many of them  \cite{MR4747045} (see also Remark \ref{nNNnniiinnnndddD}), such that $C U C = U^{*}$.

\begin{Proposition}\label{sdfhjsd9f9ds9dsf9sdf9s9df}
Suppose that $U$ is a unitary operator on $\mathcal{H}$ and $C$ is a conjugation on $\mathcal{H}$ such that $C U C = U^{*}$. Then, for any unitary operator $V$ on $\mathcal{H}$, the following are equivalent.
\begin{enumerate}
\item $V \in \mathfrak{O}_c(U)$.
\item There is a unitary operator  $W$ on $\mathcal{H}$ such that $C W C = W^{*}$ and $V = W U^{*} W^{*}$.
\end{enumerate}
\end{Proposition}

Before getting to the proof, there is a small detail to discuss concerning adjoints. Suppose that $C_1$ and $C_2$ are conjugations on $\mathcal{H}$. Then $C_1 C_2$ is a unitary operator since it is linear, isometric, and onto. Moreover, for any $\vec{x}, \vec{y} \in \mathcal{H}$, we can use  properties of conjugations to obtain 
$$
\langle C_1 C_2 \vec{x}, \vec{y}\rangle  = \langle C_1 \vec{y}, C_2 \vec{x}\rangle = \langle \vec{x}, C_2 C_1 \vec{y}\rangle
$$
and thus
\begin{equation}\label{sdfkjgfkgdsfvkjkjknknk}
(C_1 C_2)^{*} = C_2 C_1.
\end{equation}

\begin{proof}[Proof of Proposition \ref{sdfhjsd9f9ds9dsf9sdf9s9df}]
$(a) \Longrightarrow (b)$: If $V \in \mathfrak{O}_c(U)$,  there is a conjugation $C_1$ on $\mathcal{H}$ such that $V = C_1 U C_1$.
Then
\begin{align*}
V & = C_1 U C_1\\
& = C_1 (C U^{*} C) C_{1}\\
& = (C_1 C) U^{*} (C_1 C)^{*}.
\end{align*}
Note the use of \eqref{sdfkjgfkgdsfvkjkjknknk} in the last line of the above calculation.  The operator $W : = C_1 C$ is unitary and, by the previous calculation,  satisfies $W U^{*} W^{*} = V$. Moreover, 
\begin{align*}
C W C & = C (C_1 C) C\\
& = C C_1\\
& = (C_1 C)^{*}\\
& = W^{*}.
\end{align*}

$(b) \Longrightarrow (a)$: Assume that $W$ is a unitary operator that satisfies  
$$V = W U^{*} W^{*} \; \; \mbox{and} \; \; C W C = W^{*}.$$ Then $C_1 := W C$ is a conjugation. Indeed, $C_1$ is isometric and antilinear and 
$$C_{1}^{2} = (WC)(WC) = W W^{*} = I.$$ Moreover, 
\begin{align*}
C_1 U C_1 & = (W C) U (W C)\\
& = (WC) U C C (WC)\\
& = W(C U C)( C WC)\\
& = W (C U C) W^{*}\\
& = W U^{*} W^{*}\\
& = V,
\end{align*}
which completes the proof.
\end{proof}

\begin{Remark}\label{2..99999}
The adjoint identity from \eqref{sdfkjgfkgdsfvkjkjknknk} gives us an alternate way to prove Proposition \ref{SdfsdfdsfcvvvvVV} ($U^{*}$ always belongs to $\mathfrak{O}_c(U)$). Indeed, a result from \cite{MR0190750}  says that any unitary operator $U$ can be written as $U = C_1 C_2$ for some conjugations $C_1$ and $C_2$. From here use \eqref{sdfkjgfkgdsfvkjkjknknk} to  see that 
\begin{align*}
C_1 U C_1 & = C_1 (C_1 C_2) C_1\\
& = C_2 C_1\\
& = (C_1 C_2)^{*}\\
& = U^{*}.
\end{align*}
Thus, $U^{*} \in \mathfrak{O}_c(U)$. 
\end{Remark}

\section{Matrix representations and diagonalization}





Assuming that our ambient Hilbert space $\mathcal{H}$ is separable, it has an orthonormal basis $\mathcal{E} = (\vec{u}_{n})_{n \geq 1}$. What does the matrix representation of the unitary operator $U$ with respect to $\mathcal{E}$ tell us about the corresponding matrix representation of the members of $\mathfrak{O}_c(U)$? For any bounded operator $A$ on $\mathcal{H}$, let
$$[A]_{\mathcal{E}} := [\langle A \vec{u}_j, \vec{u}_i\rangle]_{i, j \geq 1}$$ denote the matrix representation of $A$ with respect to $\mathcal{E}$. Also, for a finite (or  infinite) matrix $M $ of complex numbers, let $\overline{M}$ denote the matrix $M$ with each entry conjugated.

\begin{Proposition}\label{ksdlkfjsldf88}
A unitary operator $V$ on $\mathcal{H}$ belongs to $\mathfrak{O}_c(U)$ if and only if  there is an orthonormal basis $\mathcal{E}$ for $\mathcal{H}$ such that $ [V]_{\mathcal{E}} = \overline{[U]_{\mathcal{E}}}$.
\end{Proposition}

\begin{proof}
Suppose that $V = C U C$ for some conjugation $C$. A  general fact about conjugations from \cite[Lemma 1]{MR2187654}  says  there exists a {\em $C$-real} orthonormal basis  $\mathcal{E} = (\vec{u}_j)_{j \geq 1}$ for $\mathcal{H}$ in that $C \vec{u}_j = \vec{u}_j$ for all $j \geq 1$. With this choice of $C$-real orthonormal basis $\mathcal{E}$, we have
\begin{align*}
([V]_{\mathcal{E}})_{k, j} & =  \langle V \vec{u}_j, \vec{u}_k\rangle\\
&= \langle C U C \vec{u}_j, \vec{u}_k\rangle\\
& = \langle C U \vec{u}_j, \vec{u}_k\rangle\\
& = \langle C \vec{u}_k, U \vec{u}_j\rangle\\
& = \langle \vec{u}_k, U \vec{u}_j\rangle\\
& = \overline{\langle U \vec{u}_j, \vec{u}_k\rangle}\\
& = \overline{([U]_{\mathcal{E}})_{k, j}}.
\end{align*}

Conversely, suppose that $[V]_{\mathcal{E}} = \overline{[U]_{\mathcal{E}}}$ for some orthonormal basis $\mathcal{E} = (\vec{u}_j)_{j \geq 1}$, in other words, $\langle V \vec{u}_j, \vec{u}_k\rangle = \langle \vec{u}_k, U \vec{u}_j\rangle$ for all $j, k \geq 1$. Define a conjugation $C$ on $\mathcal{H}$ by first defining it on the orthonormal basis elements by $C \vec{u}_j = \vec{u}_j$ and then extending it antilinearly to all of $\mathcal{H}$. By this we mean 
\begin{equation}\label{Creeeeeeel}
C\Big(\sum_{j \geq 1} a_j \vec{u}_j \Big) := \sum_{j \geq 1} \bar a_j \vec{u}_j.
\end{equation}
We are essentially defining a conjugation $C$ so that $(\vec{u}_j)_{j \geq 1}$ is a  $C$-real basis. Then, reversing the steps of  the calculation in the previous part of the argument shows that $\langle C U C \vec{u}_{j}, \vec{u}_k\rangle = \langle V  \vec{u}_{j}, \vec{u}_k\rangle$ for all $j, k \geq 1$ and thus  $V = C UC$.
\end{proof}


When a unitary operator $U$ is representable by a diagonal matrix, in other words, $U$ has an orthonormal basis of eigenvectors, we can obtain a specific and rich class of  members of its conjugate orbit. Below we use the standard tensor notation $\vec{a} \otimes \vec{b}$, where $\vec{a}, \vec{b}$ are vectors from a Hilbert space $\mathcal{H}$, to denote the rank one operator on $\mathcal{H}$ defined by 
$$(\vec{a} \otimes \vec{b}) \vec{x} = \langle \vec{x}, \vec{b}\rangle \vec{a}, \quad \vec{x} \in \mathcal{H}.$$
Observe that $(\vec{a} \otimes \vec{b})^{*} = \vec{b} \otimes \vec{a}$ and if $\vec{a}$ and $\vec{b}$ are unit vectors, then $(\vec{a} \otimes \vec{b})^{*} ( \vec{a} \otimes \vec{b})$ is the orthogonal projection onto $\C \vec{b}$ while 
$(\vec{a} \otimes \vec{b})(\vec{a} \otimes \vec{b})^{*}$ is the orthogonal projection onto $\C \vec{a}$. This says that if $(\vec{u}_j)_{j \geq 1}$ is an orthonormal basis for $\mathcal{H}$ and $(\xi_j)_{j \geq 1}$ are unimodular constants, then 
$$\sum_{j \geq 1} \xi_j (\vec{u}_{j} \otimes \vec{u}_j)$$ defines a unitary operator on $\mathcal{H}$. 

\begin{Corollary}\label{C27}
Suppose that $(\vec{u}_{j})_{j \geq 1}$ is an orthonormal basis for $\mathcal{H}$, $(\xi_j)_{j \geq 1}$ is a sequence of unimodular constants, and $U$ is the unitary operator on $\mathcal{H}$ defined by
$$U  := \sum_{j \geq 1} \xi_{j}  (\vec{u}_j \otimes  \vec{u}_{j}).$$
For any bijection $\sigma: \N \to \N$ with $\sigma \circ \sigma = \operatorname{id}$, the unitary operator $U_{\sigma}$ on $\mathcal{H}$ defined by
$$U_{\sigma} :=  \sum_{j \geq 1} \bar \xi_{\sigma(j)} (\vec{u}_j \otimes \vec{u}_{j})$$ belongs to $\mathfrak{O}_c(U)$.
\end{Corollary}

\begin{proof}
Define  
\begin{equation}\label{66tt55RFDDS}
C_{\sigma} \vec{u}_j = \vec{u}_{\sigma(j)} \; \mbox{for all $j \geq 1$}
\end{equation}
 and extend $C_{\sigma}$ as in \eqref{Creeeeeeel}  to be an antilinear isometry on $\mathcal{H}$. Since $\sigma \circ \sigma = \operatorname{id}$, we see that $C_{\sigma}^{2} = I$ and thus $C_{\sigma}$ is a conjugation. Now observe that for each $j \geq 1$ we have
\begin{align*}
C_{\sigma} U C_{\sigma} \vec{u}_j & = C_{\sigma} U \vec{u}_{\sigma(j)}\\
& = C_{\sigma}( \xi_{\sigma(j)} \vec{u}_{\sigma(j)})\\
& = \bar \xi_{\sigma(j)} \vec{u}_{j}\\
& = U_{\sigma} \vec{u}_j,
\end{align*}
which shows that $U_{\sigma} \in \mathfrak{O}_c(U)$.
\end{proof}

In the previous corollary, notice that when $\sigma = \operatorname{id}$, then $U_{\sigma} = U^{*}$ belongs to the conjugate orbit of $U$, which confirms Proposition \ref{SdfsdfdsfcvvvvVV}. One can formulate a finite dimensional version of the previous proposition except that $\sigma$ will be an order two bijection of $\{1, 2, \ldots, n\}$. We record this with the following example. 

\begin{Example}\label{Ex2too}
If
$U = \operatorname{diag}(\xi_1, \xi_2, \ldots, \xi_n)$ is a diagonal unitary matrix and $\sigma \in S_{n}$ (the symmetric group on $\{1, 2, \ldots, n\}$) with order two, then the unitary matrix 
$$U_{\sigma} := \operatorname{diag}(\bar \xi_{\sigma(1)}, \bar \xi_{\sigma(2)}, \ldots,  \bar \xi_{\sigma(n)})$$ belongs to $\mathfrak{O}_c(U)$. We will revisit this example below in Corollary \ref{sdfugsdfg99991008}.
Though one might be initially tempted to believe that every member of the conjugate orbit of a diagonal matrix will also be diagonal, we will show in Example \ref{twobytwo} that this is rarely  the case. 
We will explore a generalization as to which diagonal matrices belong to the conjugate orbit in Theorem \ref{6771109999}. 
\end{Example}

\begin{Example}\label{nknjdhfkhfghHJHJHJ}
Let
\begin{equation}\label{FTFTFT}
 (\mathcal{F}f)(x) = \frac{1}{\sqrt{2 \pi}} \int_{-\infty}^{\infty} f(t) e^{-ixt} dt,
 \end{equation}
  denote the classical Fourier--Plancherel transform on $L^2(\R)$ (defined initially on $L^1(\R) \cap L^2(\R)$ by the above integral and extended to be a unitary operator on $L^2(\R)$). 
  It is known that $\mathcal{F}$ has the spectral properties  $\sigma(\mathcal{F}) = \sigma_{p}(\mathcal{F}) = \{1, -1, i, -i\}$ and 
  \begin{equation}\label{Fteigenv-aces}
 \mathcal{F} H_n = (-i)^n H_n,
 \end{equation}
where $H_n$ is the $n$th Hermite function, i.e., $$H_{n}(x) = c_n e^{-\frac{x^2}{2}} h_{n}(x), \; \; n \geq 0,$$
$c_n$ is a normalizing constant so that $\|H_n\|_{L^2(\R)}= 1$, and $h_n$ is the $n$-th Hermite polynomial.
 Moreover, $(H_{n})_{n \geq 0}$ forms an orthonormal basis for $L^2(\R)$ \cite[Ch.11]{MR4545809}.
 Thus, by our previous discussion, we can write
 \begin{equation}\label{diagonalizeFFF}
 \mathcal{F} = \sum_{n = 0}^{\infty} (-i)^n (H_{n} \otimes H_n).
 \end{equation}
Corollary \ref{C27} applied to any bijection $\sigma: \N_0 \to \N_0$ with $\sigma \circ \sigma = \operatorname{id}$ says that the  unitary operator $\mathcal{F}_{\sigma}$ on $L^2(\R)$ defined by
 $$\mathcal{F}_{\sigma}  := \sum_{n = 0}^{\infty} i^{\sigma(n)} (H_{n}\otimes H_{n})$$ belongs to $\mathfrak{O}_c(\mathcal{F})$. Observe that if $\sigma = \operatorname{id}$, then $\mathcal{F}_{\sigma} = \mathcal{F}^{*}$ belongs to the conjugate orbit of $\mathcal{F}$, which is expected by Proposition \ref{SdfsdfdsfcvvvvVV}. If $\sigma: \N_0 \to \N_0$ is defined by $\sigma(2n) = 2n$, $\sigma(1) = 3, \sigma(3) = 1, \sigma(5) = 7, \sigma(7) = 5$ and so on, we see that $i^{\sigma(n)} = (-i)^n$ and thus $\mathcal{F}_{\sigma} = \mathcal{F}$ belongs to the conjugate orbit of $\mathcal{F}$. Note that the Fourier--Plancherel transform is unitarily equivalent to its adjoint. See the paragraph following Proposition \ref{e0roiekfgdFF}.
\end{Example}

\begin{Example}\label{ldfgjdfjg888uuUU}
Let
 \begin{equation}\label{Hilberttransformmmm}
 (\mathscr{H} g)(x) = \frac{1}{\pi}  \operatorname{PV} \int_{-\infty}^{\infty} \frac{g(t)}{x - t} dt,
 \end{equation}
  denote the Hilbert transform on $L^2(\R)$. It is known that $\mathscr{H}$ is unitary on $L^2(\R)$ (in fact $\mathscr{H}^{*} = - \mathscr{H}$). 
Moreover, $\mathscr{H}$ is unitarily equivalent to its adjoint via the intertwining  unitary operator on $L^2(\R)$ defined by  $f(x) \mapsto f(-x)$ and  $\mathscr{H}$ has the spectral properties $\sigma(\mathscr{H}) = \sigma_{p}(\mathscr{H}) = \{i, -i\}$ with 
$$\mathscr{H} f_n = i f_n, \;  \mbox{where} \; 
f_n(x) = \frac{1}{\sqrt{\pi}} \frac{(x + i)^{-n - 1}}{(x - i)^{-n}}, \quad n \leq - 1,$$
and 
$$\mathscr{H} f_n = - i f_n, \; \mbox{where} \; 
f_{n}(x) = \frac{1}{\sqrt{\pi}} \frac{(x - i)^{n}}{(x + i)^{n + 1}}, \quad n \geq 0.$$
It is known  $(f_n)_{n \in \Z}$ forms an orthonormal basis for $L^2(\R)$.
See \cite[Ch. 12]{MR4545809} for the above Hilbert transform details. Thus,
\begin{equation}\label{diagonalizeH}
\mathscr{H} = \sum_{n < 0} i (f_n \otimes f_n) + \sum_{n \geq 0} (-i) (f_n \otimes f_n).
\end{equation}
In other words, with respect to the ordered  orthonormal basis $(f_{n})_{n = - \infty}^{\infty}$, $\mathscr{H}$ is represented by the doubly infinite diagonal matrix
$$\operatorname{diag}(\ldots, i, i, i, \boxed{-i}, -i, -i, \ldots),$$
where the boxed term above corresponds to the $0$\textsuperscript{th} position. 
For any bijection $\sigma: \Z \to \Z$ for which $\sigma \circ \sigma = \operatorname{id}$,  the unitary operator $\mathscr{H}_{\sigma}$ on $L^2(\R)$ defined by the permutation $\sigma$ of the entries of 
$$\operatorname{diag}(\ldots, -i, -i, -i, \boxed{i}, i, i, \ldots)$$
belongs to $\mathfrak{O}_c(\mathscr{H})$. Certainly the identity permutation will yield $\mathscr{H}^{*} \in \mathfrak{O}_c(\mathscr{H})$ (as to be expected from Proposition \ref{e0roiekfgdFF}). The permutation $\sigma$ on $\Z$ defined by $-1 \leftrightarrow 0, -2 \leftrightarrow 1, -3 \leftrightarrow 2, \ldots, -n  \leftrightarrow n - 1, \ldots$, will yield $\mathscr{H} \in \mathfrak{O}_c(\mathscr{H})$. See the discussion following Proposition \ref{e0roiekfgdFF}.
\end{Example}

\section{Alternate representations of members of the orbit}

To get an idea of  the ``size'' of $\mathfrak{O}_c(U)$, it is useful to discuss when
$C_1 U C _1 = C_2 U C_2.$
In other words, when do two conjugations $C_1$ and $C_2$ yield the same member of $\mathfrak{O}_c(U)$? For a unitary operator $U$ on $\mathcal{H}$, recall that the {\em commutant} of $U$, denoted by $\{U\}'$, is the collection of  all bounded linear operators $A$ on $\mathcal{H}$ for which $AU = UA$ (see \cite{MR1721402} for general facts about the commutant). 

\begin{Proposition}\label{09876tygbhn}
For a fixed unitary operator $U$ and conjugations $C_1, C_2$ on $\mathcal{H}$, the following are equivalent.
\begin{enumerate}
\item $C_1 U C_1 = C_2 U C_2$.
\item $C_1 C_2$ is a unitary operator in the commutant of $U$.
\end{enumerate}
\end{Proposition}

\begin{proof}
If (a) holds then, manipulating the identity $C_1 U C_1 = C_2 U C_2$, gives us 
$(C_1 C_2) U = U (C_1 C_2).$ Since $C_1 C_2$ is unitary, we see that  (b) holds. Conversely, suppose (b) holds in that $C_1 C_2$ commutes with $U$. A similar working of the terms in  $(C_1 C_2) U = U (C_1 C_2)$ yields (a).
\end{proof}

\begin{Remark}\label{commuttannt}
The previous proposition gives us a way of understanding the richness of the conjugate orbit of a particular unitary operator. For example, if $U$ is a unitary operator on the sequence space $\ell^2 = \ell^2(\N)$ then, with respect to the standard orthonormal basis $(\vec{e}_j)_{j \geq 1}$, every conjugation  on $\ell^2$ can be written as $C = V J$, where $V$ is an infinite unitary matrix with $V^{t} = V$ and $J$ is the standard conjugation on $\ell^2$ defined by  $J \vec{x} = \overline{\vec{x}}$ (see below). Here $V$ is an infinite unitary matrix that defines an operator on $\ell^2$ (viewed as column vectors) by $\vec{x} \mapsto V \vec{x}$ and $V^{t}$ denotes its transpose.  With $C_{1} = V_{1} J$ and $C_2 = V_2 J$ we have
$C_{1} C_2 = V_1 J V_2 J = V_{1} \overline{V_2} = V_{1} V_{2}^{*}.$ Here, recall that $\overline{V}$ is the matrix $V$ with all of its elements conjugated (and thus $J V J = \overline{V}$) and $V^{*}$ denotes the conjugate transpose of $V$.
Thus, $C_1 U C_1 = C_2 U C_2$ if and only if $V_{1} V_{2}^{*}$ is a unitary operator in  the commutant of $U$.
\end{Remark}

\begin{Example}\label{lksjnfkdghdu88y8Y*YYYY}
Keep the notation as in  Corollary \ref{C27}, where 
$$U = \sum_{j \geq 1} \xi_j (\vec{u}_j \otimes \vec{u}_j).$$
 As in Corollary \ref{C27}, let $\sigma_1, \sigma _2$ be two bijections of $\N$ with $\sigma_{1} \circ \sigma_{1} =  \sigma_2 \circ \sigma_2 = \operatorname{id}$ and $C_{\sigma_1}, C_{\sigma_2}$ the corresponding conjugations on $\mathcal{H}$ which permute each basis vector $\vec{u}_j$ as in \eqref{66tt55RFDDS} and then extended antilinearly to all of $\mathcal{H}$ via \eqref{Creeeeeeel}. The proof of Corollary \ref{C27} says that 
 $U_{\sigma_1} = C_{\sigma_1} U C_{\sigma_1}$ and $U_{\sigma_2} = C_{\sigma_2} U C_{\sigma_2}$ and so $$U_{\sigma_1} = U_{\sigma_2} \; \mbox{when} \; C_{\sigma_1} C_{\sigma_2} U = U C_{\sigma_1} C_{\sigma_2}.$$ Observe that 
 $$C_{\sigma_1} C_{\sigma_2} U \vec{u}_j = \xi_{j} \vec{u}_{\sigma_1 (\sigma_2(j))} \; \; \mbox{and} \; \; 
 U C_{\sigma_1} C_{\sigma_2} \vec{u}_j = \xi_{\sigma_1(\sigma_{2}(j))}  \vec{u}_{\sigma_1 (\sigma_2(j))}.$$
 This holds precisely when 
 $\xi_{\sigma_1 \circ \sigma_2 (j)} = \xi_{j}$ for all $j \geq 1$. When  all of the $\xi_j$ are {\em distinct}, we see that $U_{\sigma_1}$ and $U_{\sigma_2}$ are distinct elements of $\mathfrak{O}_c(U)$ if and only if $\sigma_{1}(j) \not = \sigma_{2}(j)$ for at least one $j$.
\end{Example}

\begin{Example}
Let us consider the Fourier--Plancherel transform $\mathcal{F}$ as in Example \ref{nknjdhfkhfghHJHJHJ}. For two bijections $\sigma_1, \sigma _2$  of $\N_0$ with $\sigma_{1} \circ \sigma_1 = \sigma_{2} \circ \sigma_{2} = \operatorname{id}$, the two unitary operators
$\mathcal{F}_{\sigma_1}$  and  $\mathcal{F}_{\sigma_2}$ are distinct elements of the orbit of $\mathcal{F}$ if and only
$(-i)^{\sigma_1 \circ \sigma_2(j)} \not = (-i)^j$ for at least one $j$. In other words, $\sigma_1 \circ \sigma_2(j) \not = j$ (modulo $4$) for at least one $j$. A similar type of result holds for the Hilbert transform from Example \ref{ldfgjdfjg888uuUU}.
\end{Example}

\begin{Example}
Suppose that $(U f)(\xi) = \xi f(\xi)$ is the bilateral shift on $L^2(m)$. It is well  known \cite[p.~183]{MR4545809} that the commutant of $U$ consists of all the multiplication operators $M_{\phi} f = \phi f$ on $L^2(m)$, where $\phi \in L^{\infty}(m)$. Moreover, since $M_{\phi}^{*} = M_{\bar{\phi}}$, we see that $M_{\phi}$ is unitary precisely when $\phi$ is unimodular almost everywhere on $\T$. Thus, for conjugations $C_1, C_2$ on $L^2(m)$, we have $C_1 U C_1 = C_2 U C_2$ if and only if $C_1 C_2 = M_{\phi}$ for some unimodular $\phi$. We will say more about this in Remark \ref{QPLAKSJFHG}. As a specific example, observe that if $u$ and $v$ are unimodular on $\T$, then $C_{u} f = u \bar{f}$ and $C_{v} f = v \bar{f}$ define conjugations on $L^2(m)$ with $C_{u} U C_{u} = U^{*}$ and $C_{v} U C_{v} = U^{*}$. A short computation reveals  that $C_{u} C_{v} = M_{u \overline{v}}$. We have already seen these conjugations in Remark \ref{nNNnniiinnnndddD} and   will see them again in Example \ref{examplecalculus}.
\end{Example}

\begin{Example}\label{MxionL63}
Suppose that $L^2(\mu)$ is any Lebesgue space, where $\mu$ is a finite positive measure on $\T$. From \cite[Prop.~7.2]{MR4077555}, one can show that the mapping
$$C f = \bar f \circ \alpha,$$ where $\alpha\colon \T\to \T$ is $\mu$-measurable, defines a conjugation if and only if 
\begin{equation}\label{measpres}
\alpha \circ \alpha = \operatorname{id} \; \; \mbox{and $\mu$ is measure preserving}.
\end{equation}
The papers \cite{MR2618485, MR355679, MR492057} contain a discussion of isometric composition operators on general Lebesgue spaces. 
Define 
$M_{\xi}$ on $L^2(\mu)$ by 
$$M_{\xi} f = \xi f.$$
 Let $C_{\alpha_1},C_{\alpha_2}$ be the conjugations on $ L^2(\mu)$ given by $C_{\alpha_j} f = \bar{f} \circ \alpha_j$ with $\alpha_1,\alpha_2\colon \T\to\T$ being $\mu$-measure preserving and convolutive mappinngs as in \eqref{measpres}.
If  $C_{\alpha_1}M_{\xi} C_{\alpha_1}=C_{\alpha_2}M_{\xi} C_{\alpha_2}$ then $\alpha_1=\alpha_2$, $\mu$ almost everywhere. Indeed,
$C_{\alpha_1}M_{\xi} C_{\alpha_1}=C_{\alpha_2}M_{\xi} C_{\alpha_2}$ if and only if  $C_{\alpha_1} C_{\alpha_2}\in\{M_{\xi}\}^\prime$, the commutant of $M_{\xi}$.
On the other hand, as discussed in the previous example,  $\{M_{\xi}\}^\prime=\{M_\varphi\colon \varphi\in L^\infty(\mu)\}$.
Thus, $C_{\alpha_1} C_{\alpha_2}=M_{\phi}$ with $\phi$ unimodular $\mu$ almost everywhere.
Note that
$$
C_{\alpha_1} C_{\alpha_2} f = f \circ (\alpha_2 \circ \alpha_1)  = A_{\alpha_2\circ\alpha_1} f,
$$
where $A_{\alpha_2\circ\alpha_1}$ denotes the composition operator on $L^2(\mu)$ with symbol $\alpha_{2} \circ \alpha_{1}$. 
Letting $f = \chi = \chi_{\T}$ (the constant function equal to one on $\T$) in the above, we have
$$\varphi(\xi) = \phi(\xi) \chi(\xi)= \chi \circ \alpha_2\circ \alpha_1 (\xi).$$
Since  $\alpha_2\circ \alpha_1$ is a bijection, we see that $\varphi = 1$ $\mu$ almost everywhere.
Towards a contradiction, assume now that  
$$0<\mu\{ \xi \in \T:  \alpha_1(\xi)\not=\alpha_2(\xi)\}=\mu\{\xi \in \T:  (\alpha_2\circ\alpha_1)(\xi)\not=\xi\}.$$
This would contradict the fact that  $A_{\alpha_2\circ\alpha_1} e = e$, where $e$ denotes the identity function $e(\xi) = \xi$ on $\T$. Putting this all together, we have shown the following: Let $U = M_{\xi}$ on $L^2(\mu)$, where $\mu$ is a positive measure on $\T$. For any $\mu$-measurable $\alpha\colon\T\to\T$ as in \eqref{measpres}, define $C_\alpha$  on $L^2(\mu)$ by $C_{\alpha} f = \bar f\circ \alpha$. Then $C_\alpha$ is a conjugation; $C_{e}M_{\xi} C_{e}= M_{\xi}^{*}$; and each $\alpha$ (defined uniquely up to $\mu$ almost everywhere) gives a distinct  element  $C_\alpha M_{\xi} C_\alpha$ of $\mathfrak{O}_c(M_{\xi})$.
It is worth noticing that diagonal unitary operators on $\ell^2$ with distinct eigenvalues fulfills the hypothesis of the above discussion and Example \ref{lksjnfkdghdu88y8Y*YYYY} can be explored with this set up.
\end{Example}

In Chapter \ref{S-Five} we will continue the discussion from the previous example when $\mu$ is normalized Lebesgue measure $m$ on $\T$ and $M_{\xi}$ above will be the classical bilateral shift on $L^2(m)$. 

\chapter{Unitary Matrices}\label{S-Four}

This chapter focuses on the conjugate orbit of a unitary matrix. Let $\mathscr{U}_{n}$ denote the set of all $n \times n$ complex unitary matrices $U$ ($U^{*} U = U U^{*} = I$) which we view as unitary operators on $\C^n$ (as columns)  in the usual way by $\vec{x} \mapsto U \vec{x}$, $\vec{x} \in \C^n$.

\section{Unitary matrices and conjugations}

Let $J$ denote the conjugation on $\C^n$ given by $J \vec{x} = \overline{\vec{x}}$, that is,
$$J \begin{bmatrix} x_1\\ x_2\\ \vdots\\ x_n \end{bmatrix} = \begin{bmatrix} \bar x_1\\ \bar x_2\\ \vdots\\ \bar x_n \end{bmatrix}.$$
For $U \in \mathscr{U}_{n}$ one  can check that
\begin{equation}\label{youuuu}
J U J = \overline{U},
\end{equation}
where, as a reminder, $\overline{U}$ is the matrix consisting of the complex conjugates of the entries of $U$.
 A result of Garcia and Tener \cite[Lemma 3.2]{MR2966041} says that {\em any} conjugation $C$ on $\C^n$ can be written as $C = V J$, where $V \in \mathscr{U}_{n}$ for which $V = V^{t}$ (the transpose of $V$).
 Note that since $V$ is unitary we have 
 \begin{equation}\label{99hhHHaa1212344}
V^{t} = V \iff V \overline{V} = I \iff  \overline{V} V = I.
\end{equation}
For completeness, we include its short proof. 
 
 \begin{Proposition}\label{GTenner}
$C$ is a conjugation on $\C^n$ if and only if $C = V J$, where $V \in \mathscr{U}_n$ with $V^{t} = V$ and $J$ is the conjugation on $\C^n$ defined by $J \vec{x} = \overline{\vec{x}}$. 
 \end{Proposition}
 
 \begin{proof}
 If $C$ is a conjugation on $\C^n$ then $V = C J$ is a unitary operator on $\C^n$ which we think of as an element of $\mathscr{U}_n$. Moreover, 
 $$V \overline{V} = V JVJ = C^2 = I.$$ From \eqref{99hhHHaa1212344} it follows that $V = V^{t}$. Conversely, suppose that $V \in \mathscr{U}_n$ and $V^{t} = V$, equivalently, again by \eqref{99hhHHaa1212344}, $V \overline{V} = I$. Then $C = V J$ is antilinear and isometric on $\C^n$. Moreover, 
 $$C^2 = V J VJ = V \overline{V} = I$$ and so $C$ is a conjugation.
 \end{proof}
 
  A large class of such self transpose unitary matrices $V$ can be created by using the Householder matrices $V = I - 2 \vec{v} \otimes \vec{v}$, where $\vec{v}$ is any unit vector in $\C^n$ (or diagonal blocks of such matrices). 
Proposition \ref{GTenner} extends to any (singly or doubly) infinite  unitary matrix $U$ acting on $\ell^2(\N_{0})$ or $\ell^2(\Z)$ (see Proposition \ref{OShfgfdsfa} below with the same proof as the previous proposition).

\begin{Proposition}\label{dfgdfgkkKKK12}
For fixed $U \in \mathscr{U}_n$ we have
 $$\mathfrak{O}_c(U) = \{V \overline{U} V^{*}: \mbox{$V \in \mathscr{U}_n$ and $V^{t} = V$}\}.$$
\end{Proposition}

\begin{proof}
Using Proposition \ref{GTenner}, observe that 
\begin{align*}
\mathfrak{O}_c(U) & = \{C U C: \mbox{$C$ is a conjugation}\}\\
& = \{(VJ) U (VJ): \mbox{$V \in \mathscr{U}_n$  and $V^{t} = V$}\}\\
& =  \{(VJ) U J J (VJ): \mbox{$V \in \mathscr{U}_n$  and $V^{t} = V$}\}\\
& =  \{V(J U J) (J VJ): \mbox{$V \in \mathscr{U}_n$  and $V^{t} = V$}\}\\
& =  \{V \overline{U}\, \overline{V}: \mbox{$V \in \mathscr{U}_n$  and $V^{t} = V$}\}\\
& = \{V \overline{U} V^{*}: \mbox{$V \in \mathscr{U}_{n}$ and $V^{t} = V$}\}.
\end{align*}
In the above calculations, note the use of \eqref{youuuu} and \eqref{99hhHHaa1212344}. 
\end{proof}

\begin{Corollary}\label{Ubbarrraaaa}
$\overline{U} \in \mathfrak{O}_c(U)$
for any $U \in \mathscr{U}_n$.
\end{Corollary}

Also observe the following corollary of Proposition \ref{dfgdfgkkKKK12} which determines when two elements of the conjugate orbit are the same. 

\begin{Corollary}\label{samememmee}
If $V_1, V_2, U \in \mathscr{U}_{n}$ with $V_1$ and $V_2$ are self transpose, then $V_1 \overline{U} V_{1}^{*} = V_2 \overline{U} V_{2}^{*}$ if and only if $V_{2} \overline{V_{1}}$ belongs to the commutant of $U$. 
\end{Corollary}

A version of Proposition \ref{ksdlkfjsldf88} for unitary matrices is the following.

\begin{Proposition}
Suppose that $U \in \mathscr{U}_n$. Then $W \in \mathfrak{O}_c(U)$ if and only if there is an orthonormal basis $\mathcal{E}$ for $\C^n$ so that $ [W]_{\mathcal{E}} = \overline{[U]_{\mathcal{E}}}$.
\end{Proposition}

From the spectral theorem for unitary matrices (representing them as diagonal matrices with respect to an orthonormal basis), we have the following corollary.

\begin{Corollary}
Let $U \in \mathscr{U}_n$ with the ordered orthonormal basis of eigenvectors
$\mathcal{E} = (\vec{u}_{j})_{j = 1}^{n} $ and
$U \vec{u}_j = \xi_j \vec{u}_j$ for $j = 1, \ldots, n$. Then the unitary matrix $W$ such that
$$[W]_{\mathcal{E}} = \operatorname{diag}(\bar\xi_1, \ldots, \bar \xi_n),$$
in other words $W \vec{u}_j = \bar \xi_j \vec{u}_j$ for all $1 \leq j \leq n$, belongs to $\mathfrak{O}_c(U)$.
\end{Corollary}

It is important to point out that $ \operatorname{diag}(\bar\xi_1, \ldots, \bar \xi_n)$ represents  $U^{*}$ with respect to the orthonormal basis $\mathcal{E}$ for $\C^n$ but may not represent $U^{*}$ with respect to the standard orthonormal basis for $\C^n$  (see Theorem \ref{6771109999} and Example \ref{diagoanlalalal}).

\section{A simple example}

Even simple $2 \times 2$ diagonal unitary matrices tell us a lot about the conjugate orbit and provide nice examples of the results we have so far. 

\begin{Example}\label{twobytwo}
 Let $U$ be  the diagonal unitary matrix 
$$U = \begin{bmatrix} \xi_1 & 0\\ 0 & \xi_2 \end{bmatrix},$$
where $\xi_1, \xi_2 \in \T$ and distinct.  From Proposition \ref{dfgdfgkkKKK12} we know that
$$\mathfrak{O}_c(U) = \{V \overline{U} V^{*}: \mbox{$V \in \mathscr{U}_2$ and $V^{t} = V$}\}.$$ A standard description of all $2 \times 2$ unitary matrices $V$  is
$$V = e^{i \frac{\varphi}{2}}\begin{bmatrix} e^{i \alpha} \cos \theta & e^{i \beta} \sin \theta\\
- e^{-i\beta} \sin \theta & e^{-i \alpha} \cos \theta\end{bmatrix},$$
where $\varphi, \alpha, \beta, \theta \in \R$. One can check, we omit the calculation, that by imposing the condition that $V^{t} = V$, the matrix $V$ takes the form
\begin{equation}\label{VVVVVmatrix}
V = e^{i \frac{\varphi}{2}}\begin{bmatrix} e^{i \alpha} \cos \theta & i (-1)^n \sin \theta\\
i (-1)^n \sin \theta & e^{-i \alpha} \cos \theta\end{bmatrix}.
\end{equation}

When $\theta = n \pi$, the matrix  $V$ becomes
\begin{equation}\label{VeeVV}
V = e^{i \frac{\varphi}{2}}\begin{bmatrix} (-1)^n e^{i \alpha} & 0\\
0 & (-1)^n e^{-i \alpha} \end{bmatrix}
\end{equation} and thus
\begin{equation}\label{9isudjhfjhhhY11}
V \overline{U} V^{*} = \begin{bmatrix} \bar \xi_1 & 0\\ 0 & \bar \xi_2 \end{bmatrix},
\end{equation}
which is $U^{*}$ (as to be expected by Corollary \ref{Ubbarrraaaa}). 
Also observe that the above is true for any choice of $\alpha$, $n$, and $\varphi$ (as long as $\theta = n \pi$). So there are many choices of $V_{\alpha, n, \varphi}$ from \eqref{VeeVV} that yield the same member of the orbit of $U$. This should come as no surprise since $V_{\alpha, n, \varphi} \overline{V_{\alpha', n', \varphi'}}$ is a diagonal matrix which clearly commutes with the diagonal matrix $U$. Note the use of Corollary \ref{samememmee} here.

 When $\theta = \frac{1}{2} \pi (2 n + 1)$, the matrix $V$ from \eqref{VVVVVmatrix} becomes
$$V = e^{i \frac{\varphi}{2}}\begin{bmatrix} 0 & i\\
i & 0\end{bmatrix}$$
and so
$$V \overline{U} V^{*} = \begin{bmatrix} \bar \xi_2 & 0\\
0 & \bar \xi_1\end{bmatrix}.$$
Notice the diagonal entries are reversed from those in \eqref{9isudjhfjhhhY11}.
Again, many choices of $V$ (for various $\varphi$) yield the same member of the orbit (see Corollary \ref{samememmee}).

In general,  $\mathfrak{O}_c(U)$ consists of all the matrices $V \overline{U} V^{*}$, where $V$ is of the form \eqref{VVVVVmatrix}, and these turn out to be
$$\begin{bmatrix} e^{i \alpha} \cos \theta & i (-1)^n \sin \theta\\
i (-1)^n \sin \theta & e^{-i \alpha} \cos \theta\end{bmatrix} \begin{bmatrix} \bar \xi_1 & 0\\ 0 & \bar \xi_2 \end{bmatrix} \begin{bmatrix} e^{- i \alpha} \cos \theta & i (-1)^{n + 1} \sin \theta\\
i (-1)^{n + 1} \sin \theta & e^{i \alpha} \cos \theta\end{bmatrix}.$$

As a specific example, when $\xi_1 = 1$ and $\xi_2 = -1$ we obtain the family of matrices
$$\begin{bmatrix}
\cos (2 \theta) & -i e^{i \alpha} (-1)^n \sin (2 \theta) \\
 i e^{-i \alpha}  (-1)^n \sin (2 \theta) & -\cos (2 \theta)
 \end{bmatrix}, \quad \theta, \alpha \in \R, n \in \N_{0},$$
 as the orbit of 
 $$U = \begin{bmatrix} 1 & \phantom{-}0\\ 0 & -1\end{bmatrix}.$$ Here it is important to observe that even though the diagonal matrices $U$ (setting $\theta = 0$) and 
 $$\begin{bmatrix} -1 &0\\\phantom{-}0 & 1\end{bmatrix}$$
  (setting $\theta = \pi/2$) belong to $\mathfrak{O}_c(U)$, there are plenty of {\em nondiagonal} matrices that belong to $\mathfrak{O}_c(U)$, even though $U$ is diagonal. 

 What about the question as to when different
 $$V_{\alpha, \theta, n} := \begin{bmatrix} e^{i \alpha} \cos \theta & i (-1)^n \sin \theta\\
i (-1)^n \sin \theta & e^{-i \alpha} \cos \theta\end{bmatrix}$$
 yield the same member of the orbit of a diagonal $U$ with distinct diagonal entries. Recall that the commutant of a diagonal matrix with distinct entries consists of all the diagonal matrices. Here we are looking for  instances as to when
 $V_{\alpha, \theta, n} V_{\beta, \phi, m}^{*}$ is a diagonal matrix, with of course unimodular entries.  One can calculate
 $V_{\alpha, \theta, n} V_{\beta, \phi, m}^{*}$ to be
 $$
{\tiny  \begin{bmatrix}
 \sin \theta (-1)^{m+n} \sin \phi +e^{i (\alpha -\beta )} \cos \theta  \cos \phi  & i
   \left(e^{i \beta } (-1)^n \sin \theta  \cos \phi -e^{i \alpha } (-1)^m \cos \theta  \sin \phi
   \right) \\
 i \left(e^{-i \beta } (-1)^n \sin \theta  \cos \phi -e^{-i \alpha } (-1)^m \cos \theta  \sin
   \phi \right) & \sin \theta  (-1)^{m+n} \sin \phi +e^{-i (\alpha -\beta )} \cos \theta  \cos
   \phi
   \end{bmatrix}.}$$
   This will be a diagonal matrix when the parameters satisfy the condition 
   $$ e^{i \beta } (-1)^n \sin \theta  \cos \phi -e^{i \alpha } (-1)^m \cos \theta  \sin \phi = 0$$
which can indeed  happen for certain well chosen values of $\alpha, \beta, \theta, \phi, n, m$. 
\end{Example}

\section{Diagonal unitary matrices}

One needs to be careful in the above discussion since, for a fixed $U \in \mathscr{U}_n$,  the  {\em unitary orbit} $\mathfrak{O}(U)$, the set of all 
$V U V^{*},$ where $V \in \mathscr{U}_n$ (which as mentioned earlier turns out to be the set of all matrices in $\mathscr{U}_n$ which have the same eigenvalues with the same corresponding multiplicities as $U$), can be much different than the {\em conjugate orbit} of $U$ which consists of the collection of unitary matrices $V \overline{U} V^{*}$, where $V \in \mathscr{U}_n$ and $V = V^{t}$. Certainly, by Proposition \ref{SdfsdfdsfcvvvvVV} and Proposition \ref{uuueeeqq}, we know that any member of $\mathfrak{O}_c(U)$ must by unitary equivalent to $U^{*}$. However, even a diagonal matrix consisting of the eigenvalues of  $U^{*}$ (with corresponding multiplicities) need not belong to $\mathfrak{O}_c(U)$. Below is a theorem that sorts this out, followed by two cautionary examples. 

\begin{Theorem}\label{6771109999}
Suppose that $U \in \mathscr{U}_n$ with eigenvalues $(\xi_j)_{j = 1}^{n}$, repeated according to multiplicity. Then $\operatorname{diag}(\bar \xi_1, \bar \xi_2, \ldots, \bar \xi_n) 
\in \mathfrak{O}_c(U)$ if and only if there is an ordered orthonormal basis $(\vec{u}_j)_{j = 1}^{n}$ for $\C^n$ with $U \vec{u}_j = \xi_j \vec{u}_j$ for all $1 \leq j \leq n$ and such that  the unitary matrix $V = [\vec{u}_1 | \vec{u}_2| \cdots |\vec{u}_n]$
satisfies $V^{t} = V$. 
\end{Theorem}

\begin{proof}
Suppose that $(\vec{u}_j)_{j = 1}^{n}$ is an orthonormal eigenbasis for $U$ in that $U \vec{u}_j = \xi_j \vec{u}_j$ for all $1 \leq j \leq n$ and that  this basis satisfies the condition that $V = [\vec{u}_1 | \vec{u}_2| \cdots |\vec{u}_n]$ is self transpose unitary matrix. Then, as seen earlier,  by a result of Garcia and Tenner  \cite[Lemma 3.2]{MR2966041} (see Proposition \ref{GTenner}), $C = V J$, more precisely, $C \vec{x} = V \overline{\vec{x}}$, defines a conjugation on $\mathbb{C}^n$. Furthermore, 
\begin{align*}
CUC & = (V J) U (VJ)\\
& = (VJ) U JJ (VJ)\\
& = V (JUJ) (JVJ)\\
& = V \overline{U}\,\overline{V}\\
& = V \overline{U} [\overline{\vec{u}_1} | \overline{\vec{u}_2}| \cdots |\overline{\vec{u}_n}]\\
& = V  [\overline{U}\overline{\vec{u}_1} | \overline{U}\overline{\vec{u}_2}| \cdots |\overline{U}\overline{\vec{u}_n}]\\
& = V[\bar \xi_1 \overline{\vec{u}_1} | \bar \xi_2 \overline{\vec{u}_2}| \cdots | \bar \xi_n\overline{\vec{u}_n}]\\
& = V   [\overline{\vec{u}_1} | \overline{\vec{u}_2}| \cdots |\overline{\vec{u}_n}] \operatorname{diag}(\bar \xi_1, \bar \xi_2, \ldots, \bar \xi_n)\\
& = V \overline{V} \operatorname{diag}(\bar \xi_1, \bar \xi_2, \ldots, \bar \xi_n)\\
& = \operatorname{diag}(\bar \xi_1, \bar \xi_2, \ldots, \bar \xi_n),
\end{align*}
where, in the last step of the above calculation, we use the fact that $V$ is unitary and self transpose along with  \eqref{99hhHHaa1212344}. Thus, $\operatorname{diag}(\bar \xi_1, \bar \xi_2, \ldots, \bar \xi_n) \in \mathfrak{O}_c(U)$. 

Conversely, suppose that $\operatorname{diag}(\bar \xi_1, \bar \xi_2, \ldots, \bar \xi_n) \in \mathfrak{O}_c(U)$. Then there is a conjugation $C$ on $\mathbb{C}^n$ such that $C U C = \operatorname{diag}(\bar \xi_1, \bar \xi_2, \ldots, \bar \xi_n)$. As discussed before,  there is a unitary matrix $V$ with $V^t = V$ such that $C = V J$, i.e., $C \vec{x} = V \overline{\vec{x}}$. By the calculation above we have 
$$\operatorname{diag}(\bar \xi_1, \bar \xi_2, \ldots, \bar \xi_n) = C U C  = V \overline{U} \,\overline{V}.$$ Taking complex conjugates of both sides of the previous identity gives us 
$$\overline{V} U V = \operatorname{diag}(\xi_1, \xi_2, \ldots, \xi_n),$$
and, again by noting that $V \overline{V} = I$, we see that 
\begin{equation}\label{09993397711}
U V = V\operatorname{diag}(\xi_1, \xi_2, \ldots, \xi_n).
\end{equation}
Multiplying  each side of the above on the right by each of the standard basis elements $\vec{e}_j$, $1 \leq j \leq n$, for $\C^n$, we get 
$$U (V \vec{e}_j) = \xi_j (V \vec{e}_j), \quad 1 \leq j \leq n.$$
This says that $(V \vec{e}_j)_{j = 1}^{n}$ is an orthonormal basis of eigenvectors for $U$ and, by assumption,   the matrix 
$V = [V \vec{e}_1| V\vec{e}_2| \cdots | V \Vec{e}_n]$
is self transpose. 
\end{proof}

One must be careful with the previous theorem in that not every diagonal matrix consisting of the eigenvalues of $U^{*}$ (with corresponding multiplicities) belongs to the conjugate orbit of $U$, just those which can be realized by eigenvectors whose columns from a self transpose matrix. This point  is made more salient with  the following example. 

\begin{Example}\label{diagoanlalalal}
For {\em distinct} $\xi_1, \xi_2, \xi_3 \in \T$, let $U$ be the $3 \times 3$ diagonal unitary matrix defined by 
$$U = \operatorname{diag}(\xi_1, \xi_2, \xi_3).$$ Suppose there exists a conjugation $C$ on $\mathbb{C}^3$ such that $C U C = \operatorname{diag}(\bar \xi_3, \bar \xi_1, \bar \xi_2)$. Observe the reordering and the complex conjugation of the diagonal elements of $U$. Following the steps of  the proof of the previous theorem, see \eqref{09993397711}, there is a self transpose $3 \times 3$ unitary matrix $V$ such that 
$$UV = V\!\operatorname{diag}(\xi_3, \xi_1, \xi_2).$$
Multiplying the vector $\vec{e}_1$ to the right hand side of each matrix from the previous line to see that 
$$U(V \vec{e}_1) = \xi_3 (V \vec{e}_1).$$
Thus $V \vec{e}_1 \in \ker(U - \xi_3 I) \setminus \{\vec{0}\}$ and so $V \vec{e}_1 = a \vec{e}_3$ for some $a$ with $|a| = 1$. In a similar way, $V \vec{e}_2 = b \vec{e}_1$ and $V \vec{e}_3 = c \vec{e}_2$ for $|b| = |c| = 1$. It follows that 
$$V = [V \vec{e}_1 | V \vec{e}_2 | V \vec{e}_3] = \begin{bmatrix} 0 & b & 0\\ 0 & 0 & c\\ a & 0 & 0\end{bmatrix},$$
which is not symmetric. From here we conclude that  $\operatorname{diag}(\bar \xi_3, \bar \xi_1, \bar \xi_2) \not \in \mathfrak{O}_c(U)$, even though its diagonal entries  consist of the eigenvalues of $U^{*}$. We leave it to the reader to use Corollary \ref{C27} to see that $\operatorname{diag}(\bar \xi_2, \bar \xi_1, \bar \xi_3)$ does belong to $\mathfrak{O}_c(U)$. 
\end{Example}

The previous example suggests that when the eigenvalues are {\em distinct}, the permutation of the conjugates of the eigenvalues of the diagonal matrix $U$ must be of order two. We formalize this with the following corollary. 

\begin{Corollary}\label{sdfugsdfg99991008}
If $\xi_1, \xi_2, \ldots, \xi_n$ are distinct unimodular constants,
$$U = \operatorname{diag}(\xi_1, \xi_2, \ldots, \xi_n),$$ and $\sigma$ is a permutation of $\{1, 2, \ldots, n\}$, then 
the following are equivalent.
\begin{enumerate}
\item $\operatorname{diag}(\bar \xi_{\sigma(1)}, \bar \xi_{\sigma(2)}, \ldots, \bar \xi_{\sigma(n)}) \in \mathfrak{O}_c(U)$;
\item   $\sigma \circ \sigma = \operatorname{id}$. 
\end{enumerate}
\end{Corollary}

\begin{proof}
$(b) \Longrightarrow (a)$ follows from Example \ref{Ex2too}. For the proof of $(a) \Longrightarrow (b)$, note that the identity $U \vec{e}_j = \xi_j \vec{e}_j$ for $1 \leq j \leq n$ and Theorem \ref{6771109999} says there is a there is a self-transpose unitarty matrix $V$ such that 
$$V = [V\vec{e}_{\sigma(1)} | V\vec{e}_{\sigma(2)} |\cdots |V \vec{e}_{\sigma(n)}].$$
From from \eqref{09993397711} it follows that 
$$U V = V \operatorname{diag}(\xi_{\sigma(1)}, \xi_{\sigma(2)}, \ldots, \xi_{\sigma(n)}).$$
Applying the analysis in the discussion of Example \ref{diagoanlalalal}, we see that 
$$U (V \vec{e}_j) = \xi_{\sigma(j)} (V \vec{e}_j), \quad 1 \leq j  \leq n,$$
and so $V \vec{e}_j = a_j \vec{e}_{\sigma(j)}$ for some $a_j \in \T$. This says that the matrix $V$ has an $a_j$ in the $(j, \sigma(j))$ position. However, $V$ is self-transpose and so this same $a_j$ must appear in the $(\sigma(j), j)$ position, and thus 
$$V \vec{e}_{j}  = a_j \vec{e}_{\sigma(j)} \; \mbox{and} \; V\vec{e}_{\sigma(j)} = a_j \vec{e}_j.$$
Since the eigenvalues $\xi_1, \ldots, \xi_n$ are distinct, we conclude that 
 $\sigma^2 = \operatorname{id}$
\end{proof}

This next example produces a unitary matrix such that {\em no} diagonal matrix  belongs to its conjugate orbit. 

\begin{Example}\label{qudiagonall9}
Consider the $3 \times 3$ matrix 
$$U = Q \begin{bmatrix} \xi_1 & 0 & 0\\ 0 & \xi_2 & 0\\ 0 & 0 & \xi_3\end{bmatrix} Q^{*},$$
where $\xi_1, \xi_2, \xi_3$ are {\em distinct} elements of $\T$ and $Q = [\vec{q}_1 | \vec{q}_2 | \vec{q}_3]$, where 
$$\vec{q}_1 = \begin{bmatrix}\frac{1+i}{\sqrt{3}}\\ \frac{i}{\sqrt{6}}\\ \frac{1}{\sqrt{6}}\end{bmatrix}, \; \; 
\vec{q}_2 = \begin{bmatrix}\frac{1-i}{\sqrt{6}}\\ \frac{1}{\sqrt{3}}\\\frac{i}{\sqrt{3}}\end{bmatrix}, \; \; 
\vec{q}_3 = \begin{bmatrix}0 \\ \frac{1}{\sqrt{6}}+\frac{i}{\sqrt{3}}\\\frac{1}{\sqrt{3}}-\frac{i}{\sqrt{6}}\end{bmatrix},$$
which forms an orthonormal basis for $\C^{3}$. One can see that $U$ is a unitary matrix with one-dimensional eigenspaces $\C \vec{q}_j$, $j  = 1, 2, 3$, corresponding to the distinct eigenvalues $\xi_1, \xi_2, \xi_3$. By Theorem \ref{6771109999}, a diagonal operator (which must consist of the diagonal entries $\bar \xi_1, \bar \xi_2, \bar \xi_3$ in some order) belongs to the conjugate orbit of $U$ if and only if some permutation of the (column) vectors $z_1 \vec{q}_1, z_2 \vec{q}_2, z_3 \vec{q}_3$, where $z_1, z_2, z_3$ are unimodular constants, form the columns of a self transpose matrix. For the particular unitary matrix $U$, this will not happen. To justify this, consider 
the matrix 
$$[z_1 \vec{q}_1 | z_2 \vec{q}_2 | z_2 \vec{q}_3] =
\begin{bmatrix} z_1 \frac{1+i}{\sqrt{3}} & z_2 \frac{1-i}{\sqrt{6}} & z_3 \cdot 0 \\
z_1 \frac{-i}{\sqrt{6}} & z_2 \frac{1}{\sqrt{3}} & z_3 (\frac{1}{\sqrt{6}}+\frac{i}{\sqrt{3}}) \\
z_1  \frac{1}{\sqrt{6}} & z_2 \frac{i}{\sqrt{3}} & z_3 (\frac{1}{\sqrt{3}}-\frac{i}{\sqrt{6}} )
 \end{bmatrix}.$$
 Focusing on the $(3, 1)$ and $(1, 3)$ entries of this matrix, we see there is no choice of unimodular constants $z_1$ and  $z_3$ which makes these two entries equal. As another example, consider the matrix 
 $$[w_1 \vec{q}_3 |w_2 \vec{q}_2 | w_3 \vec{q}_1] = \begin{bmatrix}
 w_1 \cdot 0 & w_2 \frac{1-i}{\sqrt{6}} & w_3 \frac{1+i}{\sqrt{3}} \\
 w_1 (\frac{1}{\sqrt{6}}+\frac{i}{\sqrt{3}} )& w_2 \frac{1}{\sqrt{3}} & w_3 \frac{-i}{\sqrt{6}} \\
 w_1 (\frac{1}{\sqrt{3}}-\frac{i}{\sqrt{6}}) & w_2 \frac{i}{\sqrt{3}} & w_3 \frac{1}{\sqrt{6}} \\
\end{bmatrix}.$$
Looking  at the $(3, 2)$ and $(2, 3)$ entries of this matrix, we see there is no choice of unimodular constants $w_2$ and $w_3$ that makes these entries equal. Now consider all six permutations of the vectors $\vec{q}_1, \vec{q}_2, \vec{q}_3$, and check that it is always possible to find two symmetric entires $(i, j)$ and $(j, i)$ that can not be made equal by multiplying by unimodular constants. Thus, in this particular example,  {\em no} diagonal matrix belongs to the conjugate orbit of $U$.
\end{Example}

\chapter{The Bilateral Shift}\label{S-Five}

The bilateral shift is an example of a unitary operator where we can concretely describe its conjugate orbit.  We will expand this discussion in various directions later in the paper. The {\em bilateral shift}  $M$ on $\ell^2(\Z)$  is initially defined by 
$$M \vec{e}_{n} = \vec{e}_{n + 1},$$
where, for each $n \in \Z$, 
$$\vec{e}_n := (\cdots 0, 0, 0, 1, 0, 0, \cdots)$$
(the $1$ appears in the $n$\textsuperscript{th} slot)
 is the two-sided standard orthonormal basis vector for the sequence space 
 \begin{equation}\label{sdfg7sd7fsdf7sdfaaaQQQQ}
 \ell^2(\Z) := \Big\{\vec{a} = (a_j)_{j = -\infty}^{\infty}: \|\vec{a}\| := \Big(\sum_{j = -\infty}^{\infty} |a_j|^{2}\Big)^{\frac{1}{2}} < \infty\Big\},
 \end{equation}
 and then extended linearly to a unitary operator on $\ell^2(\Z)$. In other words, 
 $$M\Big(\sum_{j = -\infty}^{\infty} a_{j} \vec{e}_j \Big)= \sum_{j = -\infty}^{\infty} a_j \vec{e}_{j + 1}.$$ Of course we have seen this operator earlier as its unitary equivalent realization as the multiplication operator $(M_{\xi} f)(\xi) = \xi f(\xi)$ (multiplication by the independent variable $\xi$) on $L^2(m)$, where $m$ is normalized Lebesgue measure on the unit circle $\T$. One can see this unitary equivalence via Fourier series and identifying each vector $\vec{e}_n$ in $\ell^2(\Z)$ with the monomial $\xi^{n}$ ($\xi = e^{i \theta}$) in $L^2(m)$. More precisely via the isometric isomorphism 
 $$\sum_{j = -\infty}^{\infty} a_j \vec{e}_j \mapsto \sum_{j = -\infty}^{\infty} a_j \xi^j$$
 from $\ell^2(\Z)$ onto $L^2(m)$. 

\section{The conjugate orbit as unitary shifts}

A preliminary step towards describing $\mathfrak{O}_c(M)$ is an  extension  of the result of Garcia and Tener  \cite[Lemma 3.2]{MR2966041} for  conjugations on $\C^n$ to the infinite dimensional setting of  $\ell^2(\Z)$. The proof is essentially the same as the proof of the analogous result for $\C^n$ in Proposition \ref{GTenner}. Define the conjugation 
$$J: \ell^2(\Z) \to \ell^2(\Z), \quad J \vec{x} = \overline{\vec{x}}.$$

\begin{Proposition}\label{OShfgfdsfa}
For a map $C$ on $\ell^2(\Z)$, the following are equivalent. 
\begin{enumerate}
\item $C$ is a conjugation; 
\item There is a doubly infinite unitary matrix with $V = [V_{i j}]_{i, j \in \Z}$ with $V = V^{t}$ and $C = V J$. In other words, $C \vec{x} = V \overline{\vec{x}}$.
\end{enumerate}
\end{Proposition}

\begin{Definition}
A unitary operator $W$ on $\ell^2(\Z)$ is called a {\em unitary shift} if there is an orthonormal basis $(\vec{v}_n)_{n = -\infty}^{\infty}$ for $\ell^2(\Z)$ such that $W \vec{v}_n = \vec{v }_{n + 1}$ for each $n \in \Z.$
\end{Definition}

Our description of $\mathfrak{O}_c(M)$ on $\ell^2(\Z)$ is the following.

\begin{Theorem}\label{0934t4rhhHHCccvvJ}
For a unitary operator $W$ on $\ell^2(\Z)$, the following are equivalent. 
\begin{enumerate}
\item $W \in \mathfrak{O}_c(M)$; 
\item $W$ is a unitary shift with $W \vec{v}_{j} = \vec{v}_{j + 1}$, where $V = [\cdots| \vec{v}_{-1}|\vec{v}_0|\vec{v}_1|\cdots]$ is a doubly infinite unitary matrix with $V^{t} = V$. Moreover, $W = (VJ) M (VJ)$.
\end{enumerate}
\end{Theorem}

\begin{proof}
$(a) \Longrightarrow (b)$: If $W \in \mathfrak{O}_c(M)$ then, by definition of the conjugate orbit, there is a conjugation $C$ on $\ell^2(\Z)$ with $C M C = W$. By Proposition \ref{OShfgfdsfa}, $C = V J$, where  $V$ is a doubly infinite self transpose unitary matrix. Then, with respect to the orthonormal basis  $(V \vec{e}_{n})_{n \in \Z}$ for $\ell^2(\Z)$, which are just the columns of the matrix  $V$, we have
\begin{align*}
W (V \vec{e}_n) & = 
C M C (V \vec{e}_n)\\
 & = C M C V J \vec{e}_n\\
& = C M C^2 \vec{e}_n \\
&= C M \vec{e}_n \\
&= C \vec{e}_{n + 1} \\
& = V J \vec{e}_{n + 1} \\
& = V \vec{e}_{n + 1}.
\end{align*}
Thus, $W$ is a unitary shift of the desired form.

$(b) \Longrightarrow (a)$: 
Suppose that $W$ is a unitary shift with corresponding column vectors arising from a doubly infinite unitary matrix  $V$ with $V = V^{t}$ (equivalently, $J V J = V^{*}$),  in that $W \vec{v}_n = \vec{v}_{n + 1}$ with $\vec{v}_{n} = V \vec{e}_n$ for all $n \in \Z$.  Then $C = V J$ is a conjugation on $\ell^2(\Z)$. 
Moreover,
\begin{align*}
(CMC) \vec{v}_n & = (C M C) V \vec{e}_n\\
& = C M C C J \vec{e}_n\\
&= C M C C \vec{e}_{n}\\
& = C M \vec{e}_{n}\\
& = C \vec{e}_{n + 1}\\
& = C J \vec{e}_{n + 1}\\
& = V \vec{e}_{n + 1}\\
& = \vec{v}_{n + 1}\\
& = W \vec{v}_n.
\end{align*}
The above says that the two unitary operators $W$ and $CMC$ on $\ell^2(\Z)$ agree on the orthonormal basis $(\vec{v}_n)_{n = -\infty}^{\infty}$ (and hence everywhere) and so $W \in \mathfrak{O}_c(M)$. 

To verify that $W = (VJ) M (VJ)$, observe that since $V^{t} = V$, we see that $V \overline{\vec{v}_j} = \vec{e}_j$. Thus, 
\begin{align*}
(VJ) M (VJ) \vec{v}_j & = VJ M V \overline{\vec{v}_j}\\
& = VJ M \vec{e}_j\\
& = VJ \vec{e}_{j + 1}\\
& = V \vec{e}_{j + 1}\\
& = \vec{v}_{j + 1}\\
& = W \vec{v}_{j},
\end{align*}
which proves that $W = (V J) M (VJ)$.
\end{proof}

The previous theorem says that each member of $\mathfrak{O}_c(M)$ can be viewed as shifting each column of the doubly infinite, self transpose, unitary matrix $V$ to the right. Equivalently, each element of the conjugate orbit  acts on the  columns of some 
$$V = [\cdots|\vec{v}_{-2} |\vec{v}_{-1} |\vec{v}_0 | \vec{v}_1| \cdots]$$
by $\vec{v}_{j} \mapsto \vec{v}_{j + 1}, j \in \Z$, and extending linearly.

\begin{Example}
If $V$ is a doubly infinite unitary diagonal matrix with unimodular eigenvalues $(\xi_{j})_{j = -\infty}^{\infty}$, i.e., 
$V = [\cdots|\xi_{-1} \vec{e}_{-1}|\xi_{0} \vec{e}_0|\xi_{1} \vec{e}_1|\cdots]$, then certainly $V^{t} = V$. Thus, the unitary shifts
$$W (\xi_{j} \vec{e}_j) = \xi_{j + 1} \vec{e}_{j + 1}, \quad j \in \Z,$$
belong to the conjugate orbit of the bilateral shift.
\end{Example}

\begin{Example}
One might also consider doubly infinite block diagonal matrices $V$ with diagonal blocks $B_{j}$, that is, 
$$V = \operatorname{diag}(\ldots, B_{-2}, B_{-1}, B_{0}, B_{1}, B_{2}, \ldots),$$
where each $B_k \in \mathscr{U}_{n_k}$, $n_k  \in \{1, 2, \ldots\}$, and $B_{k}^{t} = B_{k}$.  Observe that $V$ is unitary and $V^{t} = V$. Thus the unitary operator
$$W \vec{v}_j = \vec{v}_{j + 1}, \quad i \in \Z,$$
where $\vec{v}_j$ is the $j$\textsuperscript{th} column of $V$, belongs to the conjugate orbit of the bilateral shift.
As a specific example, take the above  blocks $B$ to be  Householder matrices $B= I  - 2 (\vec{v} \otimes \vec{v})$, where $\vec{v}$ is a unit vector in some $\C^{n_k}$. These $B$ are known to be unitary and self transpose. Of course, the $B$ appearing as the blocks are not necessarily the same. 
\end{Example}

\begin{Example}
From Proposition \ref{SdfsdfdsfcvvvvVV} and Proposition \ref{e0roiekfgdFF} (and the discussion following), both $M$ and $M^{*}$ belong to $\mathfrak{O}_c(M)$. How are these two operators realized as unitary shifts $\vec{v}_{j} \mapsto \vec{v}_{j + 1}$, where $$V =  [\cdots|\vec{v}_{-2} |\vec{v}_{-1} |\vec{v}_0 | \vec{v}_1| \cdots]$$ is a self transpose unitary matrix? Let $V$ be the doubly infinite matrix whose $j$\textsuperscript{th} column is $\vec{e}_{-j}$, in other words, the doubly infinite Hankel matrix with $1$ along the cross diagonal passing through the $(0, 0)$ position (and zeros elsewhere), and $C$ is the conjugation on $\ell^2(\Z)$ defined by $C \vec{e}_n = \vec{e}_{-n}$ (and extended antilinearly). Then
\begin{align*}
(C M C) \vec{v}_{j} & =  CMC(V \vec{e}_j)\\
& = C M C \vec{e}_{-j}\\
& = C M \vec{e}_{j}\\
& = C \vec{e}_{j + 1}\\
& = \vec{e}_{-j - 1}\\
& = V \vec{e}_{j + 1}\\
& = \vec{v}_{j + 1}.
\end{align*}
Now observe that 
\begin{align*}
C M C \vec{e}_{j} & = C M \vec{e}_{-j}\\
& = C  \vec{e}_{-j + 1}\\
& = \vec{e}_{j - 1}\\
& = M^{*} \vec{e}_{j}.
\end{align*}
In other words, $M^{*}$ can be realized as $\vec{v}_{j} \mapsto \vec{v}_{j + 1}$ and thus belongs to $\mathfrak{O}_c(M)$ (which of course we know already). In a similar way, if $C$ is the conjugation on $\ell^2(\Z)$ defined by $C \vec{e}_{j} = \vec{e}_{j}$ (and extended antilinearly) and $V$ is the doubly infinite identity matrix (which is obviously self transpose), one can check that $(C M C) \vec{v}_{j} = \vec{v}_{j + 1} = M \vec{e}_j$ and $C M C \vec{e}_{j} = \vec{e}_{j + 1}$. In other words, $M$ can be realized as $\vec{v}_{j} \mapsto \vec{v}_{j + 1}$ and thus belongs to $\mathfrak{O}_c(M)$ (which, again, we already know). 
\end{Example}

What about the issue as to when two members of $\mathfrak{O}_c(M)$ are equal? Suppose
$$(V_{1} J) M (V_{1} J) = (V_{2} J) M (V_{2} J).$$
By Proposition \ref{09876tygbhn} and Remark \ref{commuttannt}, this happens precisely when $V_{1} J V_{2} J = V_{1} \overline{V_{2}}$ belongs to the commutant of $M$. The commutant of $M$ (in its matrix form) consists of all doubly infinite Toeplitz matrices (matrices which are constant along each of their diagonals) \cite[p.~358]{MR4545809}. The fact that $V_{1} \overline{V_{2}}$ is unitary, implies that  $V_{1} \overline{V_{2}}$ is a doubly infinite Toeplitz matrix $[\widehat{\phi}(i - j)]_{i, j \in \Z}$ where $\phi \in L^{\infty}(m)$ is  a unimodular. In fact, this Toeplitz matrix is the matrix representation $[M_{\phi}]$ of the multiplication operator $M_{\phi} f = \phi f$ on $L^2(m)$ with respect to the orthonormal basis $(\xi^n)_{n = -\infty}^{\infty}$. In other words,
$$[M_{\phi}] =[ \langle \xi^{i} \phi, \xi^{j}\rangle]_{i, j \in \Z} = [\widehat{\phi}(i - j)]_{i, j \in \Z},$$
where 
$$\widehat{\phi}(k) : = \langle \phi, \xi^{k}\rangle = \int_{\T} f(\xi) \bar \xi^{n} dm(\xi)$$ is the $k$\textsuperscript{th} Fourier coefficient of $\phi$. 
  Putting this all together, we have the following. 

\begin{Corollary}
Suppose that $V_{1}, V_{2}$ are two doubly infinite self transpose unitary matrices. Then the unitary shift operators on $\ell^2(\Z)$ defined by
\begin{equation}\label{bbbbbVoneVtwo}
V_1 \vec{e}_{j} \mapsto V_1 \vec{e}_{j + 1} \; \; \mbox{and} \; \; V_{2} \vec{e}_{j} \mapsto V_{2} \vec{e}_{j + 1}, \quad j \in \Z,
\end{equation}
which are generic members of the orbit of $M$, are equal if and only if
$$V_{1} = V_{2} [\widehat{\phi}(i - j)]_{i, j \in \Z}$$ for some unimodular function $\phi \in L^{\infty}(m)$.
\end{Corollary}

\begin{Remark}\label{QPLAKSJFHG}
One might be somewhat suspicious as to whether the condition that $V_{1} = V_{2} [\widehat{\phi}(i - j)]_{i, j \in \Z}$ for some unimodular function $\phi$ on $\T$ might only hold in a trivial way. To show there is some depth here, let us denote the doubly infinite Toeplitz matrix by  $T = [\widehat{\phi}(i - j)]_{i, j \in \Z}$ (which will be unitary since we are assuming that $\phi$ is unimodular). The conditions $V_1 =  V_2 T$, along with  $V_{1}^{t} = V_{1}$ and $V_{2}^{t} = V_{2}$ will yield $T = \overline{V_2} T^{t}   V_{2}$. So if $V_{2}$ is taken to be the (doubly infinite)  identity matrix $I$ and $V_{1} = T$, we see that $T = T^{t}$. It follows from the structure of a Toeplitz matrix  (constant entries along each diagonal) that $\widehat{\phi}(n) = \widehat{\phi}(-n)$ for all $n \in \Z$. This can indeed happen, for example,  if we define the unimodular function $\phi$ on $\T$ as  $\phi(e^{i \theta}) = 1$ for $-\pi/2 \leq \theta \leq \pi/2$ and $\phi(e^{i \theta})= -1$ when $\pi/2 < \theta < 3 \pi/2$. In this case a calculation shows that 
$$\widehat{\phi}(n)  = -\frac{2 (-1)^n \sin \left(\frac{\pi  n}{2}\right)}{\pi  n}$$
for $n \in \Z \setminus \{0\}$ and $0$ when $n = 0$, 
which indeed has the property that $\widehat{\phi}(n) = \widehat{\phi}(-n)$ for all $n \in \Z$. From \eqref{bbbbbVoneVtwo} these yield the operators 
$\vec{e}_{j} \mapsto \vec{e}_{j + 1}$ (coming from $V_2 = I$) -- which is just the bilateral shift $M$, while the one coming from $V_1$, i.e., $V_{1} \vec{e}_{j} \mapsto V_{1} \vec{e}_{j + 1}$, which thinking of $V_{1}$ as the multiplication operator $M_{\phi} f= \phi f$ on $L^2(m)$, corresponds to the operator $\phi z^j \mapsto \phi z^{j + 1}$. However, using the fact that $\phi$ is unimodular, this is precisely the operator $z^{j} \mapsto z^{j + 1}$, which is the bilateral shift again. Of course one could keep $V_2 = I$ and set $V_1$ to be the doubly infinite Toeplitz matrix corresponding to {\em any} unimodular $\phi$ which satisfies $\phi(e^{i \theta}) = \phi(e^{-i \theta})$ (which will force the condition $\widehat{\phi}(n) = \widehat{\phi}(-n)$ for all $n$). This will yield a class of operators in the orbit of $M$ that are precisely $M$. There are plenty of pairs $V_1$ and $V_2$ of doubly infinite self transpose unitary matrices such that $\overline{V_2} V_1$ is not a Toeplitz matrix. These will correspond to distinct members of the orbit of $M$. 
\end{Remark}

\section{Multiplication operators in the conjugate orbit}

If one wanted to {\em exclude} certain unitary operators on $L^2(m)$ from consideration as members of the conjugate orbit of the bilateral shift $M$, one could make the observation that both $M$ and $M^{*}$ are cyclic ($M$ is certainly $*$-cyclic and so by Bram's theorem, see \cite{MR68129} or \cite[p.~232]{MR1112128}, it must be cyclic, likewise with $M^{*}$). Using Proposition \ref{uuueeeqq}, and the facts that $M$ and $M^{*}$ belong to $\mathfrak{O}_c(M)$, we see that any operator  in the conjugate orbit must also be cyclic. It is also known that the spectrum of $M$, as well as $M^{*}$, is $\T$ \cite[p.~178]{MR4545809} (spectrum of $M_{\phi}$, $\phi \in L^{\infty}(m)$, is its essential range). Again by Proposition \ref{uuueeeqq}, anything in the conjugate orbit must  also have spectrum equal to $\T$. For example, the (unitary) multiplication operator $M_{\phi}$, where $\phi$ is equal to one on the top half of $\T$ and minus one on the bottom half, is not contained in the conjugate orbit of the bilateral shift since its (essential) range is $\{-1, 1\}$ and not $\T$.

 By the cyclicity discussion from the previous paragraph,  it is also the case that in order for $M_{\phi}$ to belong to $\mathfrak{O}_c(M)$,  the essential range of $\phi$ must be $\T$ {\em and} $\phi$ must be injective on a set of full measure on $\T$ \cite{MR361898, MR3292513}.  Thus, for example, if $\phi$ is inner and has multiplicity greater then one (an infinite Blaschke product or maybe a singular inner function for example), then, even though the essential range of these $\phi$ is $\T$, it is certainly  far from being injective on a set of full measure and thus $M_{\phi} \not \in \mathfrak{O}_c(M)$. We will explore such multiplication operators below. The definitive characterization as to when $M_{\phi}$, for general $\phi \in L^{\infty}(m)$ (and unimodular),  is unitarily equivalent to $M^{*}$ depends on the ``multiplicity function'' for $\phi$ and is a delicate measure theory matter \cite{MR526530, MR320797} that we will not get into here. Once the unitary equivalence of $M_{\phi}$ with $M^{*}$ is established, there is also the issue as to when, via Proposition \ref{sdfhjsd9f9ds9dsf9sdf9s9df}, the intertwining operator yielding the unitary equivalence satisfies certain additional properties. Certainly  $M = M_{\xi}$ and $M^{*} = M_{\bar \xi}$ belong to the conjugate orbit of $M$. Are there {\em other} multiplication operators in the conjugate orbit of $M$?  We begin our exploration of this with several examples.

\begin{Example}\label{66yTTxx)99}
On the positive side of the ledger, consider a (Lebesgue) measure preserving  $\alpha: \T \to \T$ such that $\alpha \circ \alpha = \operatorname{id}$. Examples include $\alpha(\xi) = \xi, \alpha(\xi) = - \xi, \alpha(\xi) = \bar \xi$, or, more generally, $\alpha$ is a reflection of $\T$ across any line through the origin. One can check that $(C_{\alpha} f)(\xi) = \bar {f} \circ  \alpha$ defines a conjugation on $L^2(m)$ and a short computation shows that $C_{\alpha} M C_{\alpha} f = M_{\bar \alpha} f$ for all $f \in L^2(m)$. 
Thus, there are multiplication operators, besides $M$ and $M^{*}$, namely $M_{\bar \alpha}$,  that belong to $\mathfrak{O}_c(M)$ but they seem to come from a quite restricted class. 
\end{Example} 

\begin{Example}\label{examplecalculus}
To demonstrate the restrictions on multiplication operators $M_{\phi}$ that belong to the conjugate orbit of $M$, besides the restriction that $\phi$ must be injective almost everywhere and have essential range equal to $\T$, consider any smooth function $\psi: [0, 2 \pi] \to [0, 2 \pi]$ with $\psi(0) = 0, \psi(2 \pi) = 2 \pi$ and $\psi'(t) > 0$ for all $t$. By the inverse function theorem, this implies that $\psi$ is invertible and smooth with $(\psi^{-1})' > 0$. If $\phi(e^{i t}) = e^{i \psi(t)}$, then $\phi$ is unimodular on $\T$ and so the multiplication operator $M_{\phi}$ on $L^2(m)$ is unitary. Also note that $\phi$ is injective and $\phi(\T) = \T$. When does $M_{\phi}$ belong to $\mathfrak{O}_c(M)$? The result here is that $\psi(t) = t$ is the only possibility  and so $M_{\phi}$ is just the bilateral shift  $M$. To see why  this is true, observe that if $M_{\phi} \in \mathfrak{O}_c(M)$, Proposition \ref{sdfhjsd9f9ds9dsf9sdf9s9df} says  there must be a unitary operator $W$ on $L^2(m)$ such that $M_{\phi} W = W M^{*}$ and there must also be a conjugation $C$ on $L^2(m)$ for which $C M C = M^{*}$ and $C W C = W^{*}$. A result from \cite{MR4169409} says that any $C$ for which $C M C = M^{*}$ takes the form $C = C_{u}$, where $u$ is a unimodular function on $\T$ and $C_{u} f = u \bar{f}$. 

If $(X f)(\xi) = f(\bar{\xi})$, then, as discussed earlier,  $X$ is a unitary operator on $L^2(m)$ with $X^{*} = X$ and $X M X^{*} = M^{*}$. 
By a change of variables, note that 
$$Q f := (f \circ  \phi) \sqrt{|\phi'|} =(f \circ  \phi) \sqrt{|\bar \phi'|} $$ defines a unitary operator on $L^2(m)$ such that $M_{\phi} (Q X^{*}) = (Q X^{*}) M^{*}$. 
The two unitary operators $W$ and $Q X^{*}$ both intertwine $M_{\phi}$ and $M^{*}$ and so $$(W^{*} Q X^{*}) M^{*} = M^{*} (W^{*} Q X^{*}).$$  In other words, by taking adjoints of the previous identity, $X Q^{*} W$  is unitary and belongs to the commutant of $M$. Thus, there is a unimodular $g \in L^{\infty}(m)$ for which $X Q^{*} W = M_{g}$. From here it follows  that 
$W = Q X^{*} M_g = Q X M_{g}$ (recall that $X = X^{*}$) and thus  
$$W f = (g \circ  \bar \phi) (f \circ \bar \phi) \sqrt{|\bar \phi'|}.$$
Moreover, a calculation using the fact that $W$ is unitary and thus $W^{*}  = W^{-1}$,  yields
$$W^{*} f = \frac{f \circ \bar\phi^{-1}}{g  \sqrt{|\bar\phi' \circ \bar \phi^{-1}|}}.$$
So, from our discussion above, it boils down to finding a unimodular function $u$ for which $C_{u} W C_{u} = W^{*}$. Applying both sides of this identity to a  given $f \in L^2(m)$, we see that
$$u \cdot (\bar g \circ \bar \phi) (\bar u \circ \bar \phi ) (f \circ \bar \phi) \sqrt{|\bar \phi'|} =  \frac{f \circ \bar \phi^{-1}}{g  \sqrt{|\bar \phi' \circ \bar \phi^{-1}|}}.$$ Setting  $f$ to be the constant function equal to one on $\T$ yields 
$$u \cdot (\bar g \circ \bar \phi) (\bar u \circ \bar \phi)  \sqrt{|\bar \phi'|} = \frac{1}{g  \sqrt{|\bar \phi' \circ \bar \phi^{-1}|}}.$$
Taking absolute values and using the fact that  the functions $g$ and $u$ are unimodular, we conclude that 
$$\sqrt{|\phi'|} \sqrt{|\bar \phi' \circ \bar \phi^{-1}|} = 1.$$
The above says that 
$\psi' \cdot (\psi' \circ \psi^{-1}) = 1$. An argument with inverse functions and the chain rule will show that $\psi'(t) = (\psi^{-1}(t))'$ for all $0 \leq t \leq 2\pi$. From the fact that $\psi(0) = \psi^{-1}(0) = 0$, it follows that $\psi(t) = \psi^{-1}(t)$. 
Finally observe that if $a, b \in (0, 2 \pi)$ with $a < b$ but $\psi(a) = b$ and $\psi(b) = a$, then by Lagrange's mean value theorem, there is a $c \in (a, b)$ 
$$\psi'(c) = \frac{\psi(a) - \psi(b)}{a - b} = -1.$$
However, we are assuming that $\psi' > 0$ on $[0, 2 \pi]$ and thus  conclude that $\psi(t) = t$ for all $t \in [0, 2 \pi]$.

We end this example with two comments.  The first is that $M_{\phi}$, for the symbol $\phi$ defined at the beginning of this example, is indeed unitary equivalent to $M$ (their multiplicity functions are equal -- see \cite{MR526530, MR320797}) and so $M_{\phi}$ belongs to $\mathfrak{O}(M)$ but not $\mathfrak{O}_c(M)$ unless $\phi(e^{i t}) = e^{i t}$. Second, we note that $(C f)(\xi) = \bar{f}(-\xi)$ defines  a conjugation on $L^2(m)$ such that $C M C = -M^{*}$ and $(C f)(\xi) = \bar{f}(-\overline{\xi})$ defines a conjugation such that $C M C = -M$. Thus, $-M$ and $-M^{*}$ belong to $\mathfrak{O}_c(M)$. For a general unitary operator $U$, it is usually not the case that $-U$ or $-U^{*}$ belong to the conjugate orbit of $U$ due to spectral restrictions (recall Corollary \ref{Sdfsdfdsf11112}). 
\end{Example}

\begin{Example}
The previous example can be extended to any smooth $\phi(e^{i \theta}) = e^{i \psi(\theta)}$, where $\psi: [\eta, \eta + 2 \pi] \to [\eta, \eta + 2 \pi]$, $\psi' > 0$, $\psi(\eta) = \eta$, and $\psi(\eta + 2 \pi) = \eta + 2 \pi$. In this case. $M_{\phi}$ belongs to $\mathfrak{O}_c(M)$ only if $\psi(\theta) = \theta + \eta$. This observation will exclude symbols such as 
$$\phi(\xi) = \frac{a - \xi}{1 - \bar a \xi}, \quad a \in \D \setminus \{0\},$$
from consideration as $M_{\phi} \in \mathfrak{O}_c(M)$. This is because 
$\psi(\theta) = \arg \phi(e^{i \theta})$, which will never be of the form $\psi(\theta) = \theta + \eta$ (recall that $a \not = 0$). 
\end{Example}

\begin{Example}
Does  $M_{\phi}$ belong to $\mathfrak{O}_c(M)$ when  $\phi(e^{i t}) = e^{i \psi(t)}$ and $\psi$ is a strictly increasing continuous function on $[0, 2 \pi]$ with  $\psi(0) = 0$ and $\psi(2 \pi) = 2 \pi$, like in Example \ref{examplecalculus}, but with the added twist that  $\psi'(t) = 0$ for almost every $t$? Recall that such Cantor-like functions exist \cite{MR197639}.  {\em Prima facie}, $M_{\phi}$ is not immediately eliminated from consideration for membership in the conjugate orbit of $M$, since it is injective and has range equal to  $\T$. However, it does not belong to $\mathfrak{O}_c(M)$ since it is not unitarily equivalent to $M^{*}$ (recall Corollary \ref{Sdfsdfdsf11112}). This is because  $\psi' = 0$ almost everywhere and thus $\psi$  has no absolutely continuous part. This implies that the spectral measure corresponding to $M_{\phi}$ is singular, where it certainly needs to be absolutely continuous for $M_{\phi}$ to be unitarily equivalent to $M^{*}$. 
\end{Example}

These examples all lead us to the possible conclusion that $M_{\phi} \in \mathfrak{O}_c(M)$ only under special circumstances. We confirm this observation with the following theorem.

\begin{Theorem}\label{thmmainmult}
If $\phi$ is a unimodular Borel function on $\T$, then $M_{\phi} \in \mathfrak{O}_c(M)$ if and only if the following three conditions are satisfied:
\begin{enumerate}
\item $\bar \phi \circ  \bar \phi = \operatorname{id}$ almost everywhere on $\T$;
\item $m \circ \bar \phi^{-1}$ is mutually absolutely continuous with respect to $m$;
\item If 
${\displaystyle h = \frac{d (m \circ \bar \phi^{-1})}{dm},}$
then $h \cdot (h \circ \bar \phi) = 1$ almost everywhere on $\T$.
\end{enumerate}
\end{Theorem}

\begin{proof}
Suppose that $\phi$ is a unomodular Borel function  with $M_{\phi} \in \mathfrak{O}_c(M)$. From the discussion above, we see that  $\bar \phi$ is invertible on a set of full measure on $\T$. The measure $m \circ \bar \phi^{-1}$ is a well defined Borel measure on $\T$ and the operator
$$Y: L^2(m) \to L^2(m \circ \bar \phi^{-1}), \quad Y f = f \circ  \bar \phi^{-1},$$ is an isometric isomorphism. Indeed, there is the well-known change of variables formula
\begin{equation}\label{kkh@YYY}
\int_{\T} |Y f|^2 d(m \circ \bar \phi^{-1}) = \int_{\T} |f|^2 d m, \quad f \in L^2(m).
\end{equation}
We direct the reader to \cite[p.~162-163]{MR33869} for the technical details of the Borel measure theory discussion above. 
Moreover, $(Y M_{\bar \phi} w)(\xi) = \xi (Y w)(\xi)$, $w \in L^2(m)$. This says that the spectral measure for $M_{\bar \phi}$ on $L^2(m)$ is $m \circ \phi^{-1}$, in other words, $M_{\bar \phi}$ on $L^2(m)$ is unitarily equivalent to $M_{\xi}$ on $L^2(m \circ \bar \phi^{-1})$. By the spectral theorem (the uniqueness part), the two measures $m$ and $m \circ \bar \phi^{-1}$ on $\T$ are mutually absolutely continuous and thus condition $(b)$ is satisfied. Therefore,  there is a nonnegative $h \in L^1(m)$ (with $1/h \in L^1(m)$) such that 
$$d (m \circ \bar \phi^{-1}) = h\,dm.$$ Thus,  the operator 
\begin{equation}\label{kkk9900888GG}
Q: L^2(m) \to L^2(m), \quad Q f  = (f \circ \bar \phi) \sqrt{h},
\end{equation}
 is unitary since it is isometric via 
\begin{align*}
\int_{\T} |Q f|^2 dm & = \int_{\T} |f \circ \bar \phi|^2 h dm\\
& = \int_{\T} |f \circ \bar \phi|^2 d(m \circ \bar \phi^{-1})\\
& = \int_{\T} |f|^2 dm.
\end{align*}
Moreover, if $g \in L^2(m)$, then, similar to the computation in \eqref{kkh@YYY}, 
$$f := \frac{g \circ \bar \phi^{-1}}{\sqrt{h \circ \bar \phi^{-1}}}$$
belongs to $L^2(m)$ and $Q f = g$.

Since we are assuming that $M_{\phi} \in \mathfrak{O}_c(M)$, Proposition \ref{sdfhjsd9f9ds9dsf9sdf9s9df} says there must be a unitary operator $W$ on $L^2(m)$ such 
$$J W J = W^{*} \; \; \mbox{and} \; \; M_{\phi} = W M^{*} W^*,$$
where $J f = \bar f$ on $L^2(m)$. Recall that $J M J = M^{*}$. The first condition in the previous line says that 
\begin{equation}\label{BBBvvvCCC}
\overline{W(\bar f)} = W^{*} f, \quad f \in L^2(m),
\end{equation}
while the second condition yields 
$$W^{*}(\phi f)(\xi) = \bar \xi (W^{*} f)(\xi).$$
Using \eqref{BBBvvvCCC} and taking complex conjugates gives us 
$$W(\bar \phi \bar f)= \xi (W \bar f), \quad f \in L^2(m),$$
and setting $g = \bar f$ says that 
\begin{equation}\label{kkkBv12229}
W g = \bar \xi W(\bar \phi g), \quad g \in L^2(m).
\end{equation}
Iterate this previous identity as
\begin{align*}
W g  & = \bar \xi W(\bar \phi g)\\
& = \bar \xi^2 W(\bar \phi^2 g)\\
& = \bar \xi^3 W(\bar \phi^3 g)\\
 &  \; \; \vdots \\
& = \bar \xi^n W(\bar \phi^n g), \quad n \geq 0.
\end{align*}
Replace $g$ with $\phi^n g$ in the previous equation, and recall that $\phi$ is unimodular,  to see that 
$$W(\phi^n g) = \bar \xi^n W(\bar \phi^n \phi^n g)= \bar \xi^n W g, \quad n \geq 0.$$
Combining these two facts give us the identity 
$$W(\bar \phi^m g) = \xi^m W g, \quad m \in \Z.$$
With $g = \sqrt{h}$, which belongs to $L^2(m)$,  the above says that 
$$W((p \circ \bar \phi) \sqrt{h}) = p \cdot ( W(\sqrt{h})$$ for all trigonometric polynomials $p$, that is,  finite linear combinations of the monomials $\{\xi^{m}: m \in \Z\}$.  Using the unitary $Q$ from \eqref{kkk9900888GG}, we can rewrite the identity from the previous line as
$$W Q p  = p \cdot W(\sqrt{h})$$
for all trigonometric polynomials $p$, and hence all $f \in L^2(m)$. Since $W$ and $Q$ are both unitary on $L^2(m)$, the above, via the density of the trigonometric polynomials in $L^2(m)$, says that  the multiplication operator $f   \mapsto f \cdot W(\sqrt{h})$ is unitary on $L^2(m)$ and thus 
$$ k := W(\sqrt{h})$$ must be a unimodular function on $\T$. 

Inverting the formula in \eqref{kkk9900888GG} reveals that 
$$Q^{-1} f  =  \frac{f \circ \bar \phi^{-1}}{\sqrt{h \circ\bar  \phi^{-1}}}$$
and thus
\begin{equation}\label{BBBBB}
W f  = (Q^{-1} f) \cdot k =  \frac{f \circ \bar \phi^{-1}}{\sqrt{h \circ \bar \phi^{-1}}}k.
\end{equation}
From the previous line it follows that 
\begin{equation}\label{11111}
W^{*} f  = W^{-1} f = (f \circ \bar \phi) \cdot (\overline{k \circ \bar \phi}) \cdot \sqrt{h}.
\end{equation}
Now bring in \eqref{BBBvvvCCC} to get
\begin{equation}\label{22222}
W^{*} f  = \overline{W\bar f} =  \frac{f \circ  \bar \phi^{-1}}{\sqrt{h \circ \bar  \phi^{-1}}} \bar k
\end{equation}
and thus, equating \eqref{11111} and \eqref{22222}, yields
\begin{equation}\label{AAAAAA}
 (f \circ \bar \phi) \cdot (\overline{k \circ \bar \phi}) \cdot \sqrt{h} =  \frac{f \circ  \bar \phi^{-1}}{\sqrt{h \circ \bar  \phi^{-1}}} \bar k.
\end{equation}
Plug in the constant function $f \equiv 1$ into both sides of \eqref{AAAAAA} to get 
$$ \overline{k \circ \bar \phi} \cdot  \sqrt{h} = \frac{1}{\sqrt{h \circ \bar \phi^{-1}}} \bar k.$$
Recall that $k$ is a unomudular function and $h$ is nonnegative and so, taking the modulus of both sides of the above, we see that 
$$\sqrt{h} = \frac{1}{\sqrt{h \circ \bar \phi^{-1}}},$$
which verifies condition $(c)$, and also gives us the identity
$$  \overline{k \circ \bar \phi}=  \bar k.$$
Plugging these last two facts into \eqref{AAAAAA} yields 
$$f \circ \bar \phi = f \circ \bar \phi^{-1}, \quad f \in L^2(m),$$
and setting $f(\xi) = \xi$ in the above, we see that 
$\bar \phi = \bar \phi^{-1}.$
This verifies condition $(a)$.

Conversely, suppose that conditions $(a), (b), (c)$ are satisfied. One can check that  the mapping 
$$C: L^2(m) \to L^2(m), \quad C f = \sqrt{h} \cdot (\bar f \circ \bar \phi),$$
satisfies the three properties of a conjugation (antilinear, isometric, and involutive). Note that  $C$ is antilinear. Furthermore,  
\begin{align*}
\int_{\T} |C f|^2 dm & = \int_{\T} h |\bar f \circ \bar \phi|^2 dm\\
& = \int_{\T} h |f \circ \bar \phi|^2 dm\\
& = \int_{\T} |Q f|^2 dm\\
& = \int_{\T} |f|^2 dm
\end{align*}
for each $f \in L^2(m)$. 
Thus $C$ maps $L^2(m)$ to $L^2(m)$ isometrically.  Moreover, for each $f \in L^2(m)$,
\begin{align*}
C^2 f & = C (\sqrt{h} \cdot (\bar f \circ \bar \phi))\\
& = \sqrt{h} \cdot (\sqrt{h} \circ \bar \phi) \cdot (f \circ \bar \phi \circ \bar \phi)\\
& = \sqrt{h} \cdot (\sqrt{h} \circ \bar \phi )\cdot f\\
& = f. 
\end{align*}
Thus, $C$ is involutive and so $C$ is a conjugation on $L^2(m)$. 

Finally, observe that 
\begin{align*}
C M C f & = C M  (\sqrt{h} \cdot (\bar f \circ \bar\phi))\\
& = C (\xi \cdot \sqrt{h} \cdot (\bar f \circ \bar \phi))\\
& = (\bar  \xi \circ \bar \phi) \sqrt{h} \cdot (\sqrt{h} \circ \bar \phi) (f \circ \bar \phi \circ \bar \phi)\\
& = \phi f,
\end{align*}
and thus $M_{\phi} \in \mathfrak{O}_c(M)$. 
\end{proof}

We remark that in Example \ref{66yTTxx)99}, the mapping $\alpha$ was assumed to be measure preserving and convolutive and thus $h \equiv 1$ in the previous theorem. 

\begin{Example}
Here we construct a specific example of a unimodular and Borel measurable $\phi$ on $\T$ satisfying conditions $(a), (b), (c)$ of Theorem \ref{thmmainmult} for which $h$ is not identically one and thus $\phi$ is not measure preserving. Let $\omega:[0,\pi] \to [0,\pi]$ be a smooth strictly increasing function with $\omega(0)=0$ and $\omega(\pi)=\pi$. A simple example is $\omega(t) = t^2/\pi$. Now define $\varphi: \mathbb{T} \to \mathbb{T}$ by
\[
\varphi(e^{it}) = \left\{
\begin{array}{ccc}
e^{-i\omega(t)} & \mbox{if} & 0 \leq t \leq \pi,\\
& & \\
e^{i\omega^{-1}(-t)} & \mbox{if} & -\pi \leq t \leq 0.\\
\end{array} \right.
\]
Observe that $\varphi$ is a continuous bijection on $\mathbb{T}$,  $\varphi(1)=1$, $\varphi(-1)=-1$, and $\varphi$ flips the upper and lower half circles. The result here is the following: $m \circ \varphi^{-1}$ us mutually absolutely continuous with respect to $m$ and
\[
h(e^{i\theta}) := \frac{d (m \circ \varphi^{-1})}{dm}(e^{i\theta}) = 
\left\{
\begin{array}{ccc}
\omega'(\theta), & \mbox{if} & 0 < \theta < \pi,\\
& & \\
{\displaystyle \frac{1}{\omega'(\omega^{-1}(-\theta))}}, & \mbox{if} & -\pi < \theta < 0.\\
\end{array} \right.
\]
Here is the proof. 
Fix $\zeta = e^{i\theta}$ in the upper half circle, i.e., $0<\theta<\pi$, and, for a fixed $\varepsilon > 0$, set
\[
I_\varepsilon := \{ e^{it} : \theta-\varepsilon < t < \theta+\varepsilon \}.
\]
Then 
\[
\varphi^{-1}(I_\varepsilon) = \{ e^{is} : -\omega(\theta+\varepsilon) < s < -\omega(\theta-\varepsilon) \}.
\]
Hence, by definition,
\[
(m \circ \varphi^{-1})(I_\varepsilon) = m(\varphi^{-1}(I_\varepsilon)) = \frac{\omega(\theta+\varepsilon)-\omega(\theta-\varepsilon)}{2\pi}.
\]
Therefore, by Besicovitch's theorem \cite{MR14414}, the Radon--Nikodym derivative can be computed as a limit by 
\begin{align*}
\frac{d(m \circ \varphi^{-1})}{dm}(e^{i\theta}) &= \lim_{\varepsilon \to 0} \frac{(m \circ \varphi^{-1})(I_\varepsilon) }{m(I_\varepsilon)}\\
&= \lim_{\varepsilon \to 0} \frac{\omega(\theta+\varepsilon)-\omega(\theta-\varepsilon)}{2\varepsilon}\\
&= \omega'(\theta), \qquad 0 < \theta < \pi.
\end{align*}

In a similar way, if $\zeta = e^{i\theta}$ in the lower half circle, i.e., $\pi<\theta<0$, we consider
\[
I_\varepsilon = \{ e^{it} : \theta-\varepsilon < t < \theta+\varepsilon \}.
\]
Then
\[
\varphi^{-1}(I_\varepsilon) = \{ e^{is} : \omega^{-1}(-\theta-\varepsilon) < s < \omega^{-1}(-\theta+\varepsilon) \}.
\]
Hence, by definition,
\[
(m \circ \varphi^{-1})(I_\varepsilon) = m(\varphi^{-1}(I_\varepsilon)) = \frac{\omega^{-1}(\theta+\varepsilon)-\omega^{-1}(\theta-\varepsilon)}{2\pi}.
\]
Therefore, again by Besicovitch's theorem, 
\begin{align*}
\frac{d(m \circ \varphi^{-1})}{dm}(e^{i\theta}) &= \lim_{\varepsilon \to 0} \frac{(m \circ \varphi^{-1})(I_\varepsilon) }{m(I_\varepsilon)}\\
&= \lim_{\varepsilon \to 0} \frac{\omega^{-1}(-\theta+\varepsilon)-\omega^{-1}(-\theta-\varepsilon)}{2\varepsilon}\\
&= \big(\omega^{-1}\big)'(-\theta)\\
& = \frac{1}{\omega'(\omega^{-1}(-\theta))}, \qquad -\pi < \theta < 0,
\end{align*}
which completes the proof. 

Note that, for $0<\theta<\pi$,
\[
h(\varphi(e^{i\theta})) = h(e^{-i\omega(\theta)}) = \frac{1}{\omega'(\omega^{-1}(\omega(\theta)))} = \frac{1}{\omega'(\theta)} = \frac{1}{h(e^{i\theta})},
\]
and similarly, for $-\pi<\theta<0$,
\[
h(\varphi(e^{i\theta})) = h(e^{i\omega^{-1}(-\theta)}) =  \omega'(\omega^{-1}(-\theta)) = \frac{1}{h(e^{i\theta})}.
\]
Hence, for almost all $\theta$,
\[
h(e^{i\theta}) \cdot h(\varphi(e^{i\theta})) = 1.
\]

In the special case $\omega(t) = t^2/\pi$, we obtain
\[
\frac{d(m \circ \varphi^{-1})}{dm}(e^{i\theta}) =
\left\{
\begin{array}{ccc}
{\displaystyle \frac{2\theta}{\pi}}, & \mbox{if} & 0 < \theta < \pi,\\
& & \\
{\displaystyle \frac{\sqrt{\pi}}{2\sqrt{-\theta}}}, & \mbox{if} & -\pi < \theta < 0.\\
\end{array} \right.
\]
\end{Example}

\begin{Remark}
The focus on this chapter is on the bilateral shift $M$ on $L^2(m)$. However, one can also fashion a version of Theorem \ref{thmmainmult} for $M_{\xi}$ on $L^2(\mu)$, where $\mu$ is a finite positive Borel measure on $\T$. Here one can give necessary and sufficient conditions as to when $M_{\phi}$, $\phi \in L^{\infty}(\mu)$, belongs to the conjugate orbit on $M_{\xi}$ on $L^2(\mu)$. The statement and proof is mainly the same except $m$ is replaced by $\mu$. 
\end{Remark}

\begin{Example}\label{Ealphabetatsst}
It is worth mentioning a member of $\mathfrak{O}_c(M)$ that combines some of the  previous examples. Let $\alpha: \T \to \T$ be  (Lebesgue) measure preserving with $\alpha \circ \alpha = \operatorname{id}$, $\beta: \T \to \{-1, 1\}$ (Lebesgue measurable) such that $\beta \circ \alpha = \beta$, and $s, t \in \R$ with $s^2 + t^2 = 1$. A specific example of such an $\alpha$ and $\beta$ is  $\alpha(\xi) = \bar \xi$  for all $\xi \in \T$ and $\beta(\xi) = 1$ when $- \pi/2 \leq \arg \xi  \leq \pi/2$ and $\beta(\xi) = -1$ when $\pi/2 < \arg \xi < 3\pi/2$.

Fixing the above parameters $\alpha, \beta, s, t$,  define the mapping $C$ on $L^2(m)$ by 
$$C f = s (\bar f \circ \alpha) + i t \beta \bar f$$ and note that $C$ is antilinear. Moreover, since $s$ and $t$ are real, 
\begin{align*}
C^2 f & = C(s (\bar f \circ \alpha) + i t \beta \bar f)\\
& = s C(\bar f \circ \alpha) - i t C(\beta \bar f)\\
& = s \big(s  (f \circ (\alpha \circ \alpha))+ i t \beta  (f \circ \alpha)\big)- i t \big(s (\beta \circ \alpha) (f \circ \alpha) + i t \beta^2 f\big)\\
& = s^2 f + i s t \beta (f \circ\alpha) - i s t \beta (f \circ \alpha) + t^2 f\\
& = (s^2 + t^2) f\\
& = f,
\end{align*}
and so $C$ is involutive. Furthermore, 
\begin{align*}
\|C f\|^2 & = \|s (\bar f \circ \alpha) + i t \beta \bar f\|^2\\
& = s^2 \|f \circ \alpha\|^2 + \big\langle s (\bar f \circ \alpha), i t \beta \bar f\big\rangle + \big\langle i t \beta \bar f, s (\bar f \circ \alpha)\big\rangle + t^2 \|f\|^2\\
& = s^2 \|f\|^2 + t^2 \|f\|^2 + i s t\Big(\int_{\T} \beta \bar f (f \circ \alpha)dm - \int_{\T}\beta (\bar f \circ \alpha) f dm\Big).
\end{align*}
Using the facts that $\beta \circ \alpha = \beta$, $\alpha \circ \alpha = \operatorname{id}$, and $\alpha$ is measure preserving, we see that 
$$\int_{\T} \beta \bar f (f \circ \alpha)dm - \int_{\T}  \beta (\bar f \circ \alpha) f dm = 0$$ and so 
$$\|C f\|^2 = (s^2 + t^2) \|f\|^2 = \|f\|^2.$$ This makes $C$ isometric and hence a conjugation on $L^2(m)$. 
From here we conclude that $C M C$ is a member of $\mathfrak{O}_c(M)$ and has the formula 
\begin{align*}
C M C f &= C M (s (\bar f \circ \alpha) + i \beta \bar f)\\
& = C (s M (\bar f \circ \alpha) + i t M (\beta \bar f))\\
& = s C (M (\bar f \circ \alpha) ) - i t C(M (\beta \bar f))\\
& = s \Big(s \bar \alpha f + i t \beta M^{*}(f \circ \alpha)\Big) - i t \Big(s (\beta \circ \alpha) \bar \alpha (f \circ \alpha) + i t \beta^2 M^{*} f\Big)\\
& = s^2 \bar \alpha f + t^2 M^{*} f + i s t (\beta M^{*} f + \beta \bar \alpha (f \circ \alpha))\\
& = (s^2 M_{\bar \alpha} + t^2 M^{*}) f + i s t \beta (M_{\bar \alpha} - M^{*}) (f \circ \alpha).
\end{align*}
\end{Example}

\begin{Remark}\label{morethanone}
 If $U$ is unitary with $\sigma(U) = \{\lambda\}$, then, as observed earlier in Remark \ref{singlepoint},  $\mathfrak{O}_c(U) = \{\overline{\lambda} I\}$ is just one operator. If $\sigma(U)$ has more than one point, then the conjugate orbit is quite rich. Indeed, one of the many versions of the spectral theorem \cite[p.~304]{ConwayFA} says that a general unitary operator $U$ is unitarity equivalent to 
$$\bigoplus_{n \geq 1} (M_{\xi}, L^2(\mu|_{\Delta_n})),$$
where $M_{\xi} f = \xi f$ on $L^2(\mu|_{\Delta_n})$, $\mu$ is a positive finite measure on $\C$, $\Delta_1 = \sigma(U)$, and $(\Delta_n)_{n \geq 1}$ is a decreasing sequence of Borel subsets of $\sigma(U)$. The assumption that $\sigma(U)$ has at least two distinct elements, shows that not all of the summands $(M_{\xi}, L^2(\mu|_{\Delta_n}))$ are trivial. Moreover, by the previous examples (Example \ref{MxionL63} in particular), the conjugate orbit of  each of the $M_{\xi}$ on $L^2(\mu|_{\Delta_n})$ (when it is nontrivial)  is quite abundant. We will see another analysis of the complexity of the conjugate orbit in the next section. 
\end{Remark}

\section{Shifts of higher multiplicity}

One can generalize the previous section and consider, for a fixed $n \in \N \cup \{\infty\}$, the unitary operator 
$$M^{(n)} := \bigoplus_{j = 1}^{n} M$$
on the Hilbert space  $\ell^{2}(\Z)^{(n)}$, where 
$$\ell^{2}(\Z)^{(n)} := \Big\{(\vec{a}_k)_{k = 1}^{n}: \|(\vec{a}_k)_{k = 1}^{n}\|_{\ell^2(\Z)^{(n)}} :=  \Big(\sum_{k = 1}^{n}\|\vec{a}_k\|^{2}_{\ell^2(\Z)}\Big)^{\frac{1}{2}} < \infty\Big\}.$$
Recall the definition of the sequence space $\ell^2(\Z)$ from \eqref{sdfg7sd7fsdf7sdfaaaQQQQ}.
 Note that $M^{(n)}$ acts on a vector $(\vec{a}_k)_{k = 1}^{n}$ in $\ell^{2}(\Z)^{(n)}$ by 
$$M^{(n)} \big((\vec{a}_k)_{k = 1}^{n} \big) := (M \vec{a}_k)_{k = 1}^{n}$$ and is called the {\em bilateral shift of multiplicity $n$.} 
One often thinks of $M^{(n)}$ in its block matrix form  acting on column vectors 
$$\begin{bmatrix}\vec{a}_1 \\ \vec{a}_2 \\ \vec{a}_{3} \\ \vdots\end{bmatrix} \in \ell^2(\Z)^{(n)}$$
as 
$$M^{(n)} = \begin{bmatrix} 
M & 0 & 0 & 0 & 0 & \cdots \\
0  & M & 0 & 0 & 0 & \cdots\\
0 & 0 & M & 0 & 0 & \cdots\\
0 & 0 & 0 & M & 0 & \cdots\\
0 & 0 & 0 & 0 & M & \cdots\\[-4pt]
0 & 0 & 0 & 0 & 0 & \ddots\\
\end{bmatrix}.$$
In this section we make some remarks concerning  the conjugate orbit of $M^{(n)}$. A skeptic might wonder if we are generalizing the discussion from the previous section just for the sake of generalization. Through the following examples, let us make the case that some ``naturally appearing'' unitary operators are actually unitarily equivalent to bilateral shifts and thus fall under the purview of the current discussion. 

\begin{Example}
For a bounded analytic function $u$ on $\D$, well-known theory of function spaces \cite{Duren} says that 
$$u(\xi) := \lim_{r \to 1^{-}} u(r \xi)$$ exists for almost every $\xi \in \T$. 
A bounded analytic function $u$ is {\em inner} if the above boundary function $u(\xi)$ is unimodular almost everywhere. 
As examined in \cite{MR4866296}, consider an inner function $u$ with degree $n$ (either the number of Blaschke factors if $u$ is a finite Blaschke product or $\infty$ if $u$ contains an infinite Blaschke factor or a singular inner factor in its factorization). We refer the reader to \cite{Duren} for an explanation of the above terms (Blaschke product, singular inner function, etc.). The multiplication operator  $M_{u} f = u f$ is unitary on $L^2(m)$ and moreover, 
$$L^{2}(m) = \bigoplus_{j = -\infty}^{\infty} u^j \mathcal{K}_{u},$$
where $\mathcal{K}_{u} := H^2 \ominus u H^2$ is the {\em model space} corresponding to the inner function $u$.
Here $H^2$ is the closed subspace of $L^2(m)$ consisting of the $L^2(m)$ functions whose negative Fourier coefficients vanish. 
 It is known that the dimension of $\mathcal{K}_{u}$ is equal to the degree of $u$ \cite[p.~117]{MR3526203}.  In the parlance of operator theory, the model space $\mathcal{K}_{u}$ is a {\em wandering subspace} for the multiplication operator $M_u$ on $L^2(m)$. It follows, see Proposition \ref{sd9f9sdi8f7777yyyTTROOOOO} below, that $M_{u}$ is unitarily equivalent to $M^{(n)}$. So, for example, if $u$ is a Blaschke product of order two, then $M_{u}$ is unitarily equivalent to $M^{(2)}$, while if $u$ is an infinite Blaschke product or a singular inner function, 
then $M_{u}$ is unitarily equivalent to $M^{(\infty)}$. As a specific example where the wandering subspace is finite dimensional, suppose $u(\xi) = \xi^2$ (an example of a Blaschke product of order two). Then $\mathcal{K}_{u} = H^2 \ominus u H^2$ is equal to the two dimensional subspace $\operatorname{span}\{1, \xi\}$. Thus, 
$$L^2(m) = \bigoplus_{j = -\infty}^{\infty} \xi^{2j} \operatorname{span}\{1, \xi\}.$$
From here we can see that $M_{u}$ on $L^2(m)$ is unitarily equivalent to $M \oplus M$ on $L^2(m) \oplus L^2(m)$. 
\end{Example}

\begin{Example}
The translation operator $(U f)(x) = f(x - 1)$  on $L^2(\R)$  is unitary and if $\mathcal{M} = \chi_{[0, 1]} L^2(\R)$, then
$$L^2(\R) = \bigoplus_{j = -\infty}^{\infty} U^j \mathcal{M}.$$
Analogous to the previous example, $\mathcal{M}$ is an infinite dimensional wandering subspace for $U$. It follows, again see Proposition \ref{sd9f9sdi8f7777yyyTTROOOOO} below,  that $U$ is unitarily equivalent to $M^{(\infty)}$. 
\end{Example}

\begin{Example}
The dilation operator  $(U f)(x) = \sqrt{2} f(2x)$ on $L^2(\R)$ is unitary. If 
$$\psi(x) = \begin{cases}
\phantom{-}1 & \mbox{for $0 \leq x < \frac{1}{2}$},\\
-1 & \mbox{for $\frac{1}{2} \leq x \leq 1$},\\
\phantom{-}0 & \mbox{otherwise},
\end{cases}$$
(Haar) wavelet theory \cite{MR1008470} (see also \cite{MR1117215}) says that the family of  functions 
$$\psi_{n, k}(x) := 2^{\frac{n}{2}} \psi(2^{n} x - k), \quad n, k \in \Z,$$
form an orthonormal basis for $L^2(\R)$. 
Consider the subspaces 
$$\mathcal{W}_{\ell} := \overline{\operatorname{span}}\{\psi_{\ell, k}: k \in \Z\}, \quad \ell \in \Z.$$
Note that  $\mathcal{W}_{\ell} \perp \mathcal{W}_{\ell'}$ for all  distinct $\ell, \ell' \in \Z$ 
and $U \mathcal{W}_{\ell} = \mathcal{W}_{\ell + 1}$ for all $\ell \in \Z$. Thus,
$$L^2(\R) = \bigoplus_{j = -\infty}^{\infty} U^{j} \mathcal{W}_{0}.$$
Hence,  $\mathcal{W}_{0}$ is an infinite dimensional wandering subspace for $U$ and 
so $U$ is unitarily equivalent to $M^{(\infty)}$.
\end{Example}

The driver for connecting these above examples with a bilateral shift of higher multiplicity comes from the following general result. 

\begin{Proposition}\label{sd9f9sdi8f7777yyyTTROOOOO}
Suppose that $\mathcal{H}$ is a separable Hilbert space and $\mathcal{M}$ is an $n$ dimensional subspace of $\mathcal{H}$, where $n \in \N \cup \{\infty\}$. If $U$ is a unitary operator on $\mathcal{H}$ satisfying 
\begin{equation}\label{wanderingU}
\mathcal{H} = \bigoplus_{j  = -\infty}^{\infty} U^{j} \mathcal{M},
\end{equation}
then $U$ is unitarily equivalent to $M^{(n)}$. 
\end{Proposition}

\begin{proof}
We follow the proof from \cite{MR4866296}. Let $(\vec{u}_k)_{k = 1}^{n}$, $1 \leq n \leq \infty$,  be an orthonormal basis for $\mathcal{M}$. By our assumption on $\mathcal{H}$ in \eqref{wanderingU}, each $\vec{x} \in \mathcal{H}$ can be written as 
$$\vec{x} = \sum_{j = -\infty}^{\infty} U^{j} \vec{x}_j, \; \; \mbox{where} \; \; \vec{x}_{j} = \sum_{k = 1}^{n} a_{j, k} \vec{u}_k \in \mathcal{M}.$$
Thus, 
$$\vec{x} = \sum_{j = -\infty}^{\infty} \Big(\sum_{k = 1}^{n} a_{j, k} U^{j} \vec{u}_{k}\Big).$$
Now observe that since $U^{j} \vec{x}_j \perp U^{j'} \vec{x}_{j'}$ (and $U$ is unitary) we have 
\begin{align*}
\|\vec{x}\|^2 & = \sum_{j = -\infty}^{\infty} \|\vec{x}_{j}\|^2\\
& = \sum_{j = -\infty}^{\infty} \Big\|\sum_{k = 1}^{\infty} a_{j, k} \vec{u}_j \Big\|^2\\
& = \sum_{j = -\infty}^{\infty} \sum_{k = 1}^{\infty} |a_{j, k}|^2.
\end{align*}
Thus, for each $1 \leq k \leq n$, the vector $\vec{a}_{k} = (a_{j, k})_{j = -\infty}^{\infty}$ belongs to $\ell^{2}(\Z)$ and 
the map 
$$\Gamma: \mathcal{H} \to \ell^2(\Z)^{(n)}, \quad 
\Gamma(\vec{x}) = (\vec{a}_{k})_{k = 1}^{n},$$ is an isometric isomorphism that intertwines $U$ with $M^{(n)}$. Indeed, for any $\vec{x} \in \mathcal{H}$ we have  
\begin{align*}
\Gamma U \vec{x} & = \Gamma U\Big(\sum_{j = -\infty}^{\infty} \sum_{k = 1}^{n} a_{j, k} U^{j} \vec{u}_k\Big)\\
& = \Gamma \Big(\sum_{j = -\infty}^{\infty} \sum_{k = 1}^{n} a_{j, k} U^{j + 1} \vec{u}_k\Big)\\
& = \Gamma \Big(\sum_{j = -\infty}^{\infty} \sum_{k = 1}^{n} a_{j - 1, k} U^{j} \vec{u}_k\Big)\\
& = ((a_{j - 1, k})_{j = -\infty}^{\infty})_{k = 1}^{n}\\
& = (M \vec{a}_{k})_{k = 1}^{n}\\
& = M^{(n)} \Gamma \vec{x},
\end{align*}
which establishes the unitary equivalence of $U$ and $M^{(n)}$. 
\end{proof}

What type of operators belong to $\mathfrak{O}_c(M^{(n)})$? This seems to be quite difficult and we do not have an adequate description yet. But we can present some rich classes of examples.

\begin{Example}Certainly, the following class of unitary operators on $\ell^2(\Z)^{(n)}$ belong to the conjugate orbit of $M^{(n)}$: 
$$W := \begin{bmatrix} 
W_1 & 0 & 0 & 0 & 0 & \cdots \\
0  & W_2 & 0 & 0 & 0 & \cdots\\
0 & 0 & W_3 & 0 & 0 & \cdots\\
0 & 0 & 0 & W_4 & 0 & \cdots\\
0 & 0 & 0 & 0 & W_5 & \cdots\\[-4pt]
0 & 0 & 0 & 0 & 0 & \ddots
\end{bmatrix},$$
where each $W_{k}$ is a unitary shift on $\ell^2(\Z)$ formed from a doubly infinite self transpose unitary matrix
$$V_{k} = [\cdots|\vec{v}_{k, -2}|\vec{v}_{k, -1}|\vec{v}_{k, 0}|\vec{v}_{k, 1}|\vec{v}_{k, 2}|\cdots]$$
with  
$$W_{k} \vec{v}_{k, j} = \vec{v}_{k, j + 1}, \quad j \in \Z.$$
One can see this by using Theorem \ref{0934t4rhhHHCccvvJ} to find conjugations $C_{k}$, $1 \leq k \leq n$, on $\ell^2(\Z)$ such that $C_{k} M C_{k} = W_{k}$ for the given unitary shifts $W_k$ as above. Then 
\begin{equation}\label{diagConj}
C := \begin{bmatrix} 
C_1 & 0 & 0 & 0 & 0 & \cdots \\
0  & C_2 & 0 & 0 & 0 & \cdots\\
0 & 0 & C_3 & 0 & 0 & \cdots\\
0 & 0 & 0 & C_4 & 0 & \cdots\\
0 & 0 & 0 & 0 & C_5 & \cdots\\[-4pt]
0 & 0 & 0 & 0 & 0 & \ddots
\end{bmatrix},
\end{equation}
satisfies  $$C ((\vec{a}_{k})_{k = 1}^{n}) = (C_{k} \vec{a}_{k})_{k = 1}^{n}$$
and defines a conjugation on $\ell^{2}(\Z)^{(n)}$. Furthermore, since  $M^{(n)} = \bigoplus_{k = 1}^{n} M$, we conclude that  
$C M^{(n)} C = W$. 
\end{Example}

When $n = 1$, Theorem \ref{0934t4rhhHHCccvvJ} says that the conjugate orbit $\mathfrak{O}_c(M^{(1)})$ consists {\em precisely}  of these unitary shifts. When $n > 1$, there are many {\em other }operators in $\mathfrak{O}_c(M^{(n)})$. The main reason is that not every conjugation on $\ell^{2}(\Z)^{(n)}$ can be diagonalized as in \eqref{diagConj}. Indeed, the paper \cite{Ko03062019} shows how one can, at least in the $n = 2$ case, for a general Hilbert space $\mathcal{H}$, create conjugations $C$ on $\mathcal{H} \oplus \mathcal{H}$ as 
$$C = \begin{bmatrix} C_{1} & C_{2}\\ C_{3} & C_{4}\end{bmatrix},$$
where each $C_j$ above is an antilinear operator on $\mathcal{H}$  (not necessarily a conjugation) which satisfies 
$$C_{1} = C_{1}^{\star}, \quad C_{4} = C_{4}^{\star},  \quad C_{3} = C_{2}^{\star}.$$
$$C_{1} C_{1}^{\star} + C_{2} C_{2}^{\star} = I, \quad C_{2} C_{2}^{\star} + C_{4} C_{4}^{\star} = I,$$
$$C_{2}^{\star} C_{1} + C_{4} C_{2}^{\star} = 0.$$
In the above, $C_{j}^{\star}$ denotes the antilinear adjoint (to be formally defined below in \eqref{oooOOkkkKKKK}) as the antilinear operator on $\mathcal{H}$ satisfying 
$$\langle C_{j} \vec{x}, \vec{y}\rangle = \overline{\langle \vec{x}, C_{j}^{\star} \vec{y}\rangle} \; \; \mbox{for all} \; \;  \vec{x}, \vec{y} \in \mathcal{H}.$$ The above matrix definition of a conjugation  is understood as 
$$C \begin{bmatrix} \vec{a} \\ \vec{b} \end{bmatrix}  =  \begin{bmatrix} C_{1} & C_{2}\\ C_{3} & C_{4}\end{bmatrix}  \begin{bmatrix} \vec{a} \\ \vec{b} \end{bmatrix} =  \begin{bmatrix}C_{1} \vec{a}  + C_{2} \vec{b}\\ C_{3} \vec{a} + C_{4} \vec{b} \end{bmatrix}, \; \; \mbox{where} \; \; \begin{bmatrix} \vec{a} \\ \vec{b} \end{bmatrix}  \in \mathcal{H} \oplus \mathcal{H}.$$

Certainly the conjugation arising from \eqref{diagConj}, with $C_{1}, C_{4}$ conjugations and $C_{2}  = C_{3} = 0$, define conjugations on $\ell^{2}(\Z)$. However, there are many others and thus $\mathfrak{O}_c(M^{(2)})$ consists of far more unitary operators than just $W_{1} \oplus W_2$, with $W_{1}, W_{2}$  unitary shifts on $\ell^2(\Z)$. As one can imagine, the situation only becomes more complicated as $n$ gets larger. 

\begin{Example}
Considering the bilateral shift as the multiplication operator $(M f)(\xi) = \xi f(\xi)$ on $L^2(m)$ there is another rich class of members of the conjugate orbit of $M^{(n)}$. For each  $1 \leq j \leq n$, let $\alpha_j$ be a (Lebesgue) measure preserving map of $\T$ with $\alpha_j \circ \alpha_j = \operatorname{id}$. As before, define the conjugation $C_{\alpha_j} = \bar f \circ \alpha_j$ on $L^2(m)$. With $C$ defined as the block diagonal conjugation 
$$C = \operatorname{diag}(C_{\alpha_1}, C_{\alpha_2}, C_{\alpha_3}, \ldots)$$ on $L^2(m)^{(n)}$, we see from our calculations in Example \ref{66yTTxx)99} that 
$$C M^{(n)} C = \operatorname{diag}(M_{\bar \alpha_1}, M_{\bar \alpha_2}, M_{\bar \alpha_3}, \ldots).$$
\end{Example}

We point out a  final comment concerning the complexity of the conjugate orbit of a general unitary operator that also provides a glimmer of hope that one can obtain get some partial understanding of the complexity involved.  A  result from \cite{MR4866296}, which uses some of the finer renditions of the spectral theorem,  says that for a given unitary operator $U$ on $\mathcal{H}$, there are invariant subspaces $\mathcal{K}_{1}, \mathcal{K}_{2}, \mathcal{K}_{3}$ for $U$ such that $\mathcal{H} = \mathcal{K}_{1} \bigoplus \mathcal{K}_{2} \bigoplus \mathcal{K}_{3}$. Moreover $U|_{\mathcal{K}_1}$ is a bilateral shift in the sense of Proposition \ref{sd9f9sdi8f7777yyyTTROOOOO}, $U_{\mathcal{K}_2}$ is diagonalizable in the sense of Corollary \ref{C27}, $U|_{\mathcal{K}_{2} \oplus \mathcal{K}_{3}}$ contains no bilateral shift part of $U$, and $\mathcal{K}_{3}$ contains no eigenvectors of $U$. In fact, the operator $U|_{\mathcal{K}_{3}}$ is sometimes known as the ``chaotic part'' of $U$ and is not completely understood. Our previous work shows the complexity of the conjugate orbits of each piece $U|_{\mathcal{K}_1}$ and $U|_{\mathcal{K}_{2}}$ and so we have some organizational tools to help us understand the possibilities of (some of)  the operators in the conjugate orbit of $U$. Also worth pointing out is a older rendition of the discussion above that decomposes a unitary operator into a direct sum of a bilateral shift and a pure part (a unitary where every invariant subspace is reducing) \cite[p.~62-63]{MR270172}. 

\chapter{Real and Complex Hilbert Spaces from a Unitary Operator}

The remainder of this paper will develop a matrix model for a generic unitary operator that involves real Hilbert spaces. We do this in order to obtain a tangible model for its conjugate orbit. Since this topic, which goes under the general heading of {\em complexification},  might be unfamiliar to the most likely audience for this paper, we will be quite generous with the exposition, details, and  examples. 

\section{A real Hilbert space from a unitary operator}

If $U$ is a unitary operator on a separable Hilbert $\mathcal{H}$ with spectral measure $E(\cdot)$, in other words
\begin{equation}\label{spectralrep}
U = \int_{\T} \xi dE(\xi)
\end{equation}
\cite[p.~269]{ConwayFA},
 a rendition of the spectral theorem yields an orthonormal sequence of vectors $(\vec{x}_k)_{k \geq 1}$ in $\mathcal{H}$ such that
\begin{equation}\label{Hhhkkk}
\mathcal{H} = \bigoplus_{k \geq 1} \mathcal{H}(\vec{x}_k),
\end{equation} where
$$\mathcal{H}(\vec{x}_k) := \overline{\operatorname{span}_{\C}}\{E(\Omega) \vec{x}_k: \Omega \in \mathcal{B}\}.$$
In the above, $\mathcal{B}$ denotes the Borel subsets of $\T$ and $\overline{\operatorname{span}_{\C}}$ denotes the closed {\em complex} linear span in $\mathcal{H}$. We will be a bit pedantic here with the subscript $\C$ above since at some point below we will  examine the {\em real} linear span of a set of vectors (and denote it by $\operatorname{span}_{\R}$ -- see \eqref{ahhhhhkkkkxxxrrr} below).
One can produce such a sequence $(\vec{x}_k)_{k \geq 1}$ satisfying \eqref{Hhhkkk}  as follows (see \cite[p.~300]{ConwayFA}). For a fixed unit vector $\vec{x} \in \mathcal{H}$ let
$$\mathcal{H}(\vec{x}) := \overline{\mathcal{W}^{*}(U) \vec{x}},$$
where $\mathcal{W}^{*}(U)$ is the von Neumann algebra generated by $U$ (the weak or, what turns out to be the same, the  strong operator closure of $p(U)$, where $p$ is a trigonometric polynomial). Since $U|_{\mathcal{H}(\vec{x})}$ is a $\ast$-cyclic unitary operator in that
$$\overline{\operatorname{span}_{\C}}\{ U^{n} \vec{x}: n \in \Z\} = \mathcal{H}(\vec{x}),$$ 
note that $U^{-n} = U^{*n}$, 
  basic facts about spectral measures, namely the important fact that $E(\Omega) \in \mathcal{W}^{*}(U)$ for every $\Omega \in \mathcal{B}$ \cite[p.~292]{ConwayFA},  will show that
\begin{equation}\label{88hhHCCNNccc12}
\mathcal{H}(\vec{x}) := \overline{\operatorname{span}_{\C}}\{E(\Omega) \vec{x}: \Omega \in \mathcal{B}\}.
\end{equation}
A Zorn's lemma argument produces a sequence $(\vec{x}_k)_{k \geq 1}$ in $\mathcal{H}$ such that
$\mathcal{H}(\vec{x}_j) \perp \mathcal{H}(\vec{x}_{k})$ for $j \not = k$, and, by  the maximality of the sequence $(\vec{x}_k)_{k \geq 1}$, we obtain \eqref{Hhhkkk}.

\begin{Example}\label{9erufgjdefght}
If $U = I$, the identity operator on a separable Hilbert space $\mathcal{H}$, choose any orthonormal basis $(\vec{x}_k)_{k \geq 1}$ for $\mathcal{H}$ and note that
$\mathcal{H}(\vec{x}_{k}) = \overline{\mathcal{W}^{*}(I) \vec{x}_k}  = \C \vec{x}_k$
and, trivially, $\mathcal{H} = \bigoplus_{k \geq 1} \C \vec{x}_k.$
\end{Example}

\begin{Example}\label{nnBBvv665}
Suppose $U = \operatorname{diag}(\xi_1, \xi_2, \xi_3, \ldots)$ is a diagonal unitary operator with $\xi_j \not = \xi_k$ for all $j, k \geq 1$ acting on $\ell^2  = \ell^2(\N)$ via $U \vec{e}_{j} = \xi_j \vec{e}_j$. Then any  unit vector $\vec{x} = (x_j)_{j \geq 1} \in \ell^2$, with $x_j \in \C \setminus \{0\}$ for all $j$, is a $\ast$-cyclic vector for $U$ and so $\mathcal{H}(\vec{x}) = \ell^2$.
\end{Example}

\begin{Remark}
It is important to emphasize here that the decomposition in \eqref{Hhhkkk} is not unique or even canonical. For instance, in the previous example, we can set $\vec{x}_1 = \vec{e}_1$ and $\vec{x}_2 = (0, x_2, x_3, x_4, \ldots) \in \ell^2$, $x_j \not = 0$. Then $\ell^2 = \mathcal{H}(\vec{x}_1) \oplus \mathcal{H}(\vec{x}_2)$.
\end{Remark}

\begin{Example}
If $U = \operatorname{diag}(\xi_1, \ldots, \xi_n)$ is an $n \times n$ diagonal unitary matrix with $\xi_j \not = \xi_k$, $1 \leq j, k \leq n$ then, as with Example \ref{nnBBvv665} with $\vec{x} = (1, 1, \ldots, 1)$, we see that  $\mathcal{H}(\vec{x}) = \C^{n}$. 
\end{Example}

\begin{Example}\label{bliexample9}
When $(U f)(\xi) = \xi f(\xi)$ is the bilateral shift on $L^2(m)$, the density of the trigonometric polynomials in $L^2(m)$ and the fact that $(U^{*} f)(\xi) = \bar \xi f(\xi)$, imply that $\chi : = \chi_{\T}$, the characteristic function on $\T$, is a $\ast$-cyclic vector for $U$ and so
$L^2(m) = \mathcal{H}(\chi) = \overline{\operatorname{span}_{\C}}\{U^n \chi: n \in \Z\}$.
\end{Example}

\begin{Example}
Suppose that $U$ is a unitary operator on $\mathcal{H}$ and $(\xi_j)_{j \geq 1}$ are {\em distinct} unimodular constants with
$\mathcal{N}_j := \ker (U  - \xi_j I) \not=  \{\vec{0}\}$ and
$\mathcal{H} = \bigoplus_{j \geq 1} \mathcal{N}_j.$
Of course the eigenspaces of a unitary operator are automatically mutually orthogonal. The above assumption is just a matter of the eigenspaces spanning  $\mathcal{H}$. Since $U|_{\mathcal{N}_j} = \xi_j I|_{\mathcal{N}_j}$ we can apply Example \ref{9erufgjdefght} to each $U|_{\mathcal{N}_j}$ to write
$\mathcal{N}_{j} = \bigoplus_{k \geq 1} \mathcal{H}(\vec{x}_k),$ where the number of summands, equivalently, the number of orthonormal basis elements $\vec{x}_k$ for $\mathcal{N}_j$,  will be the dimension of $\mathcal{N}_j$. Now reindex to
write $\mathcal{H} = \bigoplus_{n \geq 1} \mathcal{H}(\vec{x}_n)$.
\end{Example}

\begin{Example}
One could revisit Example \ref{nnBBvv665} and combine it with the analysis from the previous example to see that
$$\mathcal{N}_j = \ker(U - \xi_j I) = \C \vec{e}_j = \overline{\operatorname{span}_{\C}}\{U^{n} \vec{e}_j: n \in \Z\} =  \mathcal{H}(\vec{e}_j)$$ and
$\ell^2 = \bigoplus_{j \geq 1} \mathcal{H}(\vec{e}_j)$.
\end{Example}

\begin{Example}\label{FTHkkk}
Suppose that  $U = \mathcal{F}$ is the Fourier-Plancherel transform on $L^2(\R)$ from \eqref{FTFTFT}. As discussed in Example \ref{nknjdhfkhfghHJHJHJ},
$$L^2(\R) = \ker (\mathcal{F} - I) \oplus \ker (\mathcal{F} + I) \oplus \ker (\mathcal{F} - i I) \oplus \ker(\mathcal{F} + i I).$$
From \eqref{Fteigenv-aces} recall that 
\begin{align*}
\ker (\mathcal{F} - I) & = \overline{\operatorname{span}_{\C}}\{H_{4n}: n \geq 0\},\\
\ker (\mathcal{F} + i I) & =   \overline{\operatorname{span}_{\C}}\{H_{4n + 1}: n \geq 0\},\\
\ker (\mathcal{F} + I) & =   \overline{\operatorname{span}_{\C}}\{H_{4n + 2}: n \geq 0\},\\
\ker (\mathcal{F} - i I) & =   \overline{\operatorname{span}_{\C}}\{H_{4n + 3}: n \geq 0\}.
\end{align*}
Since $(H_k)_{k \geq 0}$ is already an orthonormal basis of eigenvectors  for $\mathcal{F}$, then
$L^2(\R) = \bigoplus_{k \geq 0} \mathcal{H}(H_k)$, where $\mathcal{H}(H_k) = \C H_k$.
\end{Example}

\begin{Example}\label{HTksdfsf}
Suppose that $U = \mathscr{H}$ is the Hilbert transform on $L^2(\R)$ from \eqref{Hilberttransformmmm}. As with the previous example, using the orthonormal basis of eigenvectors $(f_n)_{n \in \Z}$ from Example \ref{ldfgjdfjg888uuUU}, we can write $L^2(\R) = \bigoplus_{n \in \Z} \mathcal{H}(f_n)$, where $\mathcal{H}(f_n) = \C f_n$.
\end{Example}

With these specific examples in mind, let us return to the general case and our decomposition from \eqref{Hhhkkk}. For each $k \geq 1$, let
\begin{equation}\label{ahhhhhkkkkxxxrrr}
\mathcal{H}(\vec{x}_k)_{\R} := \overline{\operatorname{span}_{\R}}\{E(\Omega) \vec{x}_{k}: \Omega \in \mathcal{B}\}.
\end{equation}
 This is the closed {\em real} linear span of the vectors $\{E(\Omega) \vec{x}_k: \Omega \in \mathcal{B}\}$.
Notice that 
$$\mathcal{H}(\vec{x}_k) = \overline{\operatorname{span}_{\C}} \{E(\Omega) \vec{x}_k: \Omega \in \mathcal{B}\},$$
defined earlier in \eqref{88hhHCCNNccc12}, is the closed 
the {\em complex} linear span. Take note of the subscripts $\R$ and $\C$ in the above. 

  Observe that for any $a_1, a_2 \in \R$ and $\Omega_1, \Omega_2 \in \mathcal{B}$ we can use standard facts about the orthogonal projections $E(\Omega)$ (they are selfadjoint and idempotent) to see that 
\begin{align*}
\langle a_1 E(\Omega_1) \vec{x}_k, a_2 E(\Omega_2) \vec{x}_k\rangle & = a_1 a_2 \langle E(\Omega_1) \vec{x}_k, E(\Omega_2) \vec{x}_k\rangle\\
& = a_1 a_2 \langle E(\Omega_2) E(\Omega_1) \vec{x}_k, \vec{x}_k\rangle\\
& = a_1 a_2 \langle E(\Omega_1 \cap \Omega_2) \vec{x}_k, \vec{x}_k\rangle\\
& = a_1 a_2  \langle E(\Omega_1 \cap\Omega_2) \vec{x}_k,  E(\Omega_1 \cap \Omega_2)\vec{x}_k\rangle\\
& = a_1 a_2 \|E(\Omega_1 \cap \Omega_2) \vec{x}_k\|^2\\
& \in \R.
\end{align*}
Thus, each $\mathcal{H}(\vec{x}_k)_{\R}$ is a {\em real} Hilbert space (a vector space with $\R$ as the base field) with {\em real} inner product. In other words, 
\begin{equation}\label{reallll9}
\langle \vec{x}, \vec{y}\rangle \in \R \; \; \mbox{for every} \; \; \vec{x}, \vec{y} \in \mathcal{H}(\vec{x}_k)_{\R}.
\end{equation}

Now define
\begin{equation}\label{HrrRR}
\mathcal{H}_{\R} := \bigoplus_{k \geq 1} \mathcal{H}(\vec{x}_k)_{\R}
\end{equation}
 and notice that since $\mathcal{H}(\vec{x}_j)_{\R} \perp \mathcal{H}(\vec{x}_k)_{\R}$ when $j \not = k$,  we conclude that $\mathcal{H}_{\R}$ is a real Hilbert space with real inner product
 \begin{equation}\label{55ttreaaallll}
\langle \vec{x}, \vec{y}\rangle =\sum_{\ell \geq 1} a_{\ell} b_{\ell},
\end{equation}
where $(\vec{u}_{\ell})_{\ell \geq 1}$ is an orthonormal basis for $\mathcal{H}$ with 
\begin{equation}\label{dfgjdfg88810ooO}
\vec{u}_{\ell} \in  \overline{\operatorname{span}_{\R}}\{E(\Omega) \vec{x}_{k}: \Omega \in \mathcal{B}\}
\end{equation} for some $k$ (this is possible by \eqref{ahhhhhkkkkxxxrrr}) and, of course,
$$\vec{x} = \sum_{\ell \geq 1} a_{\ell} \vec{u}_{\ell}, \quad \vec{y} = \sum_{\ell \geq 1} b_{\ell} \vec{u}_{\ell}, \quad a_{\ell}, b_{\ell} \in \R.$$

\begin{Example}\label{bsdlfjnbvdccvbVSDf}
Let us continue our discussion from  Example \ref{bliexample9}, where the unitary operator $(U f)(\xi) = \xi f(\xi)$ is the bilateral shift on $L^2(m)$. Here
$L^2(m) = \overline{\operatorname{span}_{\C}} \{\chi_{\Omega}: \Omega \in \mathcal{B}\}$ and so 
$$L^2(m)_{\R} = \overline{\operatorname{span}_{\R}}\{\chi_{\Omega}: \Omega \in \mathcal{B}\} = L^{2}_{\R}(m),$$
the real-valued functions from $L^2(m)$.
\end{Example}

\begin{Example}\label{werfg6yhbbBNJKJHGJKIUJH}
Continuing with Example \ref{nnBBvv665}, suppose that 
$U = \operatorname{diag}(\xi_1, \xi_2, \xi_3, \ldots)$ is a diagonal unitary operator on $\ell^2$ with $\xi_j \not = \xi_k$ for all $j \not = k$. If $ \vec{x} = (x_j)_{j \geq 1} \in \ell^2$ with $x_j \in \R \setminus \{0\}$ for all $j$, then, $\mathcal{H}(\vec{x}) = \ell^2$ and thus 
$$\ell^{2}_{\R} = \{\vec{a} = (a_n)_{n \geq 1} \in \ell^2: \mbox{$a_j \in \R \; $ for all $j \geq 1$}\},$$
in other words, the real sequences in $\ell^2$. If $\vec{x} \in \ell^2$ were chosen so that the entries of $\vec{x}$ were all nonzero, {\em but not necessarily real}, then even though $\mathcal{H}(\vec{x}) = \ell^2$, the space $\mathcal{H}(\vec{x})_{\R}$ would be the set of vectors $\vec{y} = (y_j)_{j \geq 1} \in \ell^2$ such that $y_j = a_j x_j$ ($a_j \in \R$), which will be a different space than $\ell^2_{\R}$ since sequences in $\mathcal{H}(\vec{x})_{\R}$ might have some complex-valued entries. 
\end{Example}

\begin{Example}\label{exxarrrrrnnnn}
The previous example applies to unitary diagonal matrices $$U = \operatorname{diag}(\xi_1, \cdots, \xi_n), \quad   \xi_j \not = \xi_k,$$ with $\vec{x} = (1, 1, \ldots, 1),$ where one can see that $\mathcal{H}(\vec{x}) = \C^{n}$ and  $\mathcal{H}(\vec{x})_{\R} = \R^n$.
\end{Example}

\begin{Example}\label{FT666554}
Suppose that $U = \mathcal{F}$ is the Fourier--Plancherel transform on $L^2(\R)$ from \eqref{FTFTFT}. As discussed in Example \ref{FTHkkk}, with $(H_n)_{n \geq 0}$ being the normalized Hermite functions, we have $L^2(\R) = \bigoplus_{n \geq 0} \C H_n$. Since each Hermite function $H_n$ is real-valued, we have $$L^2(\R)_{\R} = \bigoplus_{n \geq 0} \R H_n$$ and thus $L^2(\R)_{\R} = L^{2}_{\R}(\R)$, the real-valued functions from $L^2(\R)$.
\end{Example}

\begin{Example}\label{HT00sdfspPPIPI}
Suppose that $U = \mathscr{H}$ is the Hilbert transform on $L^2(\R)$ from \eqref{Hilberttransformmmm}. As with Example \ref{HTksdfsf}, we see in this case that 
$$L^{2}(\R)_{\R} = \bigoplus_{n\in \Z} \R f_n.$$ Here each $f_n$ is complex valued and so, unlike the previous example with the Fourier--Plancherel transform, $L^2(\R)_{\R}$, even though it is a real vector space,  is not a space of real-valued functions.
\end{Example}

\section{A complex Hilbert space from a unitary operator}

For a fixed unitary operator $U$ on $\mathcal{H}$, we define the {\em real} vector space $\mathfrak{R}(\mathcal{H}_{\R})$ to be the collection of ordered pairs of vectors
$$\mathfrak{R}(\mathcal{H}_{\R}) := \left\{\begin{bmatrix} \vec{h}_r\\ \vec{h}_c\end{bmatrix}: \vec{h}_r, \vec{h}_c \in \mathcal{H}_{\R}\right\}.$$
Addition of vectors and multiplication of vectors  by {\em real} scalars  in $\mathfrak{R}(\mathcal{H}_{\R})$ are defined componentwise in the obvious way as 
$$\begin{bmatrix} \vec{h}_r\\ \vec{h}_c\end{bmatrix} + \begin{bmatrix} \vec{k}_r\\ \vec{k}_c\end{bmatrix} := \begin{bmatrix} \vec{h}_r + \vec{k}_r\\ \vec{h}_c + \vec{k}_c\end{bmatrix} \; \; \mbox{and} \; \; a \begin{bmatrix} \vec{h}_r\\ \vec{h}_c\end{bmatrix} := \begin{bmatrix} a\vec{h}_r\\ a \vec{h}_c\end{bmatrix}, \quad a \in \R.$$

One can make $\mathfrak{R}(\mathcal{H}_{\R})$  a {\em complex} vector space, denoted by $\mathfrak{C}(\mathcal{H}_{\R})$, by applying a well-known procedure for complexification and  {\em defining} multiplication by a complex scalar $s + i t$ to be 
\begin{equation}\label{cdef}
(s + i t) \begin{bmatrix} \vec{h}_r\\ \vec{h}_c\end{bmatrix} := \begin{bmatrix} s \vec{h}_r - t \vec{h}_c\\ s \vec{h}_c + t \vec{h}_r\end{bmatrix}.
\end{equation}
Informally, we think of multiplying out the expression $(s + i t) (\vec{h}_r + i \vec{h}_c)$ in the standard way to get
$(s \vec{h}_r - t \vec{h}_c) + i(s \vec{h}_c + t \vec{h}_r)$ and then equating the ``real'' component $s \vec{h}_r - t \vec{h}_c$ with the first entry of the vector and the ``imaginary'' component $s \vec{h}_c + t \vec{h}_r$ with the second entry.
One can check that all the necessary axioms  of a complex vector space hold for $\mathfrak{C}(\mathcal{H}_{\R})$.

It is important to remark here that the two vector spaces $\mathfrak{R}(\mathcal{H}_{\R})$ and $\mathfrak{C}(\mathcal{H}_{\R})$ consist of the same vectors. The difference is their vector space structure. The first allows only scalar multiplication by real numbers while the second allows scalar multiplication by complex numbers.

Now we make the complex vector space $\mathfrak{C}(\mathcal{H}_{\R})$ a complex {\em Hilbert} space by defining the inner product on  $\mathfrak{C}(\mathcal{H}_{\R})$ to be
$$\Big\langle \begin{bmatrix} \vec{h}_r\\ \vec{h}_c \end{bmatrix}, \begin{bmatrix} \vec{g}_r\\\vec{g}_c\end{bmatrix} \Big\rangle_{\mathfrak{C}(\mathcal{H}_{\R})} := \langle \vec{h}_r, \vec{g}_r\rangle  + \langle \vec{h}_c, \vec{g}_c\rangle + i \big(\langle \vec{h}_c, \vec{g}_r\rangle - \langle \vec{h}_r, \vec{g}_c\rangle\big).$$
Notice that
$$\Big\langle \begin{bmatrix}\vec{h}_r\\ \vec{0}\end{bmatrix}, \begin{bmatrix} \vec{g}_r\\ \vec{0}\end{bmatrix}\Big\rangle_{\mathfrak{C}(\mathcal{H}_{\R})} = \langle \vec{h}_r, \vec{g}_r\rangle$$
and so this inner product extends the real inner product on $\mathcal{H}_{\R}$ to the complex inner product space $\mathfrak{C}(\mathcal{H}_{\R})$. Also observe that
\begin{align*}
\Big\|\begin{bmatrix} \vec{h}_r\\   \vec{h}_c\end{bmatrix}\Big\|_{\mathfrak{C}(\mathcal{H}_{\R})}^2 & = \Big\langle \begin{bmatrix}\vec{h}_r\\  \vec{h}_c\end{bmatrix}, \begin{bmatrix} \vec{h}_r\\\vec{h}_c\end{bmatrix}\Big\rangle_{\mathfrak{C}(\mathcal{H}_{\R})}\\
& = \langle \vec{h}_r, \vec{h}_r\rangle + \langle \vec{h}_c, \vec{h}_c\rangle + i \big(\langle \vec{h}_c, \vec{h}_r\rangle - \langle \vec{h}_r, \vec{h}_c\rangle\big).
\end{align*}
As mentioned earlier, $\langle \vec{h}_c, \vec{h}_r\rangle$ and $ \langle \vec{h}_r, \vec{h}_c\rangle$ are both {\em real} (since $\vec{h}_r$ and $\vec{h}_c$ belong to $\mathcal{H}_{\R}$) and so  the imaginary term above drops   out (recall \eqref{reallll9} and \eqref{55ttreaaallll}). Thus,
$$\left\|\begin{bmatrix}\vec{h}_r\\   \vec{h}_c\end{bmatrix}\right\|_{\mathfrak{C}(\mathcal{H}_{\R})}^2 = \|\vec{h}_r\|^2 + \|\vec{h}_{c}\|^2.$$

\begin{Example}\label{diagonallslslslsls}
With Example \ref{exxarrrrrnnnn}, where $U$ is an $n \times n$ diagonal unitary matrix with distinct entries, we know that 
$$\mathfrak{C}(\C^n_{\R}) = \R^n \oplus \R^n  = \left\{\begin{bmatrix} \vec{a} \\ \vec{b} \end{bmatrix}: \vec{a}, \vec{b} \in \R^n\right\}$$
endowed with the complex linear structure
$$(s + i t) \begin{bmatrix} \vec{a}\\ \vec{b}\end{bmatrix} = \begin{bmatrix} s \vec{a} - t \vec{b}\\ s \vec{b} + t \vec{a}\end{bmatrix}.$$
Moreover, 
$$\left\| \begin{bmatrix} \vec{a}\\ \vec{b}\end{bmatrix}\right\|^{2}_{\mathfrak{C}(\C^{n}_{\R})} = \|\vec{a}\|_{\R^n}^{2} + \|\vec{b}\|_{\R^n}^{2} = \|\vec{a} + i \vec{b} \|^{2}_{\C^n}$$ (since $\vec{a}$ and  $\vec{b}$ belong to $\R^n$) and so $\mathfrak{C}(\C^{n}_{\R})$ is isometrically isomorphic to $\C^n$.  More about this below.
\end{Example}

\begin{Example}\label{bilallateralll}
From Example \ref{bsdlfjnbvdccvbVSDf}, where $(U f)(\xi) = \xi f(\xi)$ is the bilateral shift on $L^2(m)$,  we have that
$$\mathfrak{C}(L^2(m)_{\R}) =  L^{2}_{\R}(m) \oplus L^{2}_{\R}(m) =  \left\{\begin{bmatrix} f\\g\end{bmatrix}: f, g \in L^{2}_{\R}(m)\right\}$$
with the complex linear structure 
$$(s + i t) \begin{bmatrix} f\\ g\end{bmatrix} = \begin{bmatrix} s f - t g\\ s g + t f\end{bmatrix}.$$
Since 
$$\left\| \begin{bmatrix} f\\ g\end{bmatrix}\right\|^{2}_{\mathfrak{C}(L^2(m)_{\R}) } = \|f\|_{L^{2}(m)}^{2} + \|g\|_{L^{2}(m)}^{2} = \|f + i g\|^{2}_{L^2(m)}$$ ($f$ and $g$ are real-valued), we see that $\mathfrak{C}(L^2(m)_{\R})$ is isometrically isomorphic to $L^2(m)$. Again, more about this below.
\end{Example}

\begin{Example}\label{FTexamplesaepaeate}
Suppose $U$ is the Fourier transform on $L^2(\R)$. As with the previous example, along with Example \ref{FT666554}, we see in this case that
$$\frak{C}(L^2(\R)_{\R}) = L^2_{\R}(\R) \oplus L^{2}_{\R}(\R)  =  \left\{\begin{bmatrix} f\\g\end{bmatrix}: f, g \in L^{2}_{\R}(\R)\right\}$$
and $\mathfrak{C}(L^2(\R)_{\R})$ is isometrically isomorphic to $L^2(\R)$. 
\end{Example}

\begin{Example}
Returning to our discussion in Example \ref{werfg6yhbbBNJKJHGJKIUJH}, suppose that 
$$U = \operatorname{diag}(\xi_1, \xi_2, \xi_3, \ldots)$$ is a diagonal unitary operator on $\ell^2$ with distinct $\xi_j$. Choosing an $\vec{x} \in \ell^2$ with {\em real} entries and such that $\mathcal{H}(\vec{x}) = \ell^2$, we see, as with the two previous examples, that
$$\mathfrak{C}(\ell^2_{\R}) =  \left\{\begin{bmatrix} \vec{a}\\\vec{b}\end{bmatrix}: \vec{a}, \vec{b} \in \ell^2_{\R}\right\}$$ and 
$\mathfrak{C}(\ell^2_{\R})$ is isometrically isomorphic to $\ell^2$. 
\end{Example}

\begin{Example}\label{HT664466yY}
Suppose $U$ is the Hilbert transform on $L^2(\R)$. In this case, see Example \ref{HT00sdfspPPIPI}, $L^2(\R)_{\R}$ is the closed real linear span of the complex-valued functions $f_n, n \in \Z$, and
$$\mathfrak{C}(L^2(\R)_{\R}) =  \left\{\begin{bmatrix} f\\g\end{bmatrix}: f, g \in L^{2}(\R)_{\R}\right\}.$$
\end{Example}

Returning to the general situation, we need to relate our original complex Hilbert space $\mathcal{H}$ with our new complex Hilbert space $\mathfrak{C}(\mathcal{H}_{\R})$. We have informally observed this in a few of the previous examples.  For the orthonormal basis $(\vec{u}_k)_{k \geq 1}$ for $\mathcal{H}$ from the previous section (with each $\vec{u}_{\ell} \in \overline{\operatorname{span}_{\R}} \{E(\Omega) \vec{x}_k: \Omega \in \mathcal{B}\}$ for some $k$), define
\begin{equation}\label{XXXXex}
X: \mathcal{H} \to \mathfrak{C}(\mathcal{H}_{\R}), \quad X\Big(\sum_{\ell \geq 1} (\alpha_{\ell} + i \beta_{\ell}) \vec{u}_{\ell}\Big) := \begin{bmatrix} \sum_{\ell \geq 1} \alpha_{\ell} \vec{u}_{\ell}\\ \sum_{\ell \geq 1} \beta_{\ell} \vec{u}_{\ell}\end{bmatrix}, \quad \alpha_{\ell}, \beta_{\ell} \in \R.
\end{equation}
The above map $X$ is invertible since if
$$\vec{h}_{r} := \sum_{\ell \geq 1} \alpha_{\ell} \vec{u}_{\ell} \; \; \mbox{and} \; \; \vec{h}_c := \sum_{\ell \geq 1} \beta_{\ell} \vec{u}_{\ell},$$ where $\alpha_{\ell}$ and $\beta_{\ell}$ are real, then
$\vec{h}_r + i \vec{h}_c$ is a generic element of $\mathcal{H}$ and
$$X^{-1}\begin{bmatrix} \vec{h}_r\\ \vec{h}_c\end{bmatrix} = \vec{h}_{r} + i \vec{h}_c.$$
We now want to show that $X$ is an isometric isomorphism between the {\em complex} Hilbert spaces $\mathcal{H}$ and $\mathfrak{C}(\mathcal{H}_{\R})$.

Clearly the map $X$ is {\em real} linear in that $X(a \vec{h} + \vec{k}) = a X \vec{h} + X \vec{k}$ when $a \in \R$ and $\vec{h}, \vec{k} \in \mathcal{H}$. To show {\em complex} linearity, observe that if $\vec{h} = \vec{h}_{r} + i \vec{h}_c$, then
\begin{align*}
X ((s + i t) \vec{h}) & = X \Big((s + i t) \sum_{\ell \geq 1} (\alpha_{\ell} + i \beta_{\ell}) \vec{u}_{\ell}\Big)\\
& = X \Big(\sum_{\ell \geq 1}  ((s \alpha_{\ell} - t \beta_{\ell}) +  i (s \beta_{\ell} + t \alpha_{\ell}))\vec{u}_{\ell}\Big)\\
& = \begin{bmatrix} \sum_{\ell \geq 1} (s \alpha_{\ell} - t \beta_{\ell}) \vec{u}_{\ell}\\\sum_{\ell \geq 1} (s \beta_{\ell} + t \alpha_{\ell}) \vec{u}_{\ell}\end{bmatrix}\\
& = \begin{bmatrix}s \sum_{\ell \geq 1} \alpha_{\ell} \vec{u}_{\ell} - t \sum_{\ell \geq 1} \beta_{\ell} \vec{u}_{\ell}\\  s \sum_{\ell \geq 1} \beta_{\ell} \vec{u}_{\ell} + t \sum_{\ell \geq 1} \alpha_{\ell} \vec{u}_{\ell}\end{bmatrix}\\
& = (s + i t) \begin{bmatrix} \sum_{\ell \geq 1} \alpha_{\ell} \vec{u}_{\ell}\\ \sum_{\ell \geq 1} \beta_{\ell} \vec{u}_{\ell}\end{bmatrix}\\
& = (s + i t) \begin{bmatrix}\vec{h}_{r}\\ \vec{h}_{c}\end{bmatrix}\\
& = (s + i t) X( \vec{h}).
\end{align*}

Next we want to show that $X$ preserves the Hilbert space structure of $\mathcal{H}$ and $\mathfrak{C}(\mathcal{H}_{\R})$. Indeed, if
$$\vec{h} = \vec{h}_{r} + i  \vec{h}_{c} \; \; \mbox{and} \; \;  \vec{g} = \vec{g}_{r} + i \vec{g}_{c},$$ then
\begin{align*}
\langle X(\vec{h}), X(\vec{g})\rangle_{\mathfrak{C}(\mathcal{H}_{\R})} & = \Big\langle  \begin{bmatrix}\vec{h}_{r}\\ \vec{h}_c\end{bmatrix}, \begin{bmatrix}\vec{g}_{r}\\\vec{g}_{c}\end{bmatrix}\Big\rangle_{\mathfrak{C}(\mathcal{H}_{\R})}\\
& = \langle \vec{h}_{r}, \vec{g}_{r}\rangle + \langle \vec{h}_{c}, \vec{g}_{c}\rangle + i (\langle \vec{h}_{c}, \vec{g}_{r}\rangle - \langle \vec{h}_{r}, \vec{g}_{c}\rangle)\\
& = \langle \vec{h}_{r}, \vec{g}_{r}\rangle + \langle \vec{h}_{c}, \vec{g}_{c}\rangle + \langle i \vec{h}_c, \vec{g}_{r}\rangle + \langle \vec{h}_{r}, i \vec{g}_{c}\rangle\\
& = \langle \vec{h}_{r} + i \vec{h}_{c}, \vec{g}_{r} + i \vec{g}_{c}\rangle\\
& = \langle \vec{h}, \vec{g}\rangle.
\end{align*}
In summary, $X$ is an isometric isomorphism between the complex Hilbert spaces $\mathcal{H}$ and $\mathfrak{C}(\mathcal{H}_{\R})$.

\chapter{A Model for the Orbit}

Continuing our discussion from the previous chapter, we now represent a unitary operator as a matrix of real linear operators on the real Hilbert spaces constructed in the previous chapter. We first need to define what we mean by the term {\em real linear} as well as the associated adjoint of a real linear operator. This discussion will yield a model for a description of the conjugate orbit of a given unitary operator. 

\section{Types of linearity}

Parts of the discussion below can be found in \cite{MR2749452}. For a complex Hilbert space $\mathcal{H}$, 
a mapping $A: \mathcal{H} \to \mathcal{H}$ is {\em $\C$-linear} (or {\em complex linear}) if
$$A (\vec{x} + \vec{y}) = A \vec{x} + A \vec{y} \; \mbox{and} \; A(\lambda \vec{x}) = \lambda A \vec{x} \;  \; \mbox{for all $\vec{x}, \vec{y} \in \mathcal{H}$ and $\lambda \in \C$};$$
{\em $\R$-linear} (or {\em real linear}) if 
$$A (\vec{x} + \vec{y}) = A \vec{x} + A \vec{y} \; \mbox{and} \; A(\lambda \vec{x}) = \lambda A \vec{x}  \;  \; \mbox{for all $\vec{x}, \vec{y} \in \mathcal{H}$ and $\lambda \in \R$};$$
and {\em antilinear} if 
$$A (\vec{x} + \vec{y}) = A \vec{x} + A \vec{y} \; \mbox{and} \; A(\lambda \vec{x}) = \bar \lambda A \vec{x}  \;  \; \mbox{for all $\vec{x}, \vec{y} \in \mathcal{H}$ and $\lambda \in \C$}.$$
When a complex linear transformation $A$ is bounded, its (complex linear) adjoint is the familiar complex linear transformation $A^{*}$ on $\mathcal{H}$ satisfying $\langle A \vec{x}, \vec{y}\rangle = \langle \vec{x}, A^{*} \vec{y}\rangle$ for all $\vec{x}, \vec{y} \in \mathcal{H}$. The existence of $A^{*}$ comes from the Riesz representation theorem for Hilbert spaces.

For a bounded {\em antilinear} operator $A$ on $\mathcal{H}$, its adjoint $A^{\star}$ is the antilinear operator on $\mathcal{H}$ that satisfies 
\begin{equation}\label{oooOOkkkKKKK}
\langle A \vec{x}, \vec{y}\rangle = \overline{\langle \vec{x}, A^{\star} \vec{y}\rangle} \; \mbox{for all $\vec{x}, \vec{y} \in \mathcal{H}$.}
\end{equation} 
In the above, take special note of the $\star$ versus the $\ast$ superscript on $A$. The existence of $A^{\star}$ also comes from the Riesz representation theorem for Hilbert spaces.

 Notice that every real linear operator $A$ on $\mathcal{H}$ can be written as 
\begin{equation}\label{7796316}
A = \tfrac{1}{2} (A - i A i I) + \tfrac{1}{2} (A + i A iI) =: A_{c} + A_{a}
\end{equation}
 and one can check that the first summand $A_{c}$ is complex linear while the second $A_{a}$ is antilinear. For a bounded real linear operator $A$ on $\mathcal{H}$, its {\em real linear adjoint}
$A^{\dagger}$ is defined to be 
\begin{equation}\label{realadjoint}
A^{\dagger} := A_{c}^{*} + A_{a}^{\star}.
\end{equation}
 Note that the real adjoint $A^{\dagger}$ of $A$ satisfies $\Re \langle A \vec{x}, \vec{y}\rangle = \Re \langle \vec{x}, A^{\dagger} \vec{y}\rangle$ for all $\vec{x}, \vec{y} \in \mathcal{H}$. As this will be used later, we record the following remark.
 
 \begin{Remark}\label{reasaalalalalll}
 The above defines a real linear adjoint for a real linear operator on a {\em complex} Hilbert space. 
 Suppose that  $\mathcal{H}$ is a {\em real} Hilbert space with $\langle \vec{x}, \vec{y}\rangle \in \R$ for all $\vec{x}, \vec{y} \in \mathcal{H}$ (for example, the space $\mathcal{H}_{\R}$ from \eqref{HrrRR}). Then, via the Riesz representation theorem, there is a bounded real linear operator $A^{\dagger}$ on $\mathcal{H}$ which satisfies 
 $$\langle A \vec{x}, \vec{y}\rangle =  \langle \vec{x}, A^{\dagger} \vec{y}\rangle \; \mbox{for all $\vec{x}, \vec{y} \in \mathcal{H}$}.$$ 
 \end{Remark}

\section{A representation of a unitary operator}

For our given unitary operator $U$ on $\mathcal{H}$ and our isometric isomorphism $X: \mathcal{H} \to \mathfrak{C}(\mathcal{H}_{\R})$ from \eqref{XXXXex}, we need to compute $X U X^{-1}$ as a complex linear operator on $\mathfrak{C}(\mathcal{H}_{\R})$. Let
 $Q_\R$ denote the mapping from $\mathcal{H}$ onto  $\h_\R$ defined by
 \begin{equation}\label{QqqrrrRR}
Q_\R (\vec{h}_r + i \vec{h}_c) : =\vec{h}_r
\end{equation} and note that if $\vec{h} = \vec{h}_r + i \vec{h}_c$, then
\begin{equation}\label{QqqqQ}
\vec{h}=Q_\R \vec{h}-iQ_\R(i \vec{h}).
\end{equation}
It is important to observe that $Q_{\R}$ is real linear {\em but not necessarily complex linear}.

Now define the mappinngs $U_r$ and $U_c$ on the real Hilbert space $\mathcal{H}_{\R}$ by 
\begin{equation}\label{QQQQRRRRrrR}
U_r,U_c\colon \h_\R\to\h_\R, \quad  U_r \vec{h}_r := Q_\R U \vec{h}_r, \quad   U_c \vec{h}_r :=-Q_\R( i U (\vec{h}_r))
\end{equation}
 and note that $U_r$ and $U_c$ are {\em real linear} operators.
From \eqref{QqqqQ} we see that
\begin{align*}
XUX^{-1}\begin{bmatrix}\vec{h}_r\\\vec{h}_c\end{bmatrix} &=XU(\vec{h}_r+i \vec{h}_c)\\
&=X(U \vec{h}_r+iU \vec{h}_c)\\
&=X\left(Q_\R(U \vec{h}_r+U(i \vec{h}_c))-iQ_\R(i(U \vec{h}_r+iU( \vec{h}_c))\right)\\
&=X((Q_\R U\vec{h}_r+Q_\R U(i \vec{h}_c))+i(-Q_\R U(i \vec{h}_r)+Q_\R  U \vec{h}_c))\\
&=
\begin{bmatrix}\phm Q_\R U \vec{h}_r+Q_\R U(i \vec{h}_c)\\ -Q_\R U(i \vec{h}_r)+Q_\R U \vec{h}_c\end{bmatrix}\\
& = \begin{bmatrix}U_r \vec{h}_r - U_{c} \vec{h}_c\\U_{c} \vec{h}_r + U_{r} \vec{h}_{c}\end{bmatrix}\\
& =
\begin{bmatrix}
U_r & - U_c\\
U_c & \phm U_r
\end{bmatrix}
\begin{bmatrix}
\vec{h}_r\\
\vec{h}_c
\end{bmatrix}.
\end{align*}

The discussion above shows that the map $\widehat{U}$ on $\mathfrak{C}(\mathcal{H}_{\R})$ defined by 
$$\widehat{U} \begin{bmatrix}
\vec{h}_r\\
\vec{h}_c
\end{bmatrix} := \begin{bmatrix}
U_r & - U_c\\
U_c & \phm U_r
\end{bmatrix}
\begin{bmatrix}
\vec{h}_r\\
\vec{h}_c
\end{bmatrix} =  \begin{bmatrix}U_r \vec{h}_r - U_{c} \vec{h}_c\\ U_{c} \vec{h}_r + U_{r} \vec{h}_{c}\end{bmatrix}$$
is  $\C$-linear, due to the fact that $X$, $U$, and $X^{-1}$ are $\C$-linear maps. In addition, since $X$ is an isometric isomorphism, we conclude the following. 

\begin{Theorem}\label{sfddsfsdf}
The unitary operator $U$ on $\mathcal{H}$ is unitarily equivalent to the operator
$$ \widehat{U} := \begin{bmatrix}
U_r & - U_c\\
U_c & \phm U_r
\end{bmatrix},$$
where $U_r$ and $U_c$ are given by \eqref{QQQQRRRRrrR},
on $\mathfrak{C}(\mathcal{H}_{\R})$.
\end{Theorem}

Let us make a few observations about the real linear operators $U_r$ and $U_c$. We begin with a lemma from \cite[Lem.~4, Lem.~5]{MR0190750}. Since this reference is not readily available in English, we include its short proof.

\begin{Lemma}\label{GLuke}
For a unitary operator $U$ on $\mathcal{H}$ with
$$\mathcal{H} = \bigoplus_{k \geq 1} \mathcal{H}(\vec{x}_k) \; \mbox{and} \; 
\mathcal{H}_{\R} = \bigoplus_{k \geq 1} \mathcal{H}(\vec{x}_k)_{\R}$$
as in \eqref{Hhhkkk} and \eqref{HrrRR}, there is a conjugation $J$ on $\mathcal{H}$ such that
\begin{equation}\label{112Sdfsdf}
J\vec{x} = \vec{x}, \quad \vec{x} \in \mathcal{H}_{\R},
\end{equation} and  $J U J =U^{*}$.
\end{Lemma}

\begin{proof}
Notice that $\mathcal{H}_{\R}$ is a real Hilbert space embedded in the complex Hilbert space $\mathcal{H}$. Moreover, by \eqref{dfgjdfg88810ooO}, $(\vec{u}_{\ell})_{\ell \geq 1}$ is an orthonormal basis for $\mathcal{H}$ with $(\vec{u}_{\ell})_{\ell \geq 1} \subset  \mathcal{H}_{\R}$. Define a conjugation $J$ on $\mathcal{H}$ by $J \vec{u}_{\ell} = \vec{u}_{\ell}$ and extend antilinearly  to $\mathcal{H}$ by 
$$J \Big(\sum_{\ell \geq 1} a_{\ell} \vec{u}_{\ell} \Big) := \sum_{\ell \geq 1} \bar a_{\ell} \vec{u}_{\ell}$$ and note that $J \vec{x} = \vec{x}$ for all $\vec{x} \in \mathcal{H}_{\R}$. 

Recall the spectral representation of $U$ from \eqref{spectralrep}, i.e., 
\begin{equation}\label{UdeeeEE}
U = \int_{\T} \xi d E(\xi).
\end{equation}
For any Borel set $\Omega \subset \T$, we have 
$$J E(\Omega) J \vec{x}_{k} = J E(\Omega) \vec{x}_{k} = E(\Omega) \vec{x}_k$$
for all $k \geq 1$ (recall that $E(\Omega) \vec{x}_k \in \mathcal{H}_{\R}$ from \eqref{ahhhhhkkkkxxxrrr}) and
thus $J E(\Omega) J = E(\Omega)$ on $\mathcal{H}_{\R}$ and thus on $\mathcal{H}$. Finally, from \eqref{UdeeeEE}, we have 
\begin{align*}
J U J & = J \Big(\int_{\T} \xi d E(\xi) \Big)J\\
& = \int_{\T} \overline{\xi} (J dE(\xi) J)\\
& = \int_{\T} \overline{\xi} d E(\xi)\\
& = U^{*},
\end{align*}
which proves the result. 
\end{proof}

\begin{Proposition}\label{Jayyyyhat}
If $\widehat{J} := X J X^{-1}$, then $\widehat{J}$ is a conjugation on $\mathfrak{C}(\mathcal{H}_{\R})$ and
$$\widehat{J}\begin{bmatrix}\vec{h}_{r}\\ \vec{h}_{c}\end{bmatrix} = \begin{bmatrix}\phantom{-}\vec{h}_{r}\\ -\vec{h}_{c}\end{bmatrix} \; \; \mbox{for all} \; \;  \begin{bmatrix}\vec{h}_{r}\\ \vec{h}_{c}\end{bmatrix} \in \mathfrak{C}(\mathcal{H}_{\R}).$$ Moreover,
$$\widehat{J} \widehat{U} \widehat{J} = \widehat{U^{*}} =  \begin{bmatrix}\phm U_{r} & U_{c}\\ -U_{c} & U_{r}\end{bmatrix}.$$
\end{Proposition}

\begin{proof}
The fact that $\widehat{J}$ is a conjugation follows from the fact that $J$ is a conjugation on $\mathcal{H}$ and $X$ is an isometric isomorphism from $\mathcal{H}$ onto $\mathfrak{C}(\mathcal{H}_{\R})$. Now observe that since $J \vec{h}_r = \vec{h}_r$ and $J \vec{h}_c = \vec{h}_c$ (see \eqref{112Sdfsdf}) we have
\begin{align*}
\widehat{J}\begin{bmatrix}\vec{h}_{r}\\ \vec{h}_{c}\end{bmatrix} & = X J X^{-1}\begin{bmatrix}\vec{h}_{r} \\\vec{h}_{c}\end{bmatrix} \\
& = X J (\vec{h}_{r} + i \vec{h}_{c})\\
& = X (\vec{h}_{r} - i \vec{h}_{c})\\
& =  \begin{bmatrix}\phantom{-}\vec{h}_{r} \\-\vec{h}_{c}\end{bmatrix}.
\end{align*}

Finally, by the calculation before Theorem \ref{sfddsfsdf}, we see that
\begin{align*}
\widehat{J} \widehat{U} \widehat{J} \begin{bmatrix}\vec{h}_{r}\\ \vec{h}_{c}\end{bmatrix} & = \widehat{J} X U X^{-1} \begin{bmatrix}\phantom{-}\vec{h}_{r} \\ - \vec{h}_{c}\end{bmatrix} \\
& =  \widehat{J} X U (\vec{h}_{r} - i \vec{h}_{c})\\
& =  \widehat{J} X (U \vec{h}_{r} - i U \vec{h}_{c})\\
& = \widehat{J}\begin{bmatrix}U_{r} \vec{h}_{r} + U_{c} \vec{h}_{c}\\ U_{c} \vec{h}_{r} - U_{r} \vec{h}_{c}\end{bmatrix}\\
& = \begin{bmatrix} U_{r} \vec{h}_{r} + U_{c} \vec{h}_{c}\\ U_{r} \vec{h}_{c} - U_{c} \vec{h}_{r}\end{bmatrix}\\
& = \begin{bmatrix}\phm U_{r} & U_{c}\\ -U_{c} & U_{r}\end{bmatrix} \begin{bmatrix} \vec{h}_{r}\\ \vec{h}_{c}\end{bmatrix}.
\end{align*}
On the other hand, by Lemma \ref{GLuke}, 
\begin{align*}
\widehat{J} \widehat{U} \widehat{J}  & = (X J X^{-1})(X U X^{-1})(X J X^{-1})\\
& = X (J U J) X^{-1}\\
& = X U^{*} X^{-1}\\
& = \widehat{U^{*}},
\end{align*}
which verifies the desired adjoint formula.
\end{proof}

We now want to compute $\widehat{U}^{*}$, the adjoint of the matrix operator $\widehat{U}$. There is a more general result one can state here. Indeed, suppose
$$\widehat{W} := \begin{bmatrix} W_{rr} & W_{cr}\\ W_{rc} & W_{cc}\end{bmatrix},$$ where each of the entries $W_{rr}, W_{cr}, W_{rc}, W_{cc}$ are {\em real} linear operators on $\mathcal{H}_{\R}$. Notice that $\widehat{W}$ defines a real linear operator on $\mathfrak{R}(\mathcal{H}_{\R})$ by 
$$\widehat{W}  \begin{bmatrix}
\vec{h}_r\\
\vec{h}_c
\end{bmatrix} := \begin{bmatrix}
W_{rr} & W_{cr}\\
W_{rc} & Wcc
\end{bmatrix}
\begin{bmatrix}
\vec{h}_r\\
\vec{h}_c
\end{bmatrix} = \begin{bmatrix} W_{rr} \vec{h}_r + W_{cr} \vec{h}_c\\ W_{rc} \vec{h}_r + W_{cc} \vec{h}_c\end{bmatrix}.$$

\begin{Lemma}\label{C-linearW}
The real linear operator 
$$\widehat{W} = \begin{bmatrix} W_{rr} & W_{cr}\\ W_{rc} & W_{cc}\end{bmatrix}$$ on the real space $\mathfrak{R}(\mathcal{H}_{\R})$ is $\C$-linear on the complex space $\mathfrak{C}(\mathcal{H}_{\R})$ if and only if $W_{rr} = W_{cc}$ and $W_{cr} = - W_{rc}$.
\end{Lemma}

\begin{proof}
For 
$$\begin{bmatrix} \vec{h}_r\\ \vec{h}_c\end{bmatrix}\in \mathfrak{C}(\mathcal{H}_{\R})$$ we use the definition of multiplication by a complex scalar in $\mathfrak{C}(\mathcal{H}_{\R})$ from \eqref{cdef} to get 
\begin{align*}
\widehat{W}\Big(i \begin{bmatrix} \vec{h}_r\\ \vec{h}_c\end{bmatrix}\Big) & = \begin{bmatrix} W_{rr} & W_{cr}\\ W_{rc} & W_{cc}\end{bmatrix} \begin{bmatrix} -\vec{h}_{c}\\ \phm\vec{h}_{r}\end{bmatrix}\\
& = \begin{bmatrix} - W_{rr} \vec{h}_c + W_{cr} \vec{h}_r\\ -W_{rc} \vec{h}_c + W_{cc} \vec{h}_{r}\end{bmatrix}\\
& = i  \begin{bmatrix}\phm W_{cc} \vec{h}_r - W_{rc} \vec{h}_c\\ - W_{cr} \vec{h}_r + W_{rr} \vec{h}_c\end{bmatrix}\\
& =i  \begin{bmatrix} \phm W_{cc} & - W_{rc}\\ -W_{cr} & \phantom{-}W_{rr}\end{bmatrix}\begin{bmatrix} \vec{h}_{r}\\ \vec{h}_{c}\end{bmatrix}.
\end{align*}
On the other hand,
$$i \widehat{W} \begin{bmatrix} \vec{h}_r\\ \vec{h}_c\end{bmatrix} =  i \begin{bmatrix} W_{rr} & W_{cr}\\ W_{rc} & W_{cc}\end{bmatrix} \begin{bmatrix} \vec{h}_r\\ \vec{h}_c\end{bmatrix}.$$
Comparing the matrix entries yields the result.
\end{proof}

For the entries $W_{rr}, W_{cr}, W_{rc}, W_{cc}$, which define real linear operators on the real Hilbert space $\mathcal{H}_{\R}$, let  $W_{rr}^{\dagger}$, $W_{cr}^{\dagger}$, $W_{rc}^{\dagger}$, $W_{cc}^{\dagger}$ denote their respective {\em real adjoints} from  Remark \ref{reasaalalalalll}  in that $\langle W \vec{h}, \vec{k}\rangle = \langle \vec{h} ,W^{\dagger} \vec{k}\rangle$ for all $\vec{h}, \vec{k} \in \mathcal{H}_{\R}$. 

\begin{Lemma}
Suppose
$$\widehat{W} = \begin{bmatrix} W_{rr} & - W_{rc}\\ W_{rc} &\phm W_{rr}\end{bmatrix}$$ is $\C$-linear on a complex space $\mathfrak{C}(\mathcal{H}_{\R})$. Then the formula for its adjoint is
$$\widehat{W}^{*} = \begin{bmatrix}\phm W_{rr}^{\dagger} & W^{\dagger}_{rc}\\ -W^{\dagger}_{rc} & W_{rr}^{\dagger}\end{bmatrix}.$$
\end{Lemma}

\begin{proof}
Let
$$\begin{bmatrix}\vec{h}_r\\ \vec{h}_c\end{bmatrix},  \begin{bmatrix} \vec{g}_r\\\vec{g}_{c}\end{bmatrix} \in \mathfrak{C}(\mathcal{H}_{\R}).$$ Then
$$
 \Big\langle \widehat{W} \begin{bmatrix} \vec{h}_r \\ \vec{h}_{c}\end{bmatrix}, \begin{bmatrix} \vec{g}_r\\ \vec{g}_c\end{bmatrix}\Big\rangle  = \Big\langle \begin{bmatrix} W_{rr} & -W_{cc}\\ W_{rc} &\phm W_{rr}\end{bmatrix} \begin{bmatrix} \vec{h}_{r}\\ \vec{h}_{c}\end{bmatrix}, \begin{bmatrix} \vec{g}_{r}\\ \vec{g}_{c}\end{bmatrix}\Big\rangle.
$$
We leave it to the reader to perform the necessary matrix multiplication and apply the definition of the inner product in $\mathfrak{C}(\mathcal{H}_{\R})$, along with the definition of the real adjoint, to see this becomes
$$\Big\langle \begin{bmatrix} \vec{h}_r \\ \vec{h}_c \end{bmatrix},\left[\!{\begin{array}{cc}\phantom{-} W_{rr}^{\dagger}& W_{rc}^{\dagger} \\- W_{rc}^{\dagger} & W_{rr}^{\dagger}\end{array}}\!\right]\left[\!{\begin{array}{c} \vec{g}_r \\ \vec{g}_c \end{array}}\!\right]\Big\rangle,$$
which verifies the adjoint formula.
\end{proof}

For this next technical lemma, recall the conjugation $J$ on $\mathcal{H}$ from Lemma \ref{GLuke} and the associated conjugation $\widehat{J}$ on $\mathfrak{C}(\mathcal{H}_{\R})$  from Proposition \ref{Jayyyyhat}.

\begin{Lemma} Let
$$\widehat{W} = \begin{bmatrix} W_{rr} & - W_{rc}\\ W_{rc} &\phm W_{rr}\end{bmatrix}$$
be $\C$-linear as an operator in complex space $\frak{C}(\mathcal{H}_{\R})$. The  operator $\widehat{W}$
is $\widehat{J}$ symmetric, that is to say, $\widehat{J} \widehat{W} \widehat{J} = \widehat{W}^{*}$ if and only if $W_{rr}^{\dagger} = W_{rr}$ and $W_{rc}^{\dagger} = W_{rc}$.
\end{Lemma}

\begin{proof}
This involves a similar computation as we have done before so we leave it to reader to work through the details. 
\end{proof}

We represented our general unitary $U$ on $\mathcal{H}$ by
$$\widehat{U} =  \begin{bmatrix}
U_r & - U_c\\
U_c &\phm U_r
\end{bmatrix}$$
on $\mathfrak{C}(\mathcal{H}_{\R})$. From our discussion above we know that
$$\widehat{U^{*}} = \widehat{U}^{*} = \begin{bmatrix}\phm U_{r} & U_{c}\\ - U_{c} & U_{r}\end{bmatrix}$$
and we emphasize that $U_{r}^{\dagger} = U_{r}$ and $U_{c}^{\dagger} = U_{c}$. Using the fact that $\widehat{U}$ is unitary, we see that $$U_{r}^{2} + U_{c}^{2} = I \quad \mbox{and} \quad U_{r} U_{c} = U_{c} U_{r}.$$

\section{Revisiting some of our previous examples}

\begin{Example}\label{UaUbbbb}
Returning to Example \ref{diagonallslslslsls} with 
$$U = \sum_{j = 1}^{n} \xi_j (\vec{e}_j \otimes \vec{e}_{j}),$$
an $n \times n$  diagonal unitary matrix with distinct eigenvalues $(\xi_j)_{j = 1}^{n}$,
we know that
$$\mathfrak{C}(\C^n_{\R}) = \R^n \oplus \R^n  = \left\{\begin{bmatrix} \vec{a} \\ \vec{b} \end{bmatrix}: \vec{a}, \vec{b} \in \R^n\right\}.$$
Also observe from \eqref{QqqrrrRR} that 
$$Q_{\R} (\vec{a} + i \vec{b}) = \vec{a} = \Re (\vec{a} + i \vec{b}).$$
Since $\vec{a} = \sum_{j = 1}^{n} a_j \vec{e}_j$, it follows that 
\begin{align*}
U_{r} \vec{a} & = Q_{\R} \Big(\sum_{j = 1}^{n} \xi_j \langle \vec{a}, \vec{e}_j\rangle \vec{e}_j\Big)\\
& = Q_{\R} \Big(\sum_{j = 1}^{n} (\xi_j a_j) \vec{e}_j\Big)\\
& = \sum_{j = 1}^{n} \Re (\xi_j a_j) \vec{e}_j\\
& = \Re U \vec{a},
\end{align*}
and, in a similar way, 
$$U_{c} \vec{b} = \Im U \vec{b}.$$
Putting this all together yields 
\begin{align*}
\widehat{U} \begin{bmatrix} \vec{a}\\ \vec{b}\end{bmatrix} & =  \begin{bmatrix}
U_r & - U_c\\
U_c &\phm U_r
\end{bmatrix} \begin{bmatrix} \vec{a}\\ \vec{b}\end{bmatrix} \\
&= \begin{bmatrix} \Re U \vec{a}\\ \Re U \vec{b}\end{bmatrix} + \begin{bmatrix}-\Im U \vec{b}\\ \phantom{-}\Im U \vec{a}\end{bmatrix}\\
 &= \phantom{-}\Re \begin{bmatrix} U \vec{a}\\ U \vec{b}\end{bmatrix} + i \Im  \begin{bmatrix} U \vec{a}\\ U \vec{b}\end{bmatrix}.
\end{align*}
From here one can see that $\widehat{U}$ is unitarily equivalent to $U$. 
\end{Example}

\begin{Example}\label{sdkfsdfkkW}
Let us compute the operator matrix $\widehat{U}$ when $(U f)(\xi) = \xi f(\xi)$ is the bilateral shift on $L^2(m)$ from Example \ref{bilallateralll}. From that discussion we know that
$$\mathfrak{C}(L^2(m)_{\R}) = L^{2}_{\R}(m) \oplus L^{2}_{\R}(m),$$
where $L^{2}_{R}(m)$ are the real-valued $L^2(m)$ functions.
From \eqref{QqqrrrRR} we have, for $f, g \in L^2_{\R}(m)$, 
$$Q_{\R} (f + i g) = f = \Re (f + i g).$$
For $f \in L^{2}_{\R}(m)$, we can use  \eqref{QQQQRRRRrrR} to see that 
\begin{align*}
U_{r} f & = Q_{\R} U|_{L^{2}_{\R}(m)} f\\
& =  Q_{\R} ((x + i y) f) \\
&= Q_{\R} (x f + i y f)\\
& = x f. 
\end{align*}
Thus, $U_{r} = M_{x}$ (multiplication by $x$) on $L^{2}_{\R}(m)$. In a similar way, observe that
\begin{align*}
U_{c} f  & = - Q_{\R} ((x + i y) (if)\\
&  = - Q_{\R} (ix f - y f)\\
&= y f.
\end{align*}
Thus, $U_{c} = M_{y}$ (multiplication by $y$) on $L^{2}_{\R}(m)$. In summary,
$$\widehat{U} = \begin{bmatrix} U_{r} & - U_{c}\\ U_{c} & \phm U_{r}\end{bmatrix} = \begin{bmatrix} M_{x} & - M_{y}\\ M_{y} & \phm M_{x}\end{bmatrix}.$$ For $\phi \in L^{\infty}(m)$ and  unimodular, the above discussion yields 
$$\widehat{M_{\phi}} = \begin{bmatrix} M_{\Re \phi} & - M_{\Im \phi}\\ M_{\Im \phi} & \phantom{-}M_{\Re \phi}\end{bmatrix}.$$ 
\end{Example}

\begin{Example}\label{FTTagain}
Let us return to our discussion of the Fourier--Plancherel transform $\mathcal{F}$ on $L^2(\R)$ from Example \ref{nknjdhfkhfghHJHJHJ} and Example \ref{FTexamplesaepaeate}. From \eqref{diagonalizeFFF} recall that 
$$\mathcal{F} = \sum_{n = 0}^{\infty} (-i)^{n} (H_{n} \otimes H_{n})$$
and that each Hermite function $H_n$ is real-valued. From the discussion in Example \ref{FTexamplesaepaeate} we know that  
$$\mathfrak{C}(L^2(\R)_{\R}) = L^{2}_{\R}(\R) \oplus L^{2}_{\R}(\R).$$
With 
$$h_{r} = \sum_{n = 0}^{\infty} \alpha_{n} H_n \; \; \mbox{and} \; \; h_{c} = \sum_{n = 0}^{\infty} \beta_{n} H_n,$$
where $(\alpha_n)_{n \geq 0}, (\beta_{n})_{n \geq 0}$ are real $\ell^2$ sequences, we see from \eqref{QQQQRRRRrrR} that 
\begin{align*}
U_{r} h_{r} & = Q_{\R} \mathcal{F} \Big(\sum_{n = 0}^{\infty} \alpha_n H_n\Big)\\
& = Q_{\R} \Big(\sum_{n = 0}^{\infty}\alpha_{n} (-i)^n H_n\Big)\\
& = \sum_{n = 0}^{\infty} \alpha_{2n} (-1)^n H_{2n}.
\end{align*}
In a similar way, 
\begin{align*}
U_{c} h_{c} & = -Q_{\R} \mathcal{F}(i h_c)\\
& = -Q_{\R} \mathcal{F}\Big(i\sum_{n = 0}^{\infty} \beta_{n} H_{n}\Big)\\
& = -Q_{\R} \Big(i \sum_{n = 0}^{\infty}\beta_{n} (-i)^n H_{n}\Big)\\
& = Q_{\R} \Big(\sum_{n = 0}^{\infty} \beta_n (-i)^{n + 1} H_n\Big)\\
& = \sum_{n = 0}^{\infty} \beta_{2 n + 1} (-1)^{n} H_{2 n + 1}.
\end{align*}
Thus,  the representation $\widehat{\mathcal{F}}$ on $\mathfrak{C}(L^2(\R)_{\R})$ becomes 
\begin{align*}
\widehat{\mathcal{F}} \begin{bmatrix} h_{r}\\ h_{c}\end{bmatrix} & = \begin{bmatrix} U_{r} & - U_{c}\\ U_{c} & \phm U_{r}\end{bmatrix} \begin{bmatrix} \sum_{n \geq 0} \alpha_n H_n\\ \sum_{n \geq 0} \beta_n H_n\end{bmatrix}\\
& = \begin{bmatrix} \sum_{n \geq 0} (\alpha_{2n} (-1)^n H_{2n}  - \beta_{2 n + 1} (-1)^n H_{2n + 1})\\
 \sum_{n \geq 0} (\alpha_{2n + 1} (-1)^n H_{2n + 1}  + \beta_{2 n} (-1)^n H_{2n})\end{bmatrix}.
\end{align*}
\end{Example}

\begin{Example}\label{hhhhtranspp6}
Let us return to our discussion of  the Hilbert transform $\mathscr{H}$ on $L^2(\R)$ from Example \ref{ldfgjdfjg888uuUU} and Example \ref{HT664466yY}. From \eqref{diagonalizeH} we know that 
$$\mathscr{H} = \sum_{n < 0} i (f_n \otimes f_n) + \sum_{n \geq 0} (-i) (f_n \otimes f_n).$$ Thus, if 
$$h_{r} = \sum_{n < 0} \alpha_{n} f_n + \sum_{n \geq 0} \alpha_n f_n \; \; \mbox{and} \; \; 
h_{c} = \sum_{n < 0} \beta_{n} f_n + \sum_{n \geq 0} \beta_n f_n,$$
where $(\alpha_{n})_{n = -\infty}^{\infty}$ and $(\beta_n)_{n = - \infty}^{\infty}$ are real $\ell^2(\Z)$ sequences, 
we see that 
\begin{align*}
U_{r} h_{r} & = Q_{\R} \mathscr{H} \Big(\sum_{n < 0} \alpha_{n} f_n + \sum_{n \geq 0} \alpha_n f_n\Big)\\
& = Q_{\R} \Big(\sum_{n < 0} i \alpha_{n} f_n + \sum_{n \geq 0}(-i) \alpha_n f_n\Big)\\
& = 0.
\end{align*}
On the other hand, 
\begin{align*}
U_{c} h_{c} & = -Q_{\R} \mathscr{H}(i h_c)\\
& = - Q_{\R} \mathscr{H} \Big(\sum_{n < 0} i \beta_{n} f_n + \sum_{n \geq 0} i \beta_n f_n\Big)\\
& = -  Q_{\R} \Big(\sum_{n < 0} i (i) \beta_{n} f_n + \sum_{n \geq 0}(-i)  (i)\beta_n f_n\Big)\\
& = -Q_{\R} \Big(\sum_{n < 0} - \beta_n f_n + \sum_{n \geq 0} \beta_n f_n\Big)\\
& = \sum_{n < 0} \beta_n f_n - \sum_{n \geq 0} \beta_n f_n.
\end{align*}
Thus, the representation $\widehat{\mathscr{H}}$ of $\mathscr{H}$ on $\mathfrak{C}(L^2(\R)_{\R})$ becomes 
\begin{align*}
\widehat{\mathscr{H}} \begin{bmatrix} h_r\\ h_c\end{bmatrix} & = \begin{bmatrix} U_{r} & - U_{c}\\ U_{c} & \phm U_{r}\end{bmatrix} \begin{bmatrix} h_r\\ h_c\end{bmatrix}\\
& = \begin{bmatrix} 0 & - U_{c}\\ U_{c} & \phm 0 \end{bmatrix} \begin{bmatrix} h_r\\ h_c\end{bmatrix}\\
& = \begin{bmatrix} -  U_c h_c \\  \phm U_{c} h_r\end{bmatrix}\\
& = \begin{bmatrix} - \sum_{n < 0} \beta_n f_n + \sum_{n \geq 0} \beta_n f_n\\
\phantom{-}\sum_{n < 0} \alpha_n f_n - \sum_{n \geq 0} \alpha_n f_n 
\end{bmatrix}.
\end{align*}
\end{Example}

\section{A Description of the Orbit}

Using our model $\widehat{U}$ on $\mathfrak{C}(\mathcal{H}_{\R})$ for a unitary operator $U$ on $\mathcal{H}$ from the  previous section, we have the following description of $\mathfrak{O}_c(U)$. 

\begin{Theorem}
Let $U$ be a fixed unitary operator on a separable Hilbert space $\mathcal{H}$. For a unitary operator $V$ on $\mathcal{H}$, the following are equivalent.
\begin{enumerate}
\item[(a)] $V \in \mathfrak{O}_c(U)$.
\item[(b)] $V = W U^{*} W^{*}$, where
$$\widehat{W} = \begin{bmatrix} W_{r} & -W_{c}\\ W_{c} & \phm W_{r}\end{bmatrix},$$
where $W_{r}, W_{c}$ are real linear operators on $\mathcal{H}_{\R}$ satisfying
\begin{equation}\label{wrWccc}
W_{r}^{\dagger} = W_{r}, \; W_{c}^{\dagger} = W_{c}, \; W_{r} W_{c} = W_{c} W_{r}, \;  W_{r}^{2} + W_{c}^{2} = I.
\end{equation}
\item[(c)] There are real linear operators $W_{r}, W_{c}$  on $\mathcal{H}_{\R}$ satisfying  the conditions in \eqref{wrWccc} such that the operator $\widehat{V}$ on $\mathfrak{C}(\mathcal{H}_{\R})$ satisfies 
$\widehat{V} = \widehat{W} \widehat{U}^{*} \widehat{W}^{*},$ where
$$\widehat{W} = \begin{bmatrix} W_{r} & -W_{c}\\ W_{c} & \phm W_{r}\end{bmatrix}.$$
\end{enumerate}
\end{Theorem}

\begin{proof}
$(a) \Longrightarrow (b)$: Recall the conjugation $J$ on $\mathcal{H}$ from Lemma \ref{GLuke} which fixes every element of $\mathcal{H}_{\R}$ and satisfies $J U J = U^{*}$.  If $V \in \mathfrak{O}_c(U)$, then Proposition \ref{sdfhjsd9f9ds9dsf9sdf9s9df} says there is a unitary operator $W$ on $\mathcal{H}$ with $J W J = W^{*}$ and $V = W U^{*} W^{*}$. Now apply the results from the previous section to obtain (b).

$(b) \Longrightarrow (a)$: The conditions on $\widehat{W}$ say that $\widehat{W}$ is unitary and $\widehat{J}$ symmetric. Via the isometric isomorphism $X$ from \eqref{XXXXex}, this means that $W$ is unitary and $J W J = W^{*}$. Proposition \ref{sdfhjsd9f9ds9dsf9sdf9s9df} now completes the proof.

$(b) \iff (c)$ is clear.
\end{proof}

\begin{Remark}\label{sdfsdfRemaRKUU***}
Using (c) from the previous theorem and setting $W_{r} = I$ and $W_{c} = 0$, we see that $U^{*} \in \mathfrak{O}_c(U)$, which is confirmed by Proposition \ref{SdfsdfdsfcvvvvVV} (see also Remark \ref{2..99999} and Remark \ref{sdfsdfRemaRKUU***}).
\end{Remark}

\begin{Example}
Let us continue the discussion from Example \ref{UaUbbbb}, where $U$ is a diagonal $n \times n$ unitary matrix with distinct eigenvalues. In that example, we showed that 
$$
\widehat{U}  \begin{bmatrix} \vec{a}\\ \vec{b}\end{bmatrix} =  \begin{bmatrix}
U_r & - U_c\\
U_c &\phm U_r
\end{bmatrix} \begin{bmatrix} \vec{a}\\ \vec{b}\end{bmatrix},
$$
where $U_{r} \vec{a} = \Re (U \vec{a})$ and $U_{c} \vec{b} = \Im(U \vec{b})$ for $\vec{a}, \vec{b} \in \R^{n}$. Suppose $X$ is an $n \times n$  self transpose unitary matrix with {\em real} entries (for example, $X = I - 2 (\vec{x} \otimes \vec{x})$ and $\vec{x}$ is a unit vector from $\R^n$ -- a Householder matrix). If $W_{r} \vec{a} := X \vec{a}$, then, since $X^2 = I$ and $X^{t} = X$, we see that the real linear operator $W_{r}$ satisfies $W_{r}^2 = I$ and $W_{r}^{\dagger} = W_r$. Then, with $W_c = 0$, 
$$ \begin{bmatrix} W_{r} & 0\\ 0 & W_{r}\end{bmatrix}  \begin{bmatrix} \phantom{-}U_r &  U_c\\ - U_c & U_r\end{bmatrix}  \begin{bmatrix} W_{r} & 0\\ 0 & W_{r}\end{bmatrix} \begin{bmatrix} \vec{a}\\ \vec{b}\end{bmatrix}$$
is equal to 
\begin{align*}
\begin{bmatrix} \phantom{-}\Re(X U X \vec{a}) + \Im( X U X \vec{b})\\ -\Im(XUX \vec{a}) + \Re(X U X \vec{b})\end{bmatrix} & = \Re \begin{bmatrix} XUX \vec{a}\\ XUX \vec{b}\end{bmatrix} + \Im \begin{bmatrix} XUX \vec{b} \\ -XUX \vec{a}\end{bmatrix}\\
& = \Re \begin{bmatrix} XUX \vec{a}\\ XUX \vec{b}\end{bmatrix}  - i \Im  \begin{bmatrix} XUX \vec{a} \\ XUX \vec{b}\end{bmatrix}\\
& = \widehat{X \overline{U}X} \begin{bmatrix} \vec{a}\\\vec{b}\end{bmatrix}.
\end{align*}
Thus, $X \overline{U} X \in \mathfrak{O}_c(U)$, which is rediscovery of Proposition \ref{dfgdfgkkKKK12}.
\end{Example}

\begin{Example}
For the bilateral shift $(U f)(\xi) = \xi f(\xi)$ on $L^2(m)$, we know from Example \ref{sdkfsdfkkW} that
$$\widehat{U} = \begin{bmatrix} M_{x} & - M_{y}\\ M_{y} & \phantom{-}M_{x}\end{bmatrix}.$$
Thus, viewing everything in the $\mathfrak{C}(\mathcal{H}_{\R})$ setting, we see that the conjugate orbit of $U$ is the collection of operators 
$$ \begin{bmatrix} W_{r} & -W_{c}\\ W_{c} & \phantom{-}W_{r}\end{bmatrix}  \begin{bmatrix} \phantom{-}M_{x} &  M_{y}\\ - M_{y} & M_{x}\end{bmatrix}  \begin{bmatrix} \phantom{-}W_{r} & W_{c}\\ -W_{c} & W_{r}\end{bmatrix},$$
where $W_{r}, W_{c}$  are real linear operators on $L^{2}_{\R}(m)$ satisfying 
$$W_{r}^{\dagger} = W_{r}, \; W_{c}^{\dagger} = W_{c}, \; W_{r} W_{c} = W_{c} W_{r}, \;  W_{r}^{2} + W_{c}^{2} = I.$$
To reconnect with some of our previous examples, consider the class of operators
$$ \begin{bmatrix} W_{r} & 0\\ 0 & W_{r}\end{bmatrix}  \begin{bmatrix} \phantom{-}M_{x} &  M_{y}\\ -M_{y} & M_{x}\end{bmatrix}  \begin{bmatrix} W_{r} &0\\ 0& W_{r}\end{bmatrix},$$
where $(W_{r} f)(\xi) = f(\alpha(\xi))$ on $L^2_{\R}(m)$ and $\alpha$ is a measure preserving transformation on $\T$ for which $\alpha \circ \alpha = \operatorname{id}$. Examples include $\alpha(\xi) = \overline{\xi}$, $\alpha(\xi) = - \xi$, or any reflection of $\T$ through a line passing through the origin. Then $W_{r}^{\dagger} = W_{r}$ and $W_{r}^{2} = I$.  In this case, the above matrix multiplication turns out to be the representation of the unitary multiplication operator
 $f \mapsto  \bar \alpha f$ from Example \ref{66yTTxx)99}. 

 In a similar way, one has the class of examples
$$ \begin{bmatrix} 0 & -W_{c}\\ W_{c} & 0\end{bmatrix}  \begin{bmatrix} \phantom{-}M_{x} &  M_{y}\\ - M_{y} & M_{x}\end{bmatrix}  \begin{bmatrix} 0 & W_{c}\\ -W_{c} & 0\end{bmatrix},$$
and $W_c f = f \circ \alpha$ as above. This will yield the same unitary operator as before.
\end{Example}

\begin{Example}
Inspired by Example \ref{Ealphabetatsst},  let $\alpha: \T \to \T$ be (Lebesgue) measure preserving with $\alpha \circ \alpha = \operatorname{id}$, $\beta: \T \to \{-1, 1\}$  be (Lebesgue) measurable such that $\beta \circ \alpha = \beta$, and $s, t \in \R$ with $s^2 + t^2 = 1$. Setting 
$$W_{r} f := s (f \circ \alpha) \; \mbox{and} \; W_{c} f := t \beta f, \quad f \in L^2_{\R}(m),$$
one can verify that $W_{r}^{\dagger} = W_{r}, W_{c}^{\dagger} = W_c, W_{r} W_{c} = W_{c} W_{r}$, and $W_{r}^{2} + W_{c}^{2} = I$. Moreover, 
$$ \begin{bmatrix} W_{r} & -W_{c}\\ W_{c} & \phantom{-}W_{r}\end{bmatrix}  \begin{bmatrix} \phantom{-}M_{x} &  M_{y}\\ - M_{y} & M_{x}\end{bmatrix}  \begin{bmatrix} \phantom{-}W_{r} & W_{c}\\ -W_{c} & W_{r}\end{bmatrix},$$
 turns out to be the representation of 
the unitary operator 
$$ f \mapsto (s^2 M_{\bar \alpha} + t^2 M^{*}) f + i s t \beta (M_{\bar \alpha} - M^{*}) (f \circ \alpha)$$
from 
 Example \ref{Ealphabetatsst}.
\end{Example}

\begin{Example}\label{examdepleFFTTT}
For the Fourier--Plancherel transform $\mathcal{F}$ on $L^2(\R)$, where, from Example \ref{FTTagain} the adjoint $\mathcal{F}^{*}$ satisfies 
$$\widehat{\mathcal{F^{*}}} \begin{bmatrix}  \sum_{n \geq 0} \alpha_n H_{n}\\ \sum_{n \geq 0} \beta_{n} H_n \end{bmatrix} = \begin{bmatrix} \sum_{n \geq 0} (\alpha_{2n} (-1)^n H_{2n}  + \beta_{2 n + 1} (-1)^n H_{2n + 1})\\
\phantom{-} \sum_{n \geq 0} (-\alpha_{2n + 1} (-1)^n H_{2n + 1}  + \beta_{2 n} (-1)^n H_{2n})\end{bmatrix},$$
 we can set $(W_{r} f)(x) = f(-x)$. Then, using the fact that the Hermite functions $H_{n}$ are odd functions when $n$ is odd and even functions when $n$ is even, we see that   
 $$(W_{r} H_{n})(x) = (-1)^{n} H_{n}(x),$$
and a calculation reveals that 
 $$ \begin{bmatrix} W_{r} & 0\\ 0 & W_{r}\end{bmatrix}  \widehat{\mathcal{F^{*}}}  \begin{bmatrix} W_{r} & 0\\ 0 & W_{r}\end{bmatrix}  = \widehat{\mathcal{F^{*}}}.$$
In other words, $\mathcal{F}^{*} \in \mathfrak{O}_c(\mathcal{F})$, which we know  already from Proposition \ref{SdfsdfdsfcvvvvVV}.
 Setting $(W_{c} f)(x) =- f(-x)$, a similar calculation yields 
  $$ \begin{bmatrix} 0 & -W_{c}\\ W_{c} & 0\end{bmatrix}  \widehat{\mathcal{F^{*}}}  \begin{bmatrix} 0 & W_{c}\\ -W_c & 0 \end{bmatrix}  = \widehat{\mathcal{F}}$$ and so $\mathcal{F} \in \mathfrak{O}_c(\mathcal{F})$. This can be confirmed by the fact that $\mathcal{F}^{*}$ is unitarily equivalent  to $\mathcal{F}$ (see the discussion in Example \ref{FTFTFT}  along with  Proposition \ref{e0roiekfgdFF}).
\end{Example}

\begin{Example}\label{HTwlkdfj77YY}
For the Hilbert transform $\mathscr{H}$ on $L^2(\R)$, where from Example \ref{e0roiekfgdFF} the adjoint $\mathscr{H}^{*}$ satisfies 
\begin{align*}
& \widehat{\mathscr{H}^{*}} \begin{bmatrix} \sum_{n < 0} \alpha_{n} f_{n} + \sum_{n \geq 0} \alpha_n f_n \\ \sum_{n < 0} \beta_n f_n + \sum_{n \geq 0} \beta_n f_n \end{bmatrix}
 = \begin{bmatrix} \phantom{-} \sum_{n < 0} \beta_n f_n - \sum_{n \geq 0} \beta_n f_n \\ -\sum_{n < 0} \alpha_n f_n + \sum_{n \geq 0} \alpha_n f_n\end{bmatrix},
\end{align*}
we can set $(W_{c} f)(x) = f(x)$ to obtain 
$$\begin{bmatrix} 0 & -W_{c}\\W_{c} & 0\end{bmatrix} \widehat{\mathscr{H}^{*}} \begin{bmatrix} 0 & W_{c}\\ -W_{c} & 0\end{bmatrix} = - \widehat{\mathscr{H}}.$$
Using the well-known fact that $\mathscr{H}^{*} = - \mathscr{H}$ \cite[p.~278]{MR4545809}, we see that $\mathscr{H}^{*}$  belongs to $\mathfrak{O}_c(\mathscr{H})$ (which we know). When $(W_{c} f)(x) = f(-x)$, we use the fact that $W_{c} f_{n} = - f_{-n - 1}$, $n \in \Z$, to see that 
$$\begin{bmatrix} 0 & -W_{c}\\W_{c} & 0\end{bmatrix} \widehat{\mathscr{H}^{*}} \begin{bmatrix} 0 & W_{c}\\ -W_{c} & 0\end{bmatrix} =  \widehat{\mathscr{H}},$$
which shows that $\mathscr{H}$ belongs to $\mathfrak{O}_c(\mathscr{H})$. 
\end{Example}



\chapter{Generalizations for Future Study}

For a fixed unitary operator $U$ on a complex Hilbert space  $\mathcal{H}$, this paper focused on the collection of unitary operators  $\mathfrak{O}_c(U) = \{CUC: \mbox{$C$ is a conjugation}\}$. One might also consider a study of $\mathfrak{O}_c(N)$, analogously defined, where $N$ is a fixed normal operator on $\mathcal{H}$, i.e., $N^{*} N = N N^{*}$. An analysis  shows that $\mathfrak{O}_c(N)$ is contained in the class of normal operators and some of the  discussion in this paper carries over to this new class. An even broader investigation might consider $\mathfrak{O}_c(A)$, where $A$ is {\em any} bounded linear operator on $\mathcal{H}$. 

Still another line of exploration could be $\mathfrak{O}_{au}(A)$, the set $X A X^{\star}$, where $X$ is antilinear and unitary $X^{-1} = X^{\star}$  ($X^{\star}$ is the antilinear conjugate of $X$ from \eqref{oooOOkkkKKKK}). Notice we are not assuming $X^2 = I$. Some call these $X$ {\em antilinear unitary operators} of which conjugations form a particular subclass. Perhaps call $\mathfrak{O}_{au}(A)$ the {\em antilinear orbit} of $A$. When $U$ is unitary, how does $\mathfrak{O}_{au}(U)$ differ from the conjugate orbit $\mathfrak{O}_c(U)$ considered in this paper? We invite the reader to join the discussion. 

\bibliographystyle{plain}

\bibliography{references}

\end{document}